\documentclass[12pt]{amsart}
\usepackage{fullpage}
\usepackage{amsmath}
\usepackage{amssymb}
\usepackage{verbatim}
\usepackage[all]{xy}
\usepackage{color}
\usepackage{graphicx}
\def\p{\partial}
\newtheorem{Theorem}{Theorem}[section]
\newtheorem{Definition}{Definition}[section]

\newtheorem{Lemma}[Theorem]{Lemma}
\newtheorem{Proposition}[Theorem]{Proposition}

\newtheorem{Corollary}[Theorem]{Corollary}

\newtheorem{Remark}{Remark}[section]
\numberwithin{equation}{section}

\newcommand{\Sp}{S^{p}}
\newcommand{\Sq}{S^{q}}

\newcommand{\Dq}{D^{q+1}}
\newcommand{\m}{\mathfrak{m}}
\newcommand{\Mor}{{\mathsf M}{\mathsf o}{\mathsf r}{\mathsf s}{\mathsf e}}
\newcommand{\GMor}{{\mathsf G}{\mathsf M}{\mathsf o}{\mathsf r}{\mathsf s}{\mathsf e}}

\newcommand{\gtor}{\bar{g}_{tor}}
\def\Ver{{\mathcal{V}}\mathsf {e}\mathsf {r}\mathsf {t}}

\def\F{{\mathcal F}}

\def\H{{\mathcal H}}

\def\M{{\mathcal M}}

\def\Riem{{\mathcal R}{\mathrm i}{\mathrm e}{\mathrm m}}

\def\Ker{\mathrm{K}\mathrm{e}\mathrm{r}}
\def\grad{{\mathrm g}{\mathrm r}{\mathrm a}{\mathrm d}}

\def\Diff{{\mathrm D}{\mathrm i}{\mathrm f}{\mathrm f}}

\begin{document}
\author{Mark Walsh}
\title{Metrics of positive scalar curvature and generalised Morse functions, part II}
\begin{abstract}
The surgery technique of Gromov and Lawson may be used to construct families of positive scalar curvature metrics which are parameterised by Morse functions. This has played an important role in the study of the space of metrics of positive scalar curvature on a smooth manifold and its corresponding moduli spaces. In this paper, we extend this technique to work for families of generalised Morse functions, i.e. smooth functions with both Morse and birth-death singularities.
\end{abstract}

\maketitle

\vspace{-0.7cm} 
\tableofcontents 
\newpage

\section{Introduction}\label{intro}

This is the second part of a larger project aimed at better understanding the topology of the  space of metrics of positive scalar curvature ({\em psc-metrics}) on a smooth manifold $X$. A great deal is known about the problem of whether or not $X$ admits a psc-metric; see \cite{RS} for a survey of this problem. Much less is known about the topology of the space of all psc-metrics on $X$, denoted $\Riem^{+}(X)$, or its corresponding moduli space $\M^{+}(X)$ which is obtained as a  quotient by the action of the diffeomorphism group $\Diff(X)$.

Obtaining topological data about $\Riem^{+}{(X)}$ or $\M^{+}(X)$ is not easy. It is well known that these spaces may not be connected; see \cite{BG}, \cite{Carr} and \cite{Hit}. It has been known for some time also, that the space $\Riem^{+}{(X)}$ may have non-trivial fundamental group; see \cite{Hit}. However, the question of whether or not $\pi_{k}(\Riem^{+}(X))$ is non-trivial when $k\geq 2$, is still open. Moreover, the fact that $\pi_{k}(\M^{+}(X))$ may be non-trivial when $k\geq 1$, was only recently demonstrated in \cite{BHSW}. As yet, little is known about the algebraic structure of these groups.

In this paper, we build on and strengthen the techniques developed in Part One \cite{Walsh1} and \cite{BHSW}. Before stating our results, it is worth saying a few words about the earlier papers.

\subsection{Earlier Work}
A major goal of this project is the construction of interesting families of psc-metrics, in particular families which represent non-trivial elements in the higher homotopy groups of $\Riem^{+}(X)$ or $\M^{+}(X)$. The surgery technique of Gromov and Lawson (which we discuss in detail below) has long been the most fruitful method of building examples of psc-metrics. For instance, it is the only known method of constructing metrics which lie in distinct path components of $\M^{+}(X)$. In Part One, we perform a detailed study of this technique in the case of a smooth compact cobordism of manifolds $\{W; X_0, X_1\}$. Recall this means that $\p W=X_0 \sqcup X_1$, where $X_0$ and $X_1$ are closed manifolds. Starting with a psc-metric $g_0$ on $X_0$ and an {\em admissible} Morse function $f:W\rightarrow I$, a new metric is obtained by extending $g_0$ over $W$ using a modified form of the Gromov-Lawson construction near critical points. Admissibility means that critical points have indices which correspond to surgeries in codimension $\geq$ $3$, a necessary condition for Gromov-Lawson surgery. The resulting metric, denoted $\bar{g}=\bar{g}(g_0,f)$ has positive scalar curvature and carries a product structure near the boundary. It is called a {\em Gromov-Lawson cobordism} and is schematically described in Fig. \ref{glschem}. In the next section, we describe this construction in more detail. The main result of Part One is that, in the case when $W$ is a simply connected cylinder $X\times I$ with $\dim X\geq 5$, the metrics $g_0$ and $g_1=\bar{g}|_{X\times\{1\}}$ are isotopic, i.e. connected by a path through psc-metrics.

\begin{figure}[!htbp]
\begin{picture}(0,0)%
\includegraphics{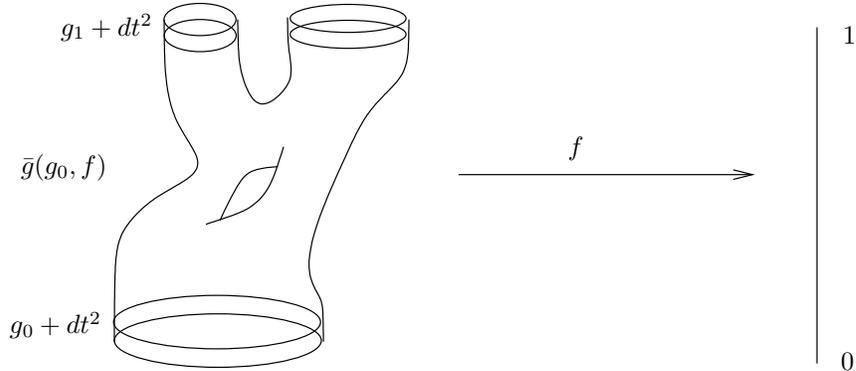}%
\end{picture}%
\setlength{\unitlength}{3947sp}%
\begingroup\makeatletter\ifx\SetFigFont\undefined%
\gdef\SetFigFont#1#2#3#4#5{%
  \reset@font\fontsize{#1}{#2pt}%
  \fontfamily{#3}\fontseries{#4}\fontshape{#5}%
  \selectfont}%
\fi\endgroup%
\begin{picture}(5717,2375)(1174,-3815)
\put(1714,-2524){\makebox(0,0)[lb]{\smash{{\SetFigFont{10}{12}{\rmdefault}{\mddefault}{\updefault}{\color[rgb]{0,0,0}$\bar{g}(g_0, f)$}%
}}}}
\put(1964,-1649){\makebox(0,0)[lb]{\smash{{\SetFigFont{10}{12}{\rmdefault}{\mddefault}{\updefault}{\color[rgb]{0,0,0}{$g_1+dt^{2}$}}%
}}}}
\put(1644,-3524){\makebox(0,0)[lb]{\smash{{\SetFigFont{10}{12}{\rmdefault}{\mddefault}{\updefault}{\color[rgb]{0,0,0}$g_0+dt^{2}$}%
}}}}
\put(5151,-2399){\makebox(0,0)[lb]{\smash{{\SetFigFont{10}{12}{\rmdefault}{\mddefault}{\updefault}{\color[rgb]{0,0,0}$f$}%
}}}}
\put(6876,-1699){\makebox(0,0)[lb]{\smash{{\SetFigFont{10}{12}{\rmdefault}{\mddefault}{\updefault}{\color[rgb]{0,0,0}$1$}%
}}}}
\put(6864,-3749){\makebox(0,0)[lb]{\smash{{\SetFigFont{10}{12}{\rmdefault}{\mddefault}{\updefault}{\color[rgb]{0,0,0}$0$}%
}}}}
\end{picture}%
\caption{The Gromov-Lawson cobordism $\bar{g}=\bar{g}(g_0, f)$ on $W$.}
\label{glschem}
\end{figure}
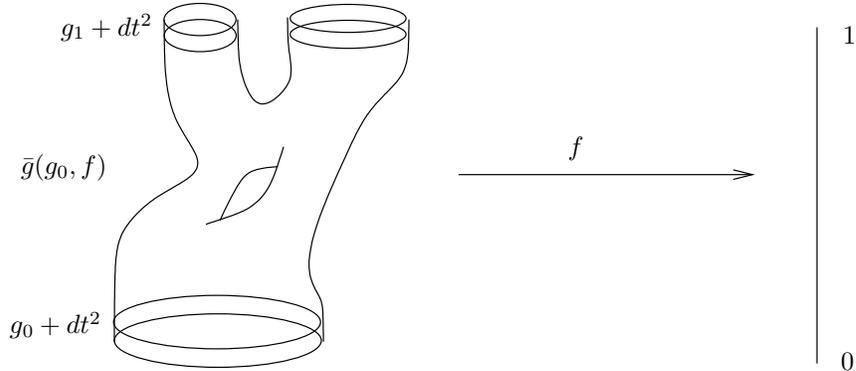

In \cite{BHSW}, the authors perform a family version of this construction.  This is done with respect to a fibrewise admissible Morse function on a smooth bundle with fibre: the cobordism $W$. Applying this construction to certain non-trivial sphere bundles defined by Hatcher (see \cite{Goette}), allows for the exhibition of non-trivial elements in the higher homotopy groups of $\M_{x}^{+}(X)$, the {\em observer moduli space of psc-metrics} on $X$. This space is obtained as a quotient of $\Riem^{+}(X)$ by the action of the subgroup $\Diff_{x}(X)\subset \Diff(X)$, of diffeomorphisms which fix a base point $x\in X$ and whose derivative maps are identity at $T_{x}X$. The authors go on to show that this implies the existence of non-trivial elements in the higher homotopy groups of the regular moduli space $\M^{+}(X)$.

\subsection{This Paper} 
The main results of this paper involve extending the techniques of Part One and \cite{BHSW} to families of generalised Morse functions. Roughly speaking, a generalised Morse function has both Morse and {\em birth-death singularities}, which allow for cancellation of certain pairs of Morse singularities. Thus, the space of Morse functions on a smooth manifold embeds naturally into the space of generalised Morse functions. Moreover, this allows us to connect up disjoint path components in the space of Morse functions. One reason for including this case, is that the topology of the space of generalised Morse functions on a smooth manifold is both non-trivial and well understood; see \cite{I1}, \cite{I2}, \cite{I3} and \cite{CO}. By extending the above constructions to generalised Morse functions, we hope to utilise this topological knowledge to better understand the space of psc-metrics. At the very least, we hope to exhibit further examples of non-trivial elements in the homotopy groups of $\Riem^{+}(X)$ and  $\M^{+}(X)$.

We begin by showing that it is possible to perform the original construction from Part One continuously over a path through generalised Morse functions, connecting a Morse function with two singularities to a Morse function with none. This is Theorem \ref{genmorsepath}, and is the geometric heart of the paper. An important implication is that, under reasonable conditions, the isotopy type of a Gromov-Lawson cobordism does not depend on the choice of Morse function and so is an invariant of the cobordism. This is Theorem \ref{isothm}, the proof of which depends heavily on results by Hatcher on the connectivity of certain subspaces of the space of generalised Morse functions; see \cite{I3}. Finally, in Theorem \ref{genmain} we show that the family construction in \cite{BHSW} goes through for fibrewise families of generalised Morse functions. 

 
\subsection{Background}
Let $X$ be a smooth closed manifold of dimension $n$. We denote by $\Riem(X)$,  the space of all Riemannian metrics on $X$ under its standard $C^{\infty}$-topology. The space of psc-metrics on $X$, $\Riem^{+}(X)$, is thus an open subspace of $\Riem(X)$. Although most of our work will involve the space $\Riem^{+}(X)$, it is worth recalling the definition of the moduli space of psc-metrics on $X$, $\M^{+} (X)$. Recall that the group $\Diff(X)$, of diffeomorphisms on $X$, acts on $\Riem(X)$ by pull-back as follows.
\begin{equation*}
\begin{split}
\Diff(X)\times \Riem(X)&\longrightarrow \Riem(X),\\
(\phi, g)&\longmapsto \phi^{*}g.
\end{split}
\end{equation*}
The moduli space $\M(X)$ is obtained as a quotient by this action on $\Riem(X)$. Finally, restricting the action to the subspace $\Riem^{+}(X)$, yields the moduli space of psc-metrics $\M^{+}(X)\subset\M(X)$. 

We now recall the notions of isotopy and concordance which play an important role in any analysis of these spaces. Metrics which lie in the same path component of $\Riem^{+}(X)$ are said to be {\em isotopic}. Two psc-metrics: $g_0$ and $g_1$ on $X$, are said to be {\em concordant} if there is a psc-metric $\bar{g}$ on the cylinder $X\times I$ ($I=[0,1]$), so that $\bar{g}=g_0+dt^{2}$ near $X\times \{0\}$ and $\bar{g}=g_1+dt^{2}$ near $X\times \{1\}$. It is clear that these notions are equivalence relations on $\Riem^{+}(X)$. It is also well known that isotopic psc-metrics are necessarily concordant; see Lemma 2.1 in \cite{Walsh1}. Whether or not the converse is true is a difficult open question (at least when $n\geq 5$), and one we devote quite a lot of time to in Part One. Indeed, the main result of Part One, Theorem 1.5 of \cite{Walsh1}, gives an affirmative answer to this question in the case of concordances constructed using the Gromov-Lawson technique, when $X$ is simply connected and $n\geq 5$.

One important reason for seeking an answer to the question of whether concordant metrics are isotopic arises when studying the path-connectivity of $\Riem^{+}(X)$. It is known that $\Riem^{+}(X)$ need not be path-connected. However, the only known method for showing that two psc-metrics on $X$ lie in distinct path components of $\Riem^{+}(X)$, is to show that these metrics are not concordant. For example, Carr's proof in \cite{Carr} that when $k\geq 2$, $\Riem^{+}(S^{4k-1})$ has an infinite number of path components, involves using index obstruction methods to exhibit a countably infinite collection of distinct concordance classes on $S^{4k-1}$. This implies that the space $\Riem^{+}(S^{4k-1})$ has at least as many path components. Note that this result also holds for the moduli space $\M^{+}(S^{4k-1})$, as  $\pi_{0}(\Diff(S^{n}))$ is finite.

As discussed earlier, little is known about the higher homotopy groups of $\Riem^{+}{(X)}$ or $\M^{+}(X)$. It is known, for example, that $\Riem^{+}(S^{2})$ is contractible (as is
$\Riem^{+}(\mathbb{R}P^{2})$), see \cite{RS}. Interestingly, Hitchin showed in \cite{Hit}, that in the spin case, $\pi_1(\Riem^{+}(X))\neq 0$ when $n\equiv -1,0$ (mod $8$) (all of these elements are mapped to zero in the moduli space). However, there are no known examples of non-trivial elements in $\pi_{k}(\Riem^{+}(X))$ when $k>1$. On the other hand, the existence of non-trivial elements in the higher homotopy groups of the moduli space $\M^{+}(X)$ was shown by the authors in \cite{BHSW}. 

We now come to the role of surgery.
Suppose $X$ is a manifold which admits a metric of positive scalar
curvature. The Surgery Theorem of Gromov-Lawson \cite{GL1} (proved independently by Schoen-Yau \cite{SY}) gives a method for constructing further
metrics of positive scalar curvature on any manifold $X'$ which is
obtained from $X$ by surgery in codimension $\geq 3$.
Under reasonable restrictions, this includes every manifold which is
cobordant to $X$. The surgery technique is therefore a powerful device in the construction of new psc-metrics. Indeed, all of the above methods of constructing distinct concordance classes of psc-metrics use some version of this technique.
As discussed above, we utilise the Gromov-Lawson technique to construct a psc-metric $\bar{g}$ on a smooth compact cobordism $\{W; X_0, X_1\}$. We will discuss this in more detail in the next section. For now, recall that the metric $\bar{g}=\bar{g}(g_0, f)$ is determined (up to some minor parameter choices) by a psc-metric $g_0$ on $X_0$ and an admissible Morse function $f:W\rightarrow I$. Henceforth, when we use the terms Gromov-Lawson construction or Gromov-Lawson cobordism, this is what we are referring to.

\subsection{A Family Version of the Gromov-Lawson Construction}
Extending the Gromov-Lawson construction to work for compact families of psc-metrics is straightforward and follows from the fact that positive scalar curvature is an open condition. The proof is a matter of carefully checking each step in the construction and is done in Part One. Generalising the construction to work for compact families of admissible Morse functions is a more delicate matter. The first problem is to define what we mean by a ``family". We employ the notion described by Igusa in chapter four of \cite{FRT} where a family is a certain bundle of fibrewise Morse functions. We will review this in more detail in the next section but, for now, a {\em family of Morse functions} can be thought of in the following way. Let $\pi:E^{n+k+1} \rightarrow B^{k}$ be a smooth fibre bundle with fibre $W^{n+1}$, the smooth cobordism described above. The spaces $E$ and $B$ are smooth compact manifolds. Let $F$ be a map $E \rightarrow B\times I$, satisfying $p_1\circ F=\pi$ (where $p_1$ is projection on the first factor) and whose restriction on the fibre $W_y=\pi^{-1}(y)$ is an admissible Morse function $W_y\rightarrow \{y\}\times I$. Later we will require some other technical conditions on $F$, but we will ignore these for now. Schematically, this is represented in Fig. \ref{fibrewisemorseone}.

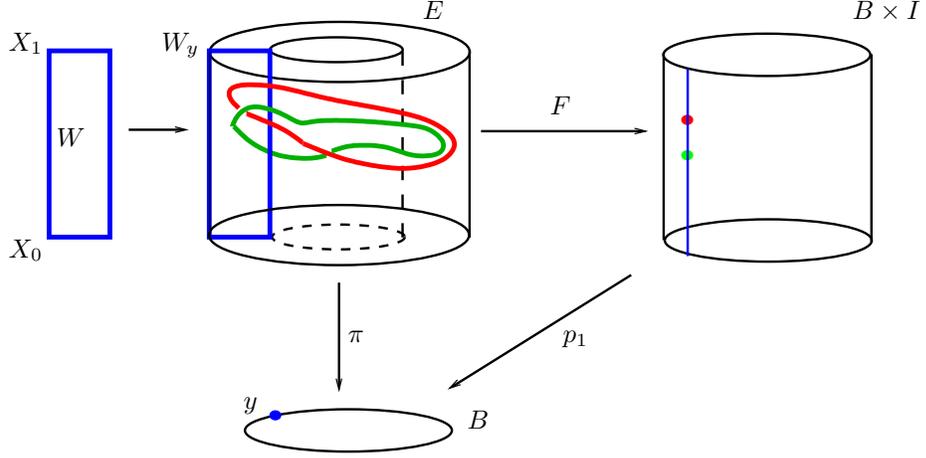
\begin{figure}[htbp]
\vspace{-2cm}
\begin{picture}(0,0)%
\includegraphics{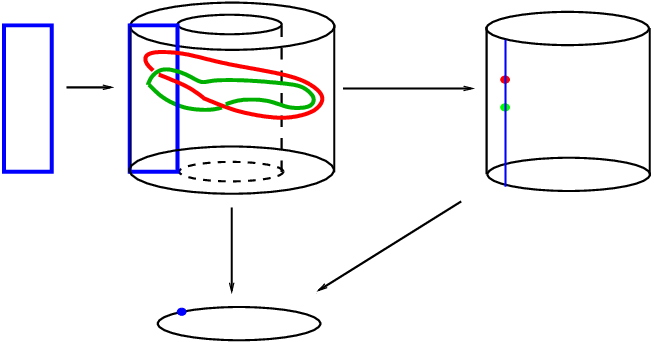}%
\end{picture}%
\setlength{\unitlength}{3947sp}%
\begingroup\makeatletter\ifx\SetFigFont\undefined%
\gdef\SetFigFont#1#2#3#4#5{%
  \reset@font\fontsize{#1}{#2pt}%
  \fontfamily{#3}\fontseries{#4}\fontshape{#5}%
  \selectfont}%
\fi\endgroup%
\begin{picture}(4536,4099)(1496,-5186)
\put(1576,-3249){\makebox(0,0)[lb]{\smash{{\SetFigFont{10}{12}{\rmdefault}{\mddefault}{\updefault}{\color[rgb]{0,0,0}{$W$}}%
}}}}
\put(1276,-2649){\makebox(0,0)[lb]{\smash{{\SetFigFont{10}{12}{\rmdefault}{\mddefault}{\updefault}{\color[rgb]{0,0,0}{$X_1$}}%
}}}}
\put(1276,-3949){\makebox(0,0)[lb]{\smash{{\SetFigFont{10}{12}{\rmdefault}{\mddefault}{\updefault}{\color[rgb]{0,0,0}{$X_0$}}%
}}}}

\put(4676,-3049){\makebox(0,0)[lb]{\smash{{\SetFigFont{10}{12}{\rmdefault}{\mddefault}{\updefault}{\color[rgb]{0,0,0}{$F$}}%
}}}}
\put(6576,-2449){\makebox(0,0)[lb]{\smash{{\SetFigFont{10}{12}{\rmdefault}{\mddefault}{\updefault}{\color[rgb]{0,0,0}{$B\times I$}}%
}}}}

\put(2236,-2649){\makebox(0,0)[lb]{\smash{{\SetFigFont{10}{12}{\rmdefault}{\mddefault}{\updefault}{\color[rgb]{0,0,0}{$W_y$}}%
}}}}

\put(3876,-2449){\makebox(0,0)[lb]{\smash{{\SetFigFont{10}{12}{\rmdefault}{\mddefault}{\updefault}{\color[rgb]{0,0,0}{$E$}}%
}}}}
\put(2751,-4899){\makebox(0,0)[lb]{\smash{{\SetFigFont{10}{12}{\rmdefault}{\mddefault}{\updefault}{\color[rgb]{0,0,0}$y$}%
}}}}
\put(3406,-4486){\makebox(0,0)[lb]{\smash{{\SetFigFont{10}{12}{\rmdefault}{\mddefault}{\updefault}{\color[rgb]{0,0,0}$\pi$}%
}}}}
\put(4156,-5026){\makebox(0,0)[lb]{\smash{{\SetFigFont{10}{12}{\rmdefault}{\mddefault}{\updefault}{\color[rgb]{0,0,0}$B$}%
}}}}
\put(4756,-4486){\makebox(0,0)[lb]{\smash{{\SetFigFont{10}{12}{\rmdefault}{\mddefault}{\updefault}{\color[rgb]{0,0,0}$p_1$}%
}}}}
\end{picture}%
\caption{A family of admissible Morse functions, each of which has two critical points.}
\label{fibrewisemorseone}
\end{figure}


This notion is utilised in a family version of the Gromov-Lawson construction, performed in Theorem 2.12 of \cite{BHSW}. Later, we will revisit this theorem as a prelude to strengthening it, and so we defer explanation of some of the more technical terms until then. 

\begin{Theorem}\label{main} {\rm [2.12 of \cite{BHSW}]}
Let $\pi: E\to B$ be a bundle of smooth compact manifolds, where the fibre
  $W$ is a compact manifold with boundary $ \p W= X_0\sqcup X_1$, and the
  structure group is $\Diff(W;X_0,X_1)$.
  Let $F: E\to B\times I$ be an admissible
  family of Morse functions, with respect to $\pi$. In addition, we
  assume that the fibre bundle $\pi : E\to B$ is given the structure
  of a Riemannian submersion $\pi: (E,{\mathfrak m}_{E})\to
  (B,{\mathfrak m}_{B})$ such that the metric ${\mathfrak m}_{E}$ is
  compatible with the map $F: E \to B\times I$.  Finally, let $g_0 : B \to \Riem^+(X_0)$ be a smooth map. 

  Then there exists a metric $\bar{g}=\bar{g}(g_0,F,{\mathfrak m}_{E})$
  such that for each $y\in B$ the restriction $\bar{g}(y)=\bar{g}|_{W_y}$ on the
  fibre $W_y=\pi^{-1}(y)$ satisfies the following properties:
\begin{enumerate}
\item[{\bf (1)}] $\bar{g}(y)$ extends $g_0(y)$;
\item[{\bf (2)}] $\bar{g}(y)$ is a product metric $g_{i}(y)+dt^2$
  near $X_{i}\subset \p W_y$, $i=0,1$;
\item[{\bf (3)}] $\bar{g}(y)$ has positive scalar curvature
  on $W_y$.
\end{enumerate}
\end{Theorem} 

One limitation of this construction is that all admissible Morse functions in the family must have the same number of critical points of the same index; see Remark 2.7 in the appendix of \cite{I3}. However, it is possible for certain pairs of Morse critical points to cancel in the form of birth-death singularities. In order to connect up admissible Morse functions which have different critical sets, we must allow for this cancellation. This means working in the space of admissible {\em generalised} Morse functions, where a generalised Morse function has Morse and birth-death singularities. A rough description of this cancellation is given in Fig. \ref{genmorseschem} where the Morse singularities $p$ and $q$ cancel at the birth-death singularity $w$. A {\em family of generalised Morse functions} is, roughly, a map $F:E\rightarrow B\times I$ of the type described above, which restricts on each fibre $W_y$ to a generalised Morse function $W_y\rightarrow \{y\}\times I$; see Fig. \ref{genmorsefamilymap}. Furthermore, it turns out that the set of $y\in B$ which have fibre $W_y$ containing a birth-death singularity forms the image of a certain codimension one manifold immersion in $B$. Again, this will be dealt with in more detail later on.

\begin{figure}[htbp]
\vspace{1cm}
\begin{picture}(0,0)%
\includegraphics{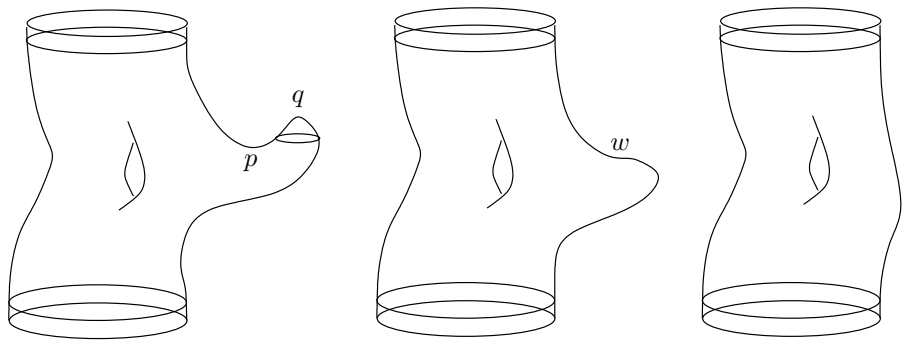}%
\end{picture}%
\setlength{\unitlength}{3947sp}%
\begingroup\makeatletter\ifx\SetFigFont\undefined%
\gdef\SetFigFont#1#2#3#4#5{%
  \reset@font\fontsize{#1}{#2pt}%
  \fontfamily{#3}\fontseries{#4}\fontshape{#5}%
  \selectfont}%
\fi\endgroup%
\begin{picture}(5717,2375)(1174,-3815)
\put(2664,-2724){\makebox(0,0)[lb]{\smash{{\SetFigFont{10}{12}{\rmdefault}{\mddefault}{\updefault}{\color[rgb]{0,0,0}$p$}%
}}}}
\put(2964,-2324){\makebox(0,0)[lb]{\smash{{\SetFigFont{10}{12}{\rmdefault}{\mddefault}{\updefault}{\color[rgb]{0,0,0}$q$}%
}}}}
\put(4964,-2624){\makebox(0,0)[lb]{\smash{{\SetFigFont{10}{12}{\rmdefault}{\mddefault}{\updefault}{\color[rgb]{0,0,0}$w$}%
}}}}
\end{picture}%
\caption{Cancelling a pair of critical points}
\label{genmorseschem}
\end{figure}
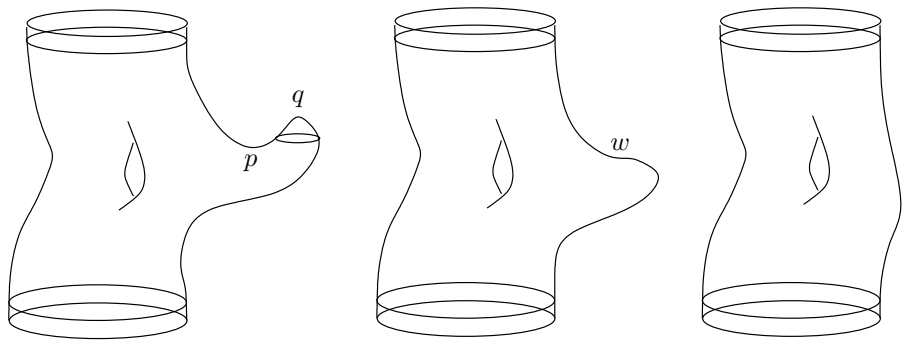


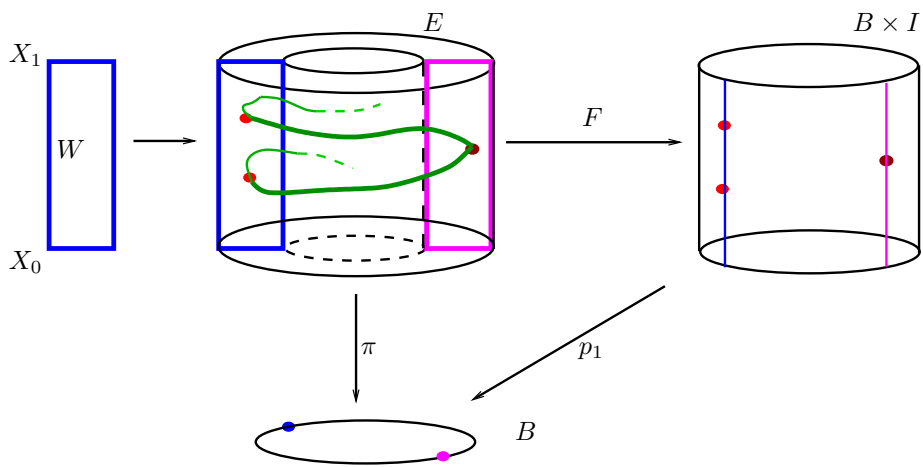
\begin{figure}[htbp]
\vspace{-2cm}
\begin{picture}(0,0)%
\includegraphics{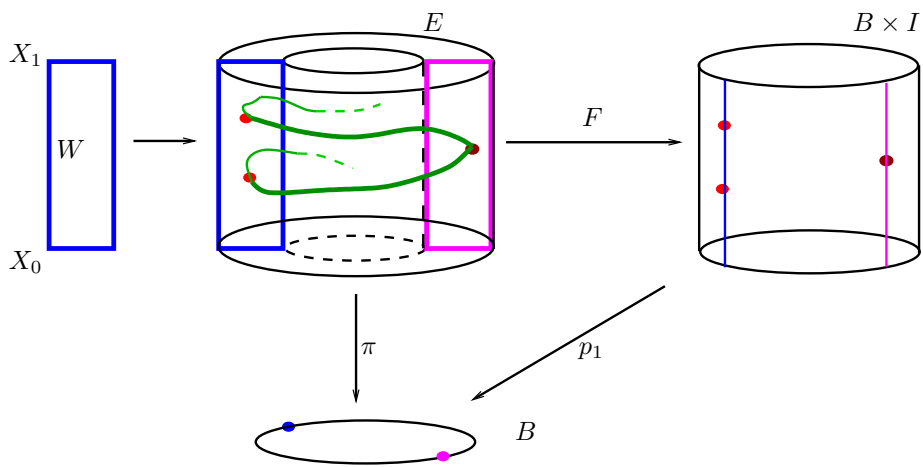}%
\end{picture}%
\setlength{\unitlength}{3947sp}%
\begingroup\makeatletter\ifx\SetFigFont\undefined%
\gdef\SetFigFont#1#2#3#4#5{%
  \reset@font\fontsize{#1}{#2pt}%
  \fontfamily{#3}\fontseries{#4}\fontshape{#5}%
  \selectfont}%
\fi\endgroup%
\begin{picture}(4536,4099)(1496,-5186)
\put(1576,-3249){\makebox(0,0)[lb]{\smash{{\SetFigFont{10}{12}{\rmdefault}{\mddefault}{\updefault}{\color[rgb]{0,0,0}{$W$}}%
}}}}
\put(1276,-2649){\makebox(0,0)[lb]{\smash{{\SetFigFont{10}{12}{\rmdefault}{\mddefault}{\updefault}{\color[rgb]{0,0,0}{$X_1$}}%
}}}}
\put(1276,-3949){\makebox(0,0)[lb]{\smash{{\SetFigFont{10}{12}{\rmdefault}{\mddefault}{\updefault}{\color[rgb]{0,0,0}{$X_0$}}%
}}}}

\put(4876,-3049){\makebox(0,0)[lb]{\smash{{\SetFigFont{10}{12}{\rmdefault}{\mddefault}{\updefault}{\color[rgb]{0,0,0}{$F$}}%
}}}}
\put(6576,-2449){\makebox(0,0)[lb]{\smash{{\SetFigFont{10}{12}{\rmdefault}{\mddefault}{\updefault}{\color[rgb]{0,0,0}{$B\times I$}}%
}}}}
\put(3876,-2449){\makebox(0,0)[lb]{\smash{{\SetFigFont{10}{12}{\rmdefault}{\mddefault}{\updefault}{\color[rgb]{0,0,0}{$E$}}%
}}}}

\put(3486,-4486){\makebox(0,0)[lb]{\smash{{\SetFigFont{10}{12}{\rmdefault}{\mddefault}{\updefault}{\color[rgb]{0,0,0}$\pi$}%
}}}}
\put(4456,-5026){\makebox(0,0)[lb]{\smash{{\SetFigFont{10}{12}{\rmdefault}{\mddefault}{\updefault}{\color[rgb]{0,0,0}$B$}%
}}}}
\put(4856,-4486){\makebox(0,0)[lb]{\smash{{\SetFigFont{10}{12}{\rmdefault}{\mddefault}{\updefault}{\color[rgb]{0,0,0}$p_1$}%
}}}}
\end{picture}%
\caption{A family of admissible generalised Morse functions}
\label{genmorsefamilymap}
\end{figure}

\subsection{Main Results.} 
To have any hope of extending Theorem \ref{main} to the case of generalised Morse functions, we must first extend the original Gromov-Lawson cobordism construction over a birth-death singularity. This is the subject of Theorem \ref{genmorsepath} below. The proof of this theorem requires a substantial strengthening of the main result from \cite{Walsh1}, involving the construction of an isotopy through Gromov-Lawson cobordisms, over a cancellation of Morse singularities. This construction is done in Theorem \ref{relconcisodoublesurgery} and is the geometric basis for all of our main results. 

\begin{Theorem}\label{genmorsepath}
Let $\{W; X_0, X_1\}$ be a smooth compact cobordism and let $F:W\times I\rightarrow I\times I$ be a moderate family of admissible generalised Morse functions. Suppose there is a point $y_0\in (0,1)$, so that $f_y=F|_{W\times \{y\}}$ is a Morse function for all $y\in I\setminus \{y_0\}$, and that $f_{y_0}$ contains exactly one birth-death critical point. Finally, let $g_0:I\rightarrow \Riem^{+}(X_0)$ be a family of psc-metrics on $X_0$. Then, there is a metric $\bar{\bar{g}}=\bar{\bar{g}}(F, g_0)$ on $W\times I$ which satisfies the following conditions.
\begin{enumerate}
\item{} For each $y\in [0,1]$,  the restriction of $\bar{\bar{g}}$ on slices $W\times\{y\}$ is a psc-metric which extends $g_0(y)$ and which has a product structure near the boundary $\p W\times\{y\}$.
\item{} For $y\in[0,1]$, away from $y_0$, the restriction of $\bar{\bar{g}}$ on slices $W\times\{y\}$, is a Gromov-Lawson cobordism.
\end{enumerate}
\end{Theorem}

One question which immediately arises from the Gromov-Lawson construction in \cite{Walsh1} concerns the choice of Morse function. In particular, how does the choice of the admissible Morse function affect the resulting metric, say, up to isotopy? By combining Theorem \ref{genmorsepath} with some results of Hatcher on the connectivity of the space of admissible generalised Morse functions, we obtain the following answer to this question. 

\begin{Theorem}\label{isothm}
Let $\{W;X_0,X_1\}$ be a smooth compact cobordism, with $\pi_{1}(W)=\pi_{1}(X_0)=\pi_1(X_1)=0$ and $\dim W=n+1\geq 6$. Let $f_0$ and $f_1$ be a pair of admissible Morse functions on $W$ and let $g_0\in \Riem^{+}(X_0)$. Then the metrics $\bar{g}(g_0, f_0)$ and $\bar{g}(g_0, f_1)$ are isotopic, relative to the metric $g_0$, in $\Riem^{+}(W, \p W)$.  
\end{Theorem}

We now state the main theorem. This extends Theorem \ref{main} above to the case of families of generalised Morse functions, using Theorem \ref{genmorsepath} as a key step. 
\begin{Theorem}\label{genmain}
  Let $\pi: E\to B$ be a bundle of smooth compact manifolds, where the fibre
  $W$ is a compact manifold with boundary $ \p W= X_0\sqcup X_1$ and the
  structure group is $\Diff(W;X_0,X_1)$.
  Let $F: E\to B\times I$ be a moderate
  family of admissible generalised Morse functions, with respect to $\pi$. In addition, we
  assume that the fibre bundle $\pi : E\to B$ is given the structure
  of a Riemannian submersion $\pi: (E,{\mathfrak m}_{E})\to
  (B,{\mathfrak m}_{B})$, such that the metric ${\mathfrak m}_{E}$ is
  compatible with the map $F: E \to B\times I$, and a gradient-like vector field $V_E$. Finally, let $g_0 : B \to \Riem^+(X_0)$ be a smooth map.

  Then there exists a metric $\bar{g}=\bar{g}(g_0,F)$ (where $F=(F, \m_E, V_E)$)
  such that for each $y\in B$ the restriction $\bar{g}(y)=\bar{g}|_{W_y}$ on the
  fibre $W_y=\pi^{-1}(y)$ satisfies the following properties:
\begin{enumerate}
\item[{\bf (1)}] $\bar{g}(y)$ extends $g_0(y)$;
\item[{\bf (2)}] $\bar{g}(y)$ is a product metric $g_{i}(y)+dt^2$
  near $X_{i}\subset \p W_y$, $i=0,1$;
\item[{\bf (3)}] $\bar{g}(y)$ has positive scalar curvature
  on $W_y$.
\end{enumerate}
\end{Theorem} 

\subsection{The Space of Gromov-Lawson Cobordisms}
In the introduction to Part One, we previewed a number of theorems applying the Gromov-Lawson cobordism construction, the idea being that these theorems would form the basis of Part Two. In the end, all but one of these theorems is contained in the results of this paper. The remaining theorem, Theorem D of \cite{Walsh1}, concerns the homotopy type of the space of all Gromov-Lawson cobordisms on $W$. As mentioned earlier, a great deal is known about the homotopy type of the space of generalised Morse functions; see \cite{I1}, \cite{I2}, \cite{I3} and \cite{CO}. Thus, it reasonable to use this space to parameterise an important subspace of psc-metrics on $W$, namely those obtained by the Gromov-Lawson construction. Unfortunately, given the size of this paper and the technical demands in explicitly describing a map from a space of admissible generalised Morse functions to a space of psc-metrics, it has been decided not to include this result. Instead, it will appear as part of a third paper dealing specifically with the topology of the space of Gromov-Lawson cobordisms.

\subsection{Acknowledgements} I am deeply grateful to Boris Botvinnik, my doctoral advisor, for suggesting this problem and for his guidance and friendship during my time at the University of Oregon and beyond. Much of this work took place at WWU M{\"u}nster, Germany as well as at Oregon State University in Corvallis, USA. My sincere thanks to both institutions. In particular, I am grateful to Wolfgang L{\"u}ck and Michael Joachim at M{\"u}nster for their hospitality. My thanks also to Christine Escher at Oregon State for taking an interest in my work. Finally, my thanks to Thomas Schick at the University of G{\"o}ttingen for his incisive comments, and to David Wraith at NUI Maynooth, Ireland, for many helpful conversations.

\section{Revisiting Part One}\label{review}

This section is actually a combination of review and new material. We will provide a review of the main results and techniques from the first paper. We will also make a number of observations about these techniques, as well as prove a theorem which significantly strengthens the main result of Part One.


\subsection{Isotopy and Relative Isotopy}\label{relisotopy}
Recall that $X$ denotes a smooth closed manifold of dimension $n$ and $\Riem(X)$ denotes the space of Riemannian metrics on $X$. Suppose that $W$ is a smooth compact $(n+1)$-manifold with $\p W = X$. We denote by $\Riem(W, X)$, the space of Riemannian metrics on $W$ which take the form of a product metric near the boundary. Thus, an element $\bar{g}\in \Riem(W, X)$ takes the form $\bar{g}=g+dt^{2}$, with $g\in \Riem(X)$, on some collar neighbourhood of the boundary $\p W=X$. 
We wish to generalise this notion to manifolds with corners. 
Here, an {\em $m$-dimensional smooth manifold with corners} has a smooth atlas consisting of charts of the form $(U,\phi)$ where $U$ is an open subset of $[0,\infty)^{k}\times \mathbb{R}^{m-k}$, $0\leq k\leq m$ and $\phi:U\rightarrow Y$ is a homeomorphism onto its image $\phi(Y)$, a non-empty open set. 

We will only be interested in the specific case where $k=2$. Let $Y$ be an $(n+2)$-dimensional smooth compact manifold with corners, the boundary of which decomposes as $\p Y= W_0\cup W_1$. Here $W_0$ and $W_1$ are smooth compact $(n+1)$-dimensional manifolds with a common closed boundary $\p^2 Y=\p W_0=\p W_1 = X$, as shown in Fig. \ref{corners}. Near the boundary components $W_0$ and $W_1$, we can specify ``collar" neighbourhoods which are diffeomorphic to $W_0\times I$ and $W_1\times I$ respectively. Furthermore, near $X$ these neighbourhoods intersect to determine a region which is diffeomorphic to $X\times I \times I$. Now let $\Riem(Y, \p Y, \p^2 Y)$ be the space of Riemannian metrics which satisfy the condition that for each element $\bar{\bar{g}}\in \Riem(Y, \p Y, \p^2 Y)$:
\begin{enumerate}
\item[{\bf (i)}] $\bar{\bar{g}}=\bar{g}_0+ds^2$ near $W_0$ and  $\bar{\bar{g}}=\bar{g}_1+dt^2$ near $W_1$, where $\bar{g}_0\in \Riem(W_0, X)$ and $\bar{g}_1\in \Riem(W_1, X)$,
\item[{\bf (ii)}] $\bar{\bar{g}} = g+ds^2+dt^2$ near $X$ where $\bar{g}_0|_{X}=\bar{g}_1|_{X}=g\in\Riem(X)$.
\end{enumerate}
\noindent We will often refer to such metrics as {\em metrics with corners}.
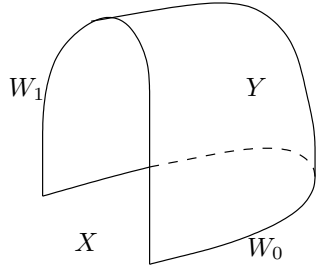
\begin{figure}[!htbp]
\hspace{5.5cm}
\begin{picture}(0,0)%
\includegraphics{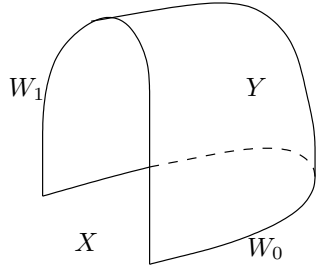}%
\end{picture}%
\setlength{\unitlength}{3947sp}%
\begingroup\makeatletter\ifx\SetFigFont\undefined%
\gdef\SetFigFont#1#2#3#4#5{%
  \reset@font\fontsize{#1}{#2pt}%
  \fontfamily{#3}\fontseries{#4}\fontshape{#5}%
  \selectfont}%
\fi\endgroup%
\begin{picture}(5079,1559)(1902,-7227)
\put(2114,-7136){\makebox(0,0)[lb]{\smash{{\SetFigFont{10}{8}{\rmdefault}{\mddefault}{\updefault}{\color[rgb]{0,0,0}$X$}%
}}}}
\put(3189,-6161){\makebox(0,0)[lb]{\smash{{\SetFigFont{10}{8}{\rmdefault}{\mddefault}{\updefault}{\color[rgb]{0,0,0}$Y$}%
}}}}
\put(3200,-7161){\makebox(0,0)[lb]{\smash{{\SetFigFont{10}{8}{\rmdefault}{\mddefault}{\updefault}{\color[rgb]{0,0,0}$W_0$}%
}}}}
\put(1700,-6161){\makebox(0,0)[lb]{\smash{{\SetFigFont{10}{8}{\rmdefault}{\mddefault}{\updefault}{\color[rgb]{0,0,0}$W_1$}%
}}}}
\end{picture}%
\caption{The manifold with corners $Y$, with boundary $\p Y = W_0\cup W_1$ and $\p^{2}Y=\p W_0=\p W_1 = X$}
\label{corners}
\end{figure}   

Finally, we denote by $\Riem^{+}(X)$, $\Riem^{+}(W, \p W)$ and $\Riem^{+}(Y, \p Y, \p^2 Y)$ the respective subspaces of $\Riem(X)$, $\Riem(W, \p W)$ and $\Riem(Y, \p Y, \p^2 Y)$ which consist of metrics with positive scalar curvature. Recall that an {\em isotopy} of psc-metrics on a smooth closed manifold $X$ is a path in the space $\Riem^{+}(X)$. We generalise this definition to metrics in $\Riem^{+}(W, \p W)$ and $\Riem^{+}(Y, \p Y, \p^2 Y)$ in the obvious way. 

\begin{Definition}
{\rm A pair of metrics which are contained in the same path component of $\Riem^{+}(X)$,  (respectively $\Riem^{+}(W,\p W ), \Riem^{+}(Y, \p Y, \p^2 Y)$) are said to be {\em isotopic}.} 
\end{Definition}

Note that in the case of an isotopy in the spaces $\Riem^{+}(W,\p W )$ and $\Riem^{+}(Y, \p Y, \p^2 Y)$, it is not necessary that the metric is fixed near the boundary, only that each metric in the isotopy has a product structure near the boundary. There will, it turns out, be a need to consider isotopies which fix the metric near the boundary. In this case, we will use the term {\em relative isotopy}. More precisely, let $g\in \Riem(X)$ and $\bar{g}, \bar{g}_{0}$ and  $\bar{g}_{1}\in \Riem(W, \p W=X)$ so that $\bar{g}|_{X}={\bar{g}_0}|_{X}={\bar{g}_{1}}|_{X}=g$. Then, $\Riem(W, (\p W, g))$ denotes the subspace of  $\Riem(W, \p W)$ consisting of metrics which restrict to $g$ on the boundary $\p W=X$. Similarly, $\Riem(Y, (\p Y, \bar{g}_0\cup \bar{g}_1), (\p^2 Y, g) )$ denotes the subspace of  $\Riem(Y, \p Y, \p^2 Y)$ consisting of metrics which restrict to $\bar{g}_0, \bar{g}_1$ and $g$ on $W_0$, $W_1$ and $X$. In the case where $g\in \Riem^{+}(X)$ and $\bar{g}_{0}$ and  $\bar{g}_{1}\in \Riem^{+}(W, \p W=X)$, $\Riem^{+}(W, (\p W,g))$ and  $\Riem^{+}(Y, (\p Y, \bar{g}_0\cup \bar{g}_1), (\p^2 Y, g) )$ denote the corresponding subspaces of psc-metrics. 

\begin{Definition}
{\rm A pair of metrics which are contained in the same path component of $\Riem^{+}(W, (\p W, g) )$ (respectively $\Riem^{+}(Y, (\p Y, \bar{g}_0\cup \bar{g}_1), (\p^2 Y, g) )$ ) are said to be {\em relative isotopic} to the metric $g$ (respectively $\bar{g}_{0}\cup\bar{g}_{1}$). }
\end{Definition}

\subsection{Torpedo Metrics}
Before discussing further this relative notion of isotopy, it is worth reviewing an important family of metrics, which were introduced in Part One: {\em torpedo metrics}.

As usual,  $S^{n}$ will denote the standard $n$-dimensional sphere. Throughout, we will assume that $n\geq 3$. We begin by recalling that the standard round metric on $S^{n}$, which we denote $ds_{n}^{2}$, is induced by the embedding
\begin{equation*}
\begin{split}
(0,\pi)\times{S^{n-1}}&\longrightarrow\mathbb{R}\times\mathbb{R}^{n},\\
(t,\theta)&\longmapsto(\cos{t},\sin{t}.\theta).
\end{split}
\end{equation*}
and computed in these coordinates as $dt^{2}+\sin^{2}(t)ds_{n-1}^{2}$. By replacing the $\sin $ term in this expression with a more general smooth function $f:(0,b)\rightarrow (0,\infty)$, we can construct various {\em warped product} metrics on the cylinder $(0,b)\times S^{n-1}$. Provided certain smoothness conditions are satisfied near the end points, we can ensure that the metric $dt^{2}+f(t)^{2}ds_{n-1}^{2}$ is a smooth metric on $S^{n}$; see Part One or \cite{P}. In Part One, we specify some conditions which guarantee such metrics have positive scalar curvature. Roughly speaking, this means ensuring that $\ddot{f}\leq 0$. Thus, by constructing appropriate homotopies of the function $f$, we obtain isotopies of the metric.  In Proposition 1.7 from Part One, we show that the space of psc-metrics which satisfy these conditions (which of course contains the round metric) is a path-connected space. 

By insisting that the function $f$ is positive and constant near $b$, we can construct psc-metrics on the disk $D^{n}$ which have the standard product structure near the boundary.  One important example is known as a {\em torpedo metric}, see Fig. \ref{torpedo}. More precisely, let $f_1$ be a smooth function on $(0,\infty)$ which satisfies the following conditions.
\begin{enumerate}
\item[{\bf (i)}] $f_1(t)=\sin{t}$ when $t$ is near $0$.
\item[{\bf (ii)}] $f_1(t)=1$ when $t\geq\frac{\pi}{2}$.
\item[{\bf (iii)}] $\ddot{f_{1}}(t)\leq 0$.
\end{enumerate}
More generally, for each $\delta>0$, the function $f_\delta:(0,\infty)\rightarrow(0,\infty)$ is defined by the formula
\begin{equation*}
f_{\delta}(t)=\delta f_1 (\frac{t}{\delta}).
\end{equation*}
By restricting $f_{\delta}$ to the interval $(0,b)$, where $b>\delta\frac{\pi}{2}$, the metric $dt^{2}+f_{\delta}(t)^{2}ds_{n-1}^{2}$ on $(0,b)\times S^{n-1}$, is a smooth $O(n)$-symmetric metric on the disk $D^{n}$ which is a round $n$-sphere of radius $\delta$ near the centre and a standard product of $(n-1)$-spheres of radius $\delta$ near the boundary. We denote this metric $\gtor^{n}(\delta)$ and note that its scalar curvature can be bounded below by an arbitrarily large positive constant, by choosing sufficiently small $\delta$.

\begin{figure}[!htbp]
\begin{picture}(0,0)%
\includegraphics{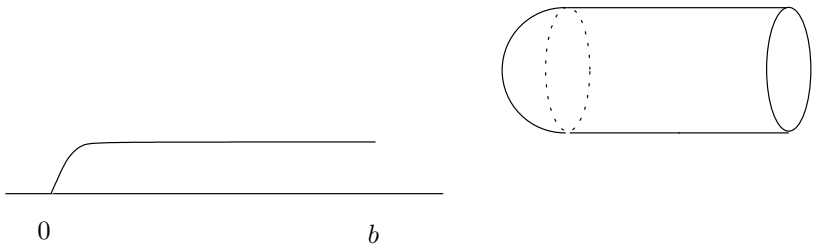}%
\end{picture}%
\setlength{\unitlength}{3947sp}%
\begingroup\makeatletter\ifx\SetFigFont\undefined%
\gdef\SetFigFont#1#2#3#4#5{%
  \reset@font\fontsize{#1}{#2pt}%
  \fontfamily{#3}\fontseries{#4}\fontshape{#5}%
  \selectfont}%
\fi\endgroup%
\begin{picture}(5079,1559)(1902,-7227)
\put(2114,-7136){\makebox(0,0)[lb]{\smash{{\SetFigFont{10}{8}{\rmdefault}{\mddefault}{\updefault}{\color[rgb]{0,0,0}$0$}%
}}}}
\put(4189,-7161){\makebox(0,0)[lb]{\smash{{\SetFigFont{10}{8}{\rmdefault}{\mddefault}{\updefault}{\color[rgb]{0,0,0}$b$}%
}}}}
\end{picture}%
\caption{A torpedo function and the resulting torpedo metric}
\label{torpedo}
\end{figure}
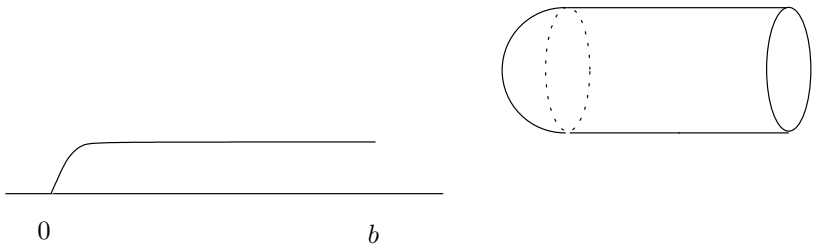   

By considering the torpedo metric as a metric on a hemisphere, we can 
obtain a metric on $S^{n}$ by taking its double. Such a metric is given by the formula $dt^{2}+\bar{f_{\delta}}(t)^{2}ds_{n-1}^{2}$, where $\bar{f_{\delta}}:(0,b)\rightarrow(0,\infty)$ agrees with $f_{\delta}$ on $(0,\frac{b}{2})$ and is given by the formula $\bar{f_\delta}(t)=f_{\delta}(b-t)$ on $(\frac{b}{2}, b)$. Here we assume that $\delta\frac{\pi}{2}<\frac{b}{2}$. Such a 
metric will be called a {\it double torpedo metric of radius $\delta$} 
and denoted $g_{Dtor}^{n}(\delta)$; see Fig. \ref{doubletorpedo}. It is easily shown, using Proposition 1.7 from Part One, that this metric is isotopic to the standard round metric $ds_{n}^{2}$. 
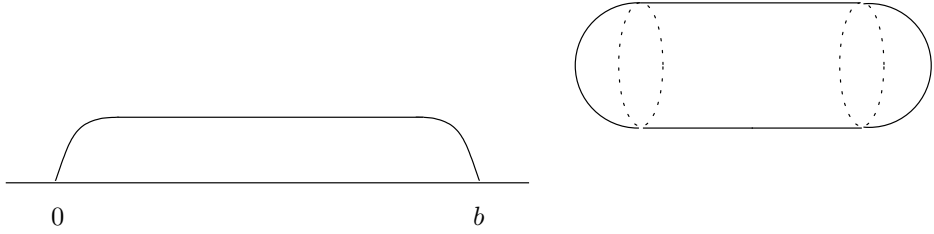
\begin{figure}[htbp]
\begin{picture}(0,0)%
\includegraphics{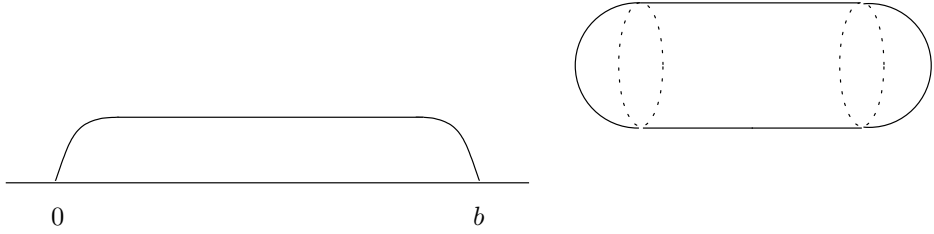}%
\end{picture}%
\setlength{\unitlength}{3947sp}%
\begingroup\makeatletter\ifx\SetFigFont\undefined%
\gdef\SetFigFont#1#2#3#4#5{%
  \reset@font\fontsize{#1}{#2pt}%
  \fontfamily{#3}\fontseries{#4}\fontshape{#5}%
  \selectfont}%
\fi\endgroup%
\begin{picture}(5833,1472)(995,-3093)
\put(3944,-3027){\makebox(0,0)[lb]{\smash{{\SetFigFont{10}{8}{\rmdefault}{\mddefault}{\updefault}{\color[rgb]{0,0,0}$b$}%
}}}}
\put(1294,-3027){\makebox(0,0)[lb]{\smash{{\SetFigFont{10}{8}{\rmdefault}{\mddefault}{\updefault}{\color[rgb]{0,0,0}$0$}%
}}}}
\end{picture}%
\caption{A double torpedo function and the resulting double torpedo metric}
\label{doubletorpedo}
\end{figure}

\begin{Remark}
In general, we will suppress the $\delta$ term when writing $\gtor^{n}(\delta)$, $g_{Dtor}^{n}(\delta)$ etc and simply write $\gtor^{n}$, $g_{Dtor}^{n}$ etc, knowing that we may choose $\delta$ to be arbitrarily small if necessary. This is further justified by the fact the scalar curvature of such metrics is positive for any choice of $\delta$ (provided $n\geq 3$) and that a continuous variation of $\delta$ induces an isotopy in the respective metrics.
\end{Remark}

An obvious property of the round sphere metric is that its restriction to the equator is also round, albeit one dimension lower. Moreover, the equator divides the round sphere into two isometric pieces: the upper and lower round hemispheres. Thus the torpedo metric can be thought of as obtained by cutting the sphere in half and then gluing one of these hemispheres along the boundary to a cylindrical product of the equator metric, as shown in the top and middle left pictures of Fig. \ref{reldoubletor}. Some smoothing is of course necessary but, hueristically, this is what happens. 

An analogous procedure can be carried out on the metric $g_{Dtor}^{n}$. We will describe this more precisely soon, but for now a rough sketch is sufficient. Viewing this metric as pictured in Fig. \ref{doubletorpedo}, there are now two equators we might consider: a vertical and a horizontal equator. Notice that the horizontal equator is of course the $(n-1)$-dimensional analogue of $g_{Dtor}^{n}$, namely $g_{Dtor}^{n-1}$; see the top right picture in Fig. \ref{reldoubletor}. By slicing the metric $g_{Dtor}^{n}$ here, we obtain a pair of hemisphere metrics which we denote $g_{Dtor}^{n}(-)$ and $g_{Dtor}^{n}(+)$. By gluing (with appropriate smoothing) the hemisphere $g_{Dtor}^{n}(+)$ to the product $g_{Dtor}^{n-1}+dt^{2}$, we obtain the metric  $\bar{g}_{Dtor}^{n}$. This metric is shown on the middle right of Fig. \ref{reldoubletor}.

We close this section by generalising the above construction one step further to manifolds with corners. 
Beginning with $\gtor^{n}$ we obtain, as an ``equator" metric, the metric $\gtor^{n-1}$ in the obvious way. That is, given the parameterisation $\gtor^{n}=dt^{2}+f_{\delta}(t)^{2}ds_{n-1}^{2}$, we consider $\gtor^{n-1}=dt^{2}+f_{\delta}(t)^{2}ds_{n-2}^{2}$ to be the metric obtained by restriction to the equator sphere $S^{n-2}\subset S^{n-1}$. As before, we obtain a decomposition into hemi-disk metrics $\gtor^{n}=\gtor^{n}(+)\cup\gtor^{n}(-)$. This is shown in the middle left picture in Fig. \ref{reldoubletor}. By attaching $\gtor^{n}(+)$ to the cylinder metric $\gtor^{n-1}+dt^{2}$ (and making appropriate smoothing adjustments) we obtain the metric $\bar{\bar{g}}_{tor}^{n}$ shown in the bottom left of Fig. \ref{reldoubletor}.
Turning our attention to $\bar{g}_{Dtor}^{n}$ on $D^{n+2}$, there is an obvious copy of $\bar{g}_{Dtor}^{n-1}$ which divides the disk into two hemidisk metrics $\bar{g}_{Dtor}^{n}(\pm)$; see the middle right picture in Fig. \ref{reldoubletor}. As before, we attach (making appropriate smoothing adjustments) $\bar{g}_{Dtor}^{n}(+)$  to the cylinder metric $\bar{g}_{Dtor}^{n-1}+ds^{2}$ on $D^{n-1}\times I$ to obtain the metric $\bar{\bar{g}}_{Dtor}^{n}$ shown in the bottom right of Fig. \ref{reldoubletor}. 

\begin{figure}[!htbp]
\vspace{7cm}
\begin{picture}(0,0)%
\hspace{-1cm}
\includegraphics{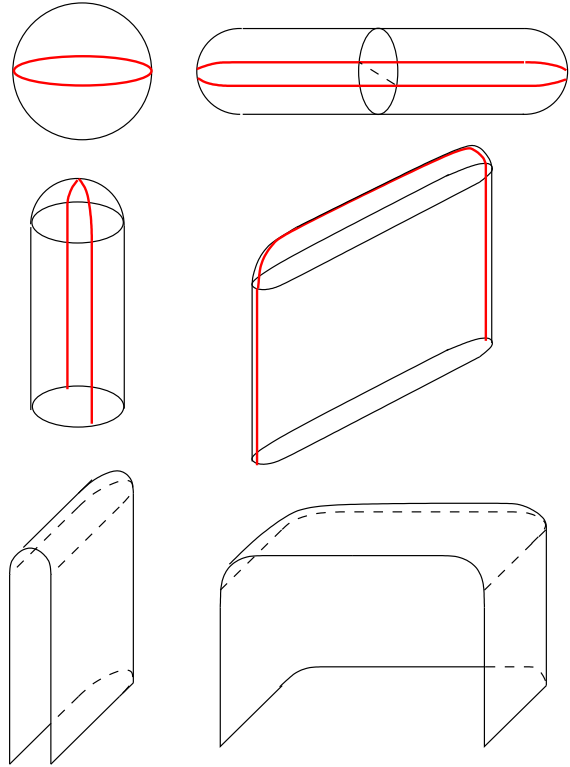}%
\end{picture}%
\setlength{\unitlength}{3947sp}%
\begingroup\makeatletter\ifx\SetFigFont\undefined%
\gdef\SetFigFont#1#2#3#4#5{%
  \reset@font\fontsize{#1}{#2pt}%
  \fontfamily{#3}\fontseries{#4}\fontshape{#5}%
  \selectfont}%
\fi\endgroup%
\begin{picture}(2079,1559)(1002,-5227)
\end{picture}%
\caption{The metrics $ds_{n}^{2}$ and ${g}_{Dtor}^{n}$ (top), $\bar{g}_{tor}^{n}$ and $\bar{g}_{Dtor}^{n}$ (middle), and  $\bar{\bar{g}}_{tor}^{n}$ and  $\bar{\bar{g}}_{Dtor}^{n}$ (bottom) }
\label{reldoubletor}
\end{figure} 

We now consider an important example of an isotopy of psc-metrics on the disk $D^{n+2}$, which is relative to the boundary $\p D^{n+2}=S^{n+1}$. This isotopy involves a continuous deformation of the standard torpedo metric $\gtor^{n+2}$ to a metric which we denote by $\bar{g}_{Dtor}^{n+2}$. 

\begin{Lemma}\label{reltorplemma}
For $n\geq 2$, the metrics  $\bar{g}_{Dtor}^{n+2}$ and $\gtor^{n+2}$ are isotopic on $D^{n+2}$.
\end{Lemma}
\begin{proof}
We begin with an alternate description of the torpedo metric $\gtor^{n+2}$. Previously, we thought of an $(n+2)$-dimensional torpedo metric as obtained by taking a round $(n+1)$-sphere and then tracing out first a cylinder and then a hemisphere, by smoothly adjusting the radius. In this case, it is better to start with a round $(n+1)$-dimensional hemisphere. We will then construct the torpedo metric by first tracing out the right hand side of the round cylinder, then rotating by an angle $\pi$ to trace out the round hemisphere, before finishing with the left hand side of the round cylinder. This is shown in Fig. \ref{newcoordtor}.
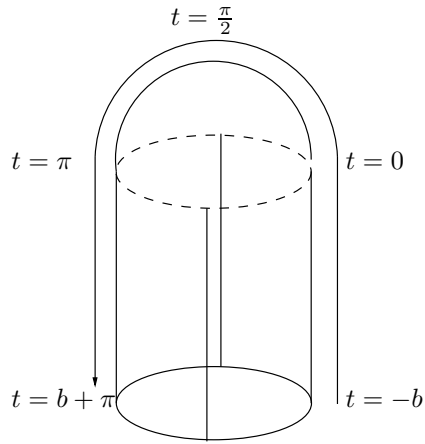
\begin{figure}[!htbp]
\vspace{3cm}
\begin{picture}(0,0)%
\includegraphics{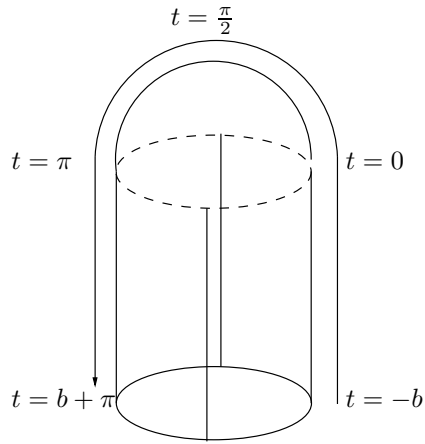}%
\end{picture}%
\setlength{\unitlength}{3947sp}%
\begingroup\makeatletter\ifx\SetFigFont\undefined%
\gdef\SetFigFont#1#2#3#4#5{%
  \reset@font\fontsize{#1}{#2pt}%
  \fontfamily{#3}\fontseries{#4}\fontshape{#5}%
  \selectfont}%
\fi\endgroup%
\begin{picture}(2079,1559)(1002,-5227)
\put(2600,-3500){\makebox(0,0)[lb]{\smash{{\SetFigFont{10}{8}{\rmdefault}{\mddefault}{\updefault}{\color[rgb]{0,0,0}$t=0$}%
}}}}
\put(2600,-5000){\makebox(0,0)[lb]{\smash{{\SetFigFont{10}{8}{\rmdefault}{\mddefault}{\updefault}{\color[rgb]{0,0,0}$t=-b$}%
}}}}

\put(1500,-2600){\makebox(0,0)[lb]{\smash{{\SetFigFont{10}{8}{\rmdefault}{\mddefault}{\updefault}{\color[rgb]{0,0,0}$t=\frac{\pi}{2}$}%
}}}}

\put(500,-3500){\makebox(0,0)[lb]{\smash{{\SetFigFont{10}{8}{\rmdefault}{\mddefault}{\updefault}{\color[rgb]{0,0,0}$t=\pi$}%
}}}}

\put(500,-5000){\makebox(0,0)[lb]{\smash{{\SetFigFont{10}{8}{\rmdefault}{\mddefault}{\updefault}{\color[rgb]{0,0,0}$t=b+\pi$}%
}}}}

\end{picture}%
\caption{Alternate description of the torpedo metric}
\label{newcoordtor}
\end{figure} 

We now determine a formula for $\gtor^{n+2}$ in coordinates which adhere to this description. Let $r\in (0,\frac{\pi}{2})$ and $t\in (-b, b+\pi)$. Now, let $\alpha$ denote a smooth function on $(0,\frac{\pi}{2})\times (-b, b+\pi)$ which is defined as follows:
\begin{enumerate}
\item[{\bf (i)}] $\alpha(r,t)=1$ when $t\in(-b,-\epsilon)$,
\item[{\bf (ii)}] $\alpha(r,t)=\sin r$ when $t\in(\epsilon, {\pi}-\epsilon)$,
\item[{\bf (iii)}] $\alpha(r,t)=1$ when $t\in (\pi+{\epsilon}, b)$.
\end{enumerate}
Here $\epsilon$ is assumed to be arbitrarily small and, furthermore, we assume that $0\leq \frac{\p \alpha}{\p r}\leq 1$ and that $\frac{\p^{2} \alpha}{\p r^{2}}\leq 0$. By choosing an appropriate bump function $\mu:(-b, \pi+b)\rightarrow [0,1]$ we may assume that $\alpha$ takes the form
\begin{equation*}
\alpha(r,t)=1-\mu(t)+\mu(t)\sin(r).
\end{equation*}
\noindent In these coordinates, the torpedo metric $\gtor^{n+2}$ is given by the formula
\begin{equation*}
g_{tor}^{n+2}=dr^{2}+\alpha^{2}dt^{2}+\cos^{2}r ds_{n}^{2},
\end{equation*}
where of course $ds_{n}^{2}$ denotes the round metric on $S^{n}$. For convenience, we are only considering the case $\gtor^{n+2}=g_{tor}^{n+2}(1)$. It will be clear that the lemma holds if $1$ is replaced by any $\delta>0$.

We will perform a deformation of the metric $\gtor^{n+2}$ over two stages. In the first stage, we will replace the round $(n+1)$-dimensional hemisphere with an $(n+1)$-dimensional torpedo. This involves a homotopy of the $\cos r$ term to a term $f(r)$ where $f$ is an appropriate torpedo function. In otherwords, $f(r)=\cos r$ when $r$ is away from $0$ and $f(r)=1$ when $r$ is near $0$. Moreover $f$ satisfies the usual condition on torpedo functions that $\ddot{f}\leq 0$. At the same, we adjust $\alpha$ so that the $\sin(r)$ term is replaced by a term $h(r)$ where $h:(0, \frac{\pi}{2})\rightarrow (0,1)$ is a smooth function satisfying:
\begin{enumerate}
\item{}$h(r)=r$, near $0$,
\item{}$h(r)=\sin(r)$, away from $0$.
\end{enumerate}
\noindent The resulting metric is represented in Fig. \ref{deformtorp1}. 

\begin{figure}[!htbp]
\begin{picture}(0,0)%
\includegraphics{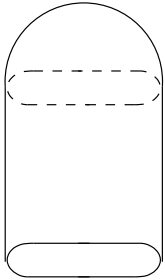}%
\end{picture}%
\setlength{\unitlength}{3947sp}%
\begingroup\makeatletter\ifx\SetFigFont\undefined%
\gdef\SetFigFont#1#2#3#4#5{%
  \reset@font\fontsize{#1}{#2pt}%
  \fontfamily{#3}\fontseries{#4}\fontshape{#5}%
  \selectfont}%
\fi\endgroup%
\begin{picture}(2079,1559)(1002,-5227)
\end{picture}%
\caption{The metric resulting from the first deformation}
\label{deformtorp1}
\end{figure} 

The second and final deformation is to stretch the metric in the horizontal direction to obtain $\bar{g}_{Dtor}^{n+2}$. This is done by adjusting the $\alpha$ term so that near $t=\frac{\pi}{2}$, $\alpha=1$. The resulting metric is now a product of torpedo metrics near $t=\frac{\pi}{2}$ and we can stretch it horizontally as much as we require. By choosing an appropriate family of ``cut-off" style functions $\mu:(-b, b+\pi)\rightarrow [0,1]$, each stage in this deformation takes the form
\begin{equation*}
\alpha(r,t)=(1-\mu(t))+\mu(t)h(r),
\end{equation*}
\noindent thus ensuring that the conditions $0\leq \frac{\p \alpha}{\p r}\leq 1$ and $\frac{\p^{2} \alpha}{\p r^{2}}\leq 0$ are preserved throughout. 

It is clear that this deformation preserves the product structure of the metric near the boundary. It remains to show that positive scalar curvature is preserved throughout. The scalar curvature $R$, of the more general metric 
\begin{equation*}
dr^{2}+\alpha^{2}(r,t)dt^{2}+F^{2}(r,t)ds_{n}^{2}\quad \mbox{on $(0,\frac{\pi}{2})\times (-b,b+\pi)\times S^{n}$,}
\end{equation*}
is given by the following formula.  
\begin{equation}\label{scaldeftorp}
\begin{split}
R = &\frac{n(n-1)}{F^{2}}\Big{[}1-\Big{(}\frac{\p F}{\p r}\Big{)}^{2}-\frac{(\frac{\p F}{\p t})^{2}}{\alpha^{2}}\Big{]}-\frac{2n}{F}\Big{[}\frac{\p^{2} F}{\p r^{2}}+\frac{\frac{\p F}{\p r}.\frac{\p \alpha}{\p r}}{\alpha}\Big{]}\\
&+\frac{2n}{\alpha^{2}F}\Big{[}-\frac{\p^{2} F}{\p t^{2}}+\frac{\frac{\p F}{\p t}.\frac{\p \alpha}{\p t}}{\alpha}\Big{]} - 2\frac{(\frac{\p^{2} \alpha}{\p r^{2}})}{\alpha}.
\end{split}
\end{equation}
This formula is the result of a straightforward, albeit long, calculation. In our case, $F(r,t)$ may be replaced with $f(r)$ where, by abuse of notation, $f$ represents any stage in the homtopy between $\cos$ and the torpedo function $f$ described above. It is obvious that $-1\leq \frac{\p f}{\p r} \leq 0$ and that $\frac{\p^{2} f}{\p r^{2}}\leq 0$. The formula in equation \ref{scaldeftorp} now simplifies to
\begin{equation}\label{scaldeftorp2}
R = \frac{n(n-1)}{f^{2}}\Big{[}1-\Big{(}\frac{\p f}{\p r}\Big{)}^{2}\Big{]}-\frac{2n}{f}\Big{[}\frac{\p^{2} f}{\p r^{2}}+\frac{\frac{\p f}{\p r}.\frac{\p \alpha}{\p r}}{\alpha}\Big{]} - 2\frac{(\frac{\p^{2} \alpha}{\p r^{2}})}{\alpha}.
\end{equation}

The fact that $-1\leq \frac{\p f}{\p r} \leq 0$ means that the first term in this expression is non-negative. Furthermore, when $r$ is away from $\frac{\pi}{2}$, $\frac{\p f}{\p r}>-1$ and so this term is strictly positive. Finally, when $r$ is near $\frac{\pi}{2}$, $f(r)=\cos{r}$ and so the first term is easily seen to equal $n(n-1)$. Hence, the first term is always strictly positive. Non-negativity of the rest of the expression then follows easily from the fact that both $\frac{\p^{2} \alpha}{\p r^{2}}$ and $\frac{\p^{2} f}{\p r^{2}}$ are non-positive, while whenever both $\frac{\p \alpha}{\p r}$ and $\frac{\p f}{\p r}$ are non-zero, they have opposite signs.
\end{proof}

A completely analogous argument gives us the following lemma.

\begin{Lemma}\label{relcornerlemma}
For $n\geq 3$, the metrics $\bar{\bar{g}}_{Dtor}^{n+2}$ and $\bar{\bar{g}}_{tor}^{n+2}$ are  isotopic.
\end{Lemma}
\begin{proof}
This follows easily from Lemma \ref{reltorplemma}. 
\end{proof}

One difficulty in describing isotopies of the metrics above is the obvious dimensional restrictions inherent in our schematic pictures. We close this section with an alternative schematic description of the metrics considered above, which will be very useful when it comes to the main theorem. To save as much ``dimensional space" as possible, it is useful to abbreviate the standard torpedo metric with a picture of a shaded disk, as shown in the upper left picture in Fig. \ref{solidschematic}. Continuing in this vein, we abbreviate the metrics $\bar{g}_{Dtor}^{n}$, $\bar{\bar{g}}_{tor}^{n}$ and $\bar{\bar{g}}_{Dtor}^{n}$ as shown in the upper right, middle left and middle right picture in Fig. \ref{solidschematic}. By retaining the shaded disk as a representation for $\bar{g}_{tor}^{n}$, the metric $\bar{\bar{g}}_{tor}^{n+1}$ can now be described with the ``solid" torpedo schematic shown at the bottom left of  Fig. \ref{solidschematic}. Similarly, $\bar{\bar{g}}_{Dtor}^{n+1}$ is depicted at the bottom right of this figure.

\begin{figure}[!htbp]
\vspace{7cm}
\begin{picture}(0,0)%
\hspace{-3cm}
\includegraphics{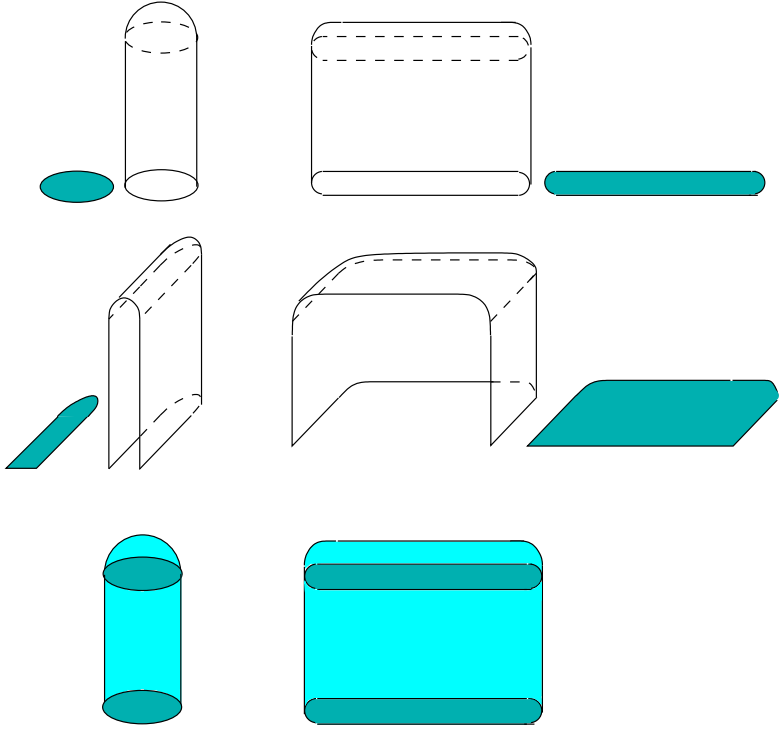}%
\end{picture}%
\setlength{\unitlength}{3947sp}%
\begingroup\makeatletter\ifx\SetFigFont\undefined%
\gdef\SetFigFont#1#2#3#4#5{%
  \reset@font\fontsize{#1}{#2pt}%
  \fontfamily{#3}\fontseries{#4}\fontshape{#5}%
  \selectfont}%
\fi\endgroup%
\begin{picture}(2079,1559)(1002,-5227)
\end{picture}%
\caption{Alternative solid schematic description of the metrics shown in Fig. \ref{reldoubletor}}
\label{solidschematic}
\end{figure} 

\subsection{Mixed Torpedo Metrics}

So far, we have mostly dealt with metrics which take the form of a warped product metric $dt^{2}+f(t)^{2}ds_{n-1}^{2}$ on $(0,b)\times S^{n-1}$. The notion of a warped product metric on 
$(0,b)\times S^{n-1}$ generalises to something called a {\it doubly 
warped product metric} on $(0,b)\times S^{p}\times S^{q}$. Here, the 
metric takes the form 
\begin{equation*}
dt^{2}+u(t)^{2}ds_{p}^{2}+v(t)^{2}ds_{q}^{2}, 
\end{equation*}
where $u,v:(0,b)\rightarrow(0,\infty)$ are smooth functions. As before, we can specify certain end point conditions on the functions $u$ and $v$ to obtain a smooth metric on the sphere $S^{n}$; see Part One or \cite{P}. For example, the standard round metric on $S^{n}$ (where $p+q+1=n$) is obtained as the metric induced by the embedding
\begin{equation}\label{map;doublewarpembedding} 
\begin{split}
(0,\frac{\pi}{2})\times{{S}^{p}}\times{{S}^{q}}&\longrightarrow\mathbb{R}^{p+1}\times\mathbb{R}^{q+1}\\
 (t,\phi,\theta)&\longmapsto(\cos{t}.\phi,\sin{t}.\theta),
\end{split}
\end{equation}
and computed as
\begin{equation*} 
\begin{array}{c}
dt^{2}+\cos^{2}(t)ds_{p}^{2}+\sin^{2}(t)ds_{q}^{2}.
\end{array}
\end{equation*} 
In Part One, we specify some conditions on the functions $u$ and $v$ which ensure that the resulting metric has positive scalar curvature. These conditions are then used to construct a path-connected space of psc-metrics on the sphere $S^{n}$. Importantly, this space contains the round metric and so we may allow the values of $p$ and $q$ to vary and still be sure that this space is path-connected. Contained in this space is an important class of sphere metrics on $S^{n}$ called {\em mixed torpedo metrics}. These metrics are constructed as follows.

Appropriate restrictions of the above embedding correspond to the following decomposition of $S^{n}$ into a union of sphere-disk products.
\begin{equation*}
\begin{array}{cl}
S^{n} &= \p D^{n+1},\\
&=\p (D^{p+1}\times D^{q+1}),\\ 
&=S^{p}\times D^{q+1}\cup_{S^{p}\times S^{q}} D^{p+1}\times S^{q}.
\end{array}
\end{equation*}
Equip $S^{p}\times D^{q+1}$ with the product metric 
$\epsilon^{2}ds_{p}^{2}+\bar{g}_{tor}^{q+1}(\delta)$.
Then equip $D^{p+1}\times S^{q}$ with $\gtor^{p+1}(\epsilon)+\delta^{2}ds_{q}^{2}$. 
These metrics glue together smoothly along the common boundary 
$S^{p}\times S^{q}$ to form a smooth metric on $S^{n}$. Such metrics will be known
as a {\it mixed torpedo metrics} on $S^{n}$ and denoted $g_{Mtor}^{p,q}$. In Fig. \ref{fig:mixed torpedo}, we schematically depict the metrics  $g_{Mtor}^{p,q}$ and $g_{Mtor}^{p+1,q-1}$.
\begin{figure}[htbp]
\includegraphics[height=30mm]{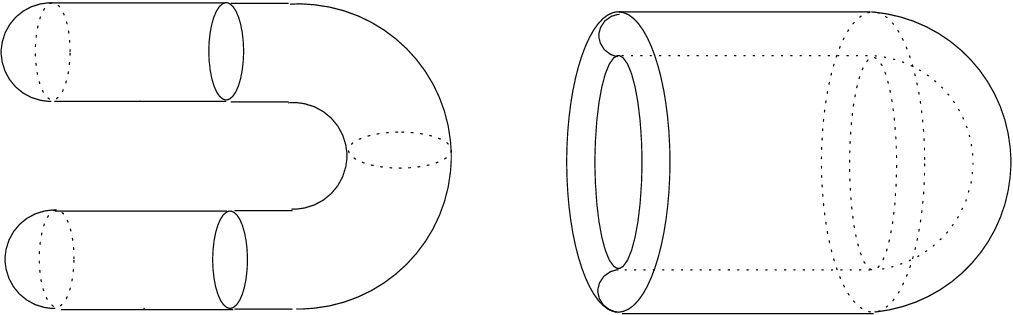}
\caption{The mixed torpedo metrics $g_{Mtor}^{p,q}$ and $g_{Mtor}^{p+1,q-1}$}
\label{fig:mixed torpedo}
\end{figure}
The metric $g_{Mtor}^{p,q}$ can be realised as a doubly warped product metric on $(0,b)\times S^{p}\times S^{q}$, given by the formula
\begin{equation}
\begin{array}{c}
 g_{Mtor}^{p,q}=dt^{2}+f_{\epsilon}(b-t)^{2}ds_{p}^{2}+f_{\delta}(t)^{2}ds_{q}^{2}.
\end{array}
\end{equation}
The fact that these mixed torpedo metrics lie, along with the standard round metric, in a path-connected subspace of $\Riem^{+}(S^{n})$ is proved in Lemma 1.11 of Part One. Given the importance of this fact in our work, we restate this lemma here. 
\begin{Lemma}{\rm [1.11 of Part One]} \label{toriso}
Let $n\geq 3$. For any non-negative integers $p$ and $q$ with $p+q+1=n$ and $p$ or $q\geq2$, the metric $g_{Mtor}^{p,q}$ is isotopic to $ds_{n}^{2}$. In particular, if $q\geq 3$, the metrics $g_{Mtor}^{p,q}$ and $g_{Mtor}^{p+1, q-1}$ are isotopic.
\end{Lemma}
We will now discuss an analogous version of this lemma for metrics on the disk. We begin by constructing a psc-metric on the disk $D^{n+1}$ which, near the boundary, takes the form $ds^{2}+g_{Mtor}^{p,q}$. 
The disk $D^{n+1}$ decomposes as 
\begin{equation*}
D^{n+1}=D^{p+1}\times D^{q+1}\cup S^{p}\times D^{q+2}. 
\end{equation*}
Schematically, this can be thought of as a solid version of the left hand picture in Fig. \ref{fig:mixed torpedo}. We now equip the $D^{p+1}\times D^{q+1}$ piece, with a product of torpedo metrics $\gtor^{p+1}\times g_{tor}^{q+1}$ and the $S^{p}\times D^{q+2}$ piece, with the metric $ds_{p}^{2}\times \bar{g}_{tor}^{q+2}$ (a product of a round metric and the torpedo metric with corners described above). These pieces glue smoothly together, resulting in the relative mixed torpedo metric $\bar{g}_{Mtor}^{p,q+1}$; see the left hand drawing in Fig. \ref{relmixtorp}.

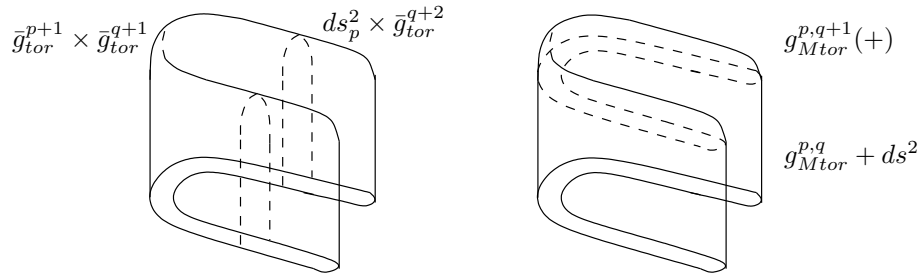
\begin{figure}[!htbp]
\vspace{-2cm}
\hspace{-3cm}
\begin{picture}(0,0)%
\includegraphics{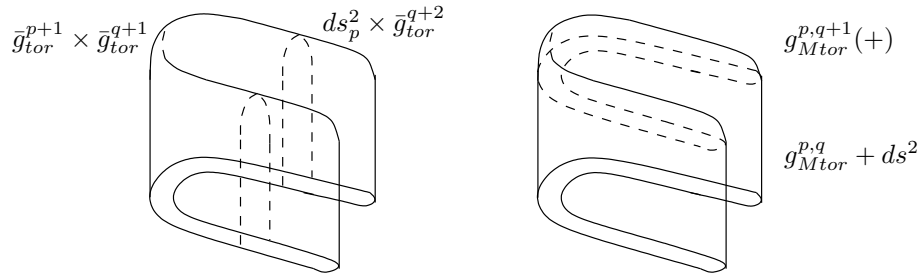}%
\end{picture}%
\setlength{\unitlength}{3947sp}%
\begingroup\makeatletter\ifx\SetFigFont\undefined%
\gdef\SetFigFont#1#2#3#4#5{%
  \reset@font\fontsize{#1}{#2pt}%
  \fontfamily{#3}\fontseries{#4}\fontshape{#5}%
  \selectfont}%
\fi\endgroup%
\begin{picture}(2079,2559)(1902,-5227)
\put(1050,-3800){\makebox(0,0)[lb]{\smash{{\SetFigFont{10}{8}{\rmdefault}{\mddefault}{\updefault}{\color[rgb]{0,0,0}$\gtor^{p+1}\times \gtor^{q+1}$}%
}}}}
\put(3000,-3700){\makebox(0,0)[lb]{\smash{{\SetFigFont{10}{8}{\rmdefault}{\mddefault}{\updefault}{\color[rgb]{0,0,0}$ds_{p}^{2}\times \gtor^{q+2}$}%
}}}}
\put(5900,-3800){\makebox(0,0)[lb]{\smash{{\SetFigFont{10}{8}{\rmdefault}{\mddefault}{\updefault}{\color[rgb]{0,0,0}${g_{Mtor}^{p,q+1}}{(+)}$}%
}}}}
\put(5900,-4527){\makebox(0,0)[lb]{\smash{{\SetFigFont{10}{8}{\rmdefault}{\mddefault}{\updefault}{\color[rgb]{0,0,0}$g_{Mtor}^{p,q}+ds^{2}$}%
}}}}
\end{picture}%
\caption{Alternative decompositions of the relative mixed torpedo metric $\bar{g}_{Mtor}^{p,q+1}$ on $D^{n+1}$}
\label{relmixtorp}
\end{figure} 

Alternatively, it is useful to consider this metric as obtained in the folloing way. Notice that the mixed torpedo metric $g_{Mtor}^{p,q+1}$ on $S^{n+1}$ has an equator sphere $S^{n}$ on which the induced metric is $g_{Mtor}^{p,q}$. We will denote by ${g_{Mtor}^{p,q+1}}{(+)}=g_{Mtor}^{p,q+1}|_{D_{+}^{n+1}}$, the restriction of $g_{Mtor}^{p,q+1}$ to the upper hemisphere of $S^{n+1}$.  We now glue together the cylinder metric $ds^{2}+g_{Mtor}^{p,q}$ on $S^{n}\times I$ and the hemisphere metric ${g_{Mtor}^{p,q+1}}{(+)}$, by identifying the top of the cylinder with the boundary of the disk $D_{+}^{n+1}$ (making appropriate smoothing adjustments near the joining slice) to obtain the metric $\bar{g}_{Mtor}^{p,q+1}$; see the right hand drawing in Fig. \ref{relmixtorp}. 

As a consequence of Lemma \ref{toriso} above, we obtain the following lemma.
\begin{Lemma}\label{relmixtorplemma}
Let $n\geq 3$, $p+q+1=n$ and $p$ or $q\geq2$. The metrics $\bar{g}_{Mtor}^{p,q}$ and $\bar{g}_{tor}^{n+1}$ are isotopic in $\Riem^{+}(D^{n+1}, \p D^{n+1}=S^{n})$.
\end{Lemma}
\begin{proof}
This follows from Lemma \ref{toriso} combined with Lemma \ref{reltorplemma}.
\end{proof}
\begin{Corollary}\label{torisomix}
Let $n\geq 4$, $p+q+1=n$ and $q\geq3$. The metrics $\bar{g}_{Mtor}^{p,q+1}$ and $\bar{g}_{Mtor}^{p+1,q}$ are isotopic in $\Riem^{+}(D^{n+1}, S^{n})$.
\end{Corollary}
Finally, using the fact that the isotopy is slicewise near the boundary we obtain the following important lemma.
\begin{Lemma}\label{reltoriso}
Let $n\geq 4$, $p+q+1=n$ and $q\geq3$. The metrics $\bar{g}_{Mtor}^{p,q+1}$ and $\bar{g}_{Mtor}^{p+1,q}$ are isotopic, relative to the metric $g_{Mtor}^{p,q}$ on the boundary $S^{n}$, in $\Riem^{+}(D^{n+1}, S^{n})$.
\end{Lemma}
\begin{proof}
As the isotopy between $\bar{g}_{Mtor}^{p,q+1}$ and $\bar{g}_{Mtor}^{p+1,q}$ is slicewise near the boundary and as both of these metrics restrict to $g_{Mtor}^{p,q}$, it makes sense to construct an isotopy $\bar{g}_t$, with $t\in I$, which is a product $g_{Mtor}^{p,q}+ds^{2}$ near the boundary, throughout. This is done using Lemma 1.3 from Part One. Essentially we perform the entire isotopy outside of a collar neighbourhood of the boundary. Near this collar neighbourhood we smoothly adjust the metric along slices. Thus at each time $t$, we have a metric $\bar{g}_t$ which takes the form of a warped product of the form $\bar{g}_t = g_t(s)+ds^{2}$ on a collar neighbourhood of the boundary, diffeomorphic to $S^{n}\times I$. Furthermore, near the boundary, i.e. on $S^{n}\times [0,\epsilon]$ for some $\epsilon>0$, $g_t(s)=g_{Mtor}^{p,q}$. Positivity of the scalar curvature is maintained by an appropriate smooth scaling of the metric in the $s$ direction, as described in Lemma 1.3 of Part One.
\end{proof}

It is helpful to schematically compare the metrics described in Corollary \ref{reltoriso}. We do this in Fig. \ref{relmixtoralt} below. The fact that these metrics are isotopic relative to $g_{Mtor}^{p,q}$, will play an important role later on.

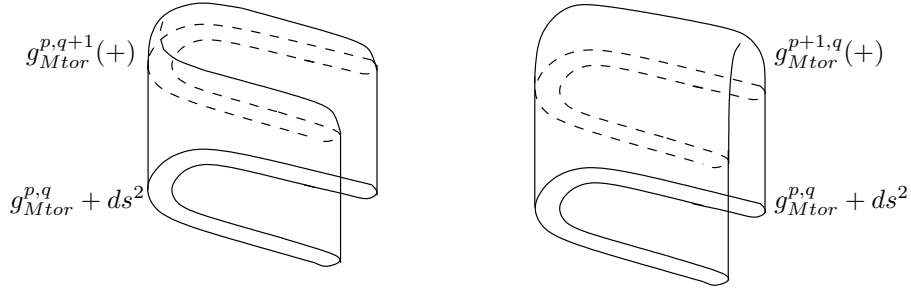
\begin{figure}[!htbp]
\vspace{-2cm}
\hspace{-3cm}
\begin{picture}(0,0)%
\includegraphics{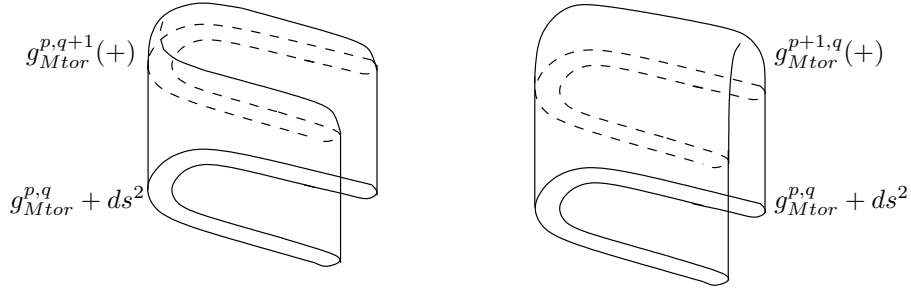}%
\end{picture}%
\setlength{\unitlength}{3947sp}%
\begingroup\makeatletter\ifx\SetFigFont\undefined%
\gdef\SetFigFont#1#2#3#4#5{%
  \reset@font\fontsize{#1}{#2pt}%
  \fontfamily{#3}\fontseries{#4}\fontshape{#5}%
  \selectfont}%
\fi\endgroup%
\begin{picture}(2079,2559)(1902,-5227)

\put(1150,-3800){\makebox(0,0)[lb]{\smash{{\SetFigFont{10}{8}{\rmdefault}{\mddefault}{\updefault}{\color[rgb]{0,0,0}${g_{Mtor}^{p,q+1}}{(+)}$}%
}}}}
\put(1050,-4727){\makebox(0,0)[lb]{\smash{{\SetFigFont{10}{8}{\rmdefault}{\mddefault}{\updefault}{\color[rgb]{0,0,0}$g_{Mtor}^{p,q}+ds^{2}$}%
}}}}
\put(5850,-3800){\makebox(0,0)[lb]{\smash{{\SetFigFont{10}{8}{\rmdefault}{\mddefault}{\updefault}{\color[rgb]{0,0,0}${g_{Mtor}^{p+1,q}}{(+)}$}%
}}}}
\put(5850,-4727){\makebox(0,0)[lb]{\smash{{\SetFigFont{10}{8}{\rmdefault}{\mddefault}{\updefault}{\color[rgb]{0,0,0}$g_{Mtor}^{p,q}+ds^{2}$}%
}}}}

\end{picture}%
\caption{The metrics $\bar{g}_{Mtor}^{p,q+1}$ (left) and $\bar{g}_{Mtor}^{p+1,q}$ (right) on $D^{n+1}$}
\label{relmixtoralt}
\end{figure}

We close this section by generalising the above construction one step further to manifolds with corners. In the previous section, we thought of a torpedo metric as the result of slicing a round sphere into two round hemispheres and then (after some smoothing adjustments) attaching one of these hemispheres to a product of its own equator. The metric $\bar{g}_{Mtor}^{p,q}$ is obtained by applying this procedure to a hemisphere $g_{Mtor}^{p+1,q}$. We will now apply an analogous procedure to the metric $\bar{g}_{Mtor}^{p,q+1}$ to obtain a mixed torpedo metric with corners on the disk $D^{n+2}$.

We now consider the metric $\bar{g}_{Mtor}^{p,q+1}$ on $D^{n+2}$. Recall that, near the boundary, this metric takes the form $ds^{2}+g_{Mtor}^{p,q+1}$, where $g_{Mtor}^{p,q+1}$ is the mixed torpedo metric on $S^{n+1}$. Applying an entirely analogous construction to this metric, we obtain the metric $\bar{\bar{g}}_{Mtor}^{p,q}$ depicted in Fig. \ref{solidrelmixtorp}. In other words, by slicing along the ``equator" metric $\bar{g}_{Mtor}^{p,q}$, we can divide $\bar{g}_{Mtor}^{p,q+1}$ into two equal parts $\bar{g}_{Mtor}^{p,q+1}(\pm)$. Then, after suitable smoothing adjustments, we can attach $\bar{g}_{Mtor}^{p,q+1}(+)$ to the cylinder metric $\bar{g}_{Mtor}^{p,q}+ ds^{2}$ in the obvious way, to obtain the desired metric. Alternatively, we may decompose $D^{n+2}$ into $D^{p+1}\times D^{q+2}\cup S^{p}\times D^{q+3}$ and obtain $\bar{\bar{g}}_{Mtor}^{p,q}$ as the union 
\begin{equation*}
\bar{\bar{g}}_{Mtor}^{p,q}=g_{tor}^{p+1}+g_{tor}^{q+2}\cup ds_{p}^{2}+ \bar{g}_{tor}^{q+3}. 
\end{equation*}
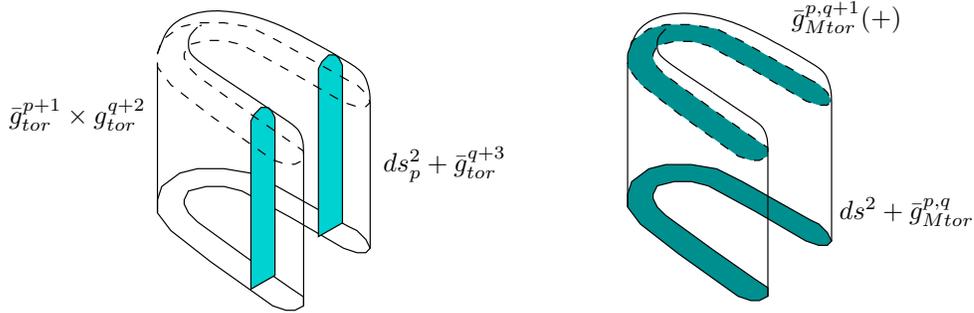
\begin{figure}[!htbp]
\vspace{-1cm}
\hspace{-4cm}
\begin{picture}(0,0)%
\includegraphics{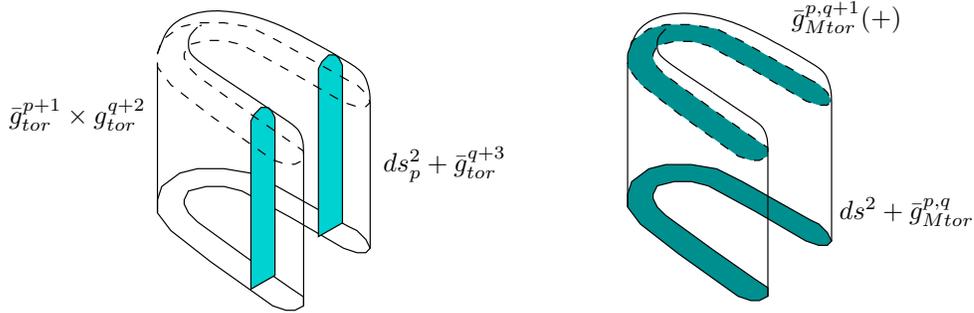}%
\end{picture}%
\setlength{\unitlength}{3947sp}%
\begingroup\makeatletter\ifx\SetFigFont\undefined%
\gdef\SetFigFont#1#2#3#4#5{%
  \reset@font\fontsize{#1}{#2pt}%
  \fontfamily{#3}\fontseries{#4}\fontshape{#5}%
  \selectfont}%
\fi\endgroup%
\begin{picture}(3079,2559)(1902,-5227)
\end{picture}%
\put(-4000, 1186){\makebox(0,0)[lb]{\smash{{\SetFigFont{10}{8}{\rmdefault}{\mddefault}{\updefault}{\color[rgb]{0,0,0}$\gtor^{p+1}\times g_{tor}^{q+2}$}%
}}}}
\put(-1650, 0886){\makebox(0,0)[lb]{\smash{{\SetFigFont{10}{8}{\rmdefault}{\mddefault}{\updefault}{\color[rgb]{0,0,0}$ds_p^{2}+\bar{g}_{tor}^{q+3}$}%
}}}}
\put(914, 1800){\makebox(0,0)[lb]{\smash{{\SetFigFont{10}{8}{\rmdefault}{\mddefault}{\updefault}{\color[rgb]{0,0,0}$\bar{g}_{Mtor}^{p,q+1}(+)$}%
}}}}
\put(1214, 0586){\makebox(0,0)[lb]{\smash{{\SetFigFont{10}{8}{\rmdefault}{\mddefault}{\updefault}{\color[rgb]{0,0,0}$ds^{2}+\bar{g}_{Mtor}^{p,q}$}%
}}}}
\caption{Alternative decompositions of the metric $\bar{\bar{g}}_{Mtor}^{p,q}$ on $D^{n+2}$}
\label{solidrelmixtorp}
\end{figure}

\begin{Remark}
For dimensional reasons it is difficult to find a suitable schematic to represent this metric. It seems best to think of a ``solid" version of the relative mixed torpedo metric from Fig. \ref{relmixtorp}. The pair of shaded vertical strips on the left picture schematically represents a product $ds_p^{2}+\bar{g}_{tor}^{q+2}$. The shaded piece at the bottom of the right picture schematically represents the the metric $\bar{g}_{Mtor}^{p,q}$.
\end{Remark}

\begin{Lemma}\label{mixcorners}
Let $n\geq 3$, $p+q+1=n$ and $p$ or $q\geq 2$. The metric $\bar{\bar{g}}_{Mtor}^{p,q}$ on $D^{n+2}$ is isotopic to the mixed torpedo metric with corners $\bar{\bar{g}}_{tor}^{n+2}$. In particular, if $q\geq 3$ the metrics $\bar{\bar{g}}_{Mtor}^{p,q}$ and $\bar{\bar{g}}_{Mtor}^{p+1,q-1}$ are isotopic.
\end{Lemma}
\begin{proof}
This is a straightforward generalisation of Lemma \ref{relmixtorplemma}.
\end{proof}

As in the case of Lemma \ref{reltoriso} (and using an almost identical argument), we obtain the following Lemma.
\begin{Lemma}\label{relativemixcorners}
Let $n\geq 3$, $p+q+1=n$ and let $q\geq 3$. Then the metrics $\bar{\bar{g}}_{Mtor}^{p,q}$ and $\bar{\bar{g}}_{Mtor}^{p+1,q-1}$ are relative isotopic, relative to the metric ${\bar{g}}_{Mtor}^{p,q-1}$.
\end{Lemma} 
\begin{proof}
The proof is almost identical to that of Lemma \ref{reltoriso}.
\end{proof}
 
\subsection{Admissible Morse Triples} \label{admreview}
Let $\{ W^{n+1};X_0,X_1\}$ be a smooth compact cobordism. Recall, this means that $W$ is a smooth compact $(n+1)$-dimensional manifold with boundary $\p W=X_0 \sqcup X_1$, the disjoint union of closed smooth $n$-manifolds $X_0$ and $X_1$. Let $\F=\F(W)$ denote the space of smooth functions $f:W\rightarrow I$ satisfying $f^{-1}(0)=X_0$ and $f^{-1}(1)=X_1$, and having no critical points near $\p {W}$. The space $\F$ is a subspace of the space of smooth functions on $W$ with its standard $C^{\infty}$ topology; see Chapter 2 of \cite{Hirsch} for the full definition. A critical point $w\in W$ of a smooth function $f:W\rightarrow I$, is a Morse critical point if, near $w$, the map $f$ is locally equivalent to the map 
\begin{equation*} \label{Morsequadratic}
\begin{split} 
 \mathbb{R}^{n+1}&\longrightarrow\mathbb{R}\\
 x&\longmapsto -\sum_{i=1}^{\lambda}{x_i}^{2}+\sum_{i=\lambda+1}^{n+1}{x_i}^{2}.
\end{split}
\end{equation*} 
The integer $\lambda$ is called the {\em Morse index} of $w$ and is an invariant of the critical point. A function $f\in\F$ is a {\em Morse function} if every critical point of $f$ is a Morse critical point. We will assume also that the set of critical points of $f$, denoted $\Sigma f$, is contained in the interior of $W$. Furthemore, we say that a Morse function $f$ is
\emph{admissible} if all of its critical points have index at most
$(n-2)$ (where $\dim W=n+1$). We denote by $\Mor(W)$ and $\Mor^{adm}(W)$, the spaces of Morse and admissible Morse functions on $W$ respectively. 

By equipping $W$ with a Riemannian metric $\m$, we can define ${\grad}_{\m}f$, the gradient vector field for $f$ with respect to $\m$. More generally, we define {\em gradient-like} vector fields on $W$ with respect to $f$ and $\m$, as follows.  
\begin{Definition}
\rm{
A {\em gradient-like} vector field with respect to $f$ and $\m$ is a vector field $V$ on $W$ satisfying the following properties.
\begin{enumerate}
\item[{\bf (1)}] $df_{x}(V_x)>0$ when $x$ is not a critical point of $f$.
\item[{\bf (2)}] Each critical point $w$ of $f$ lies in a neighbourhood $U$ so that for all $x\in U$, $V_x={\grad}_{\m}f(x)$. 
\end{enumerate}
}
\end{Definition} 

For our purposes, we impose a minor compatibility condition on the metrics we wish to work with.

\begin{Definition}\label{compat_metric}
 \rm{ Let $f$ be an admissible Morse function on $W$.  A Riemannian metric ${\mathfrak m}$ on
  $W$ is \emph{compatible} with the Morse function $f$ if for every
  critical point $z\in \Sigma f$ with index $\lambda$, the positive and
  negative eigenspaces $T_zW^+$ and $T_zW^-$ of the Hessian $d^2f$ are
  $\mathfrak m$-orthogonal, and $d^2f|_{T_zW^+}=\mathfrak
  m|_{T_zW^+}$, $d^2f|_{T_zW^-}=-\mathfrak m|_{T_zW^-}$.}
\end{Definition}
 
\begin{Definition}
\rm{
A {\em Morse triple} on a compact cobordism $\{W;X_0,X_1\}$ is a triple $(f,\m,V)$ where $f:W\rightarrow I$ is a Morse function, $\m$ is a compatible metric for $f$, and $V$ is a gradient-like vector field with respect to $f$ and $\m$. In the case when $f$ is an admissible Morse function, the triple $(f,\m,V)$ is called an {\em admissible Morse triple}. 
}
\end{Definition}

The fact that for a given Morse function $f$, the space of compatible metrics is convex, means that the space of Morse triples on $(f,\m,V)$ is homotopy equivalent to the space of Morse functions $\Mor(W)$. Similarly the space of admissible Morse triples $(f,\m, V)$ is homotopy equivalent to the space of admissible Morse functions $\Mor^{adm}(W)$. 

\begin{Remark}
 In our use of Morse triples, we will regularly use the abbreviation $f=(f,\m,V)$. This is justified by the fact in every case, the important data comes from the choice of $f$, whereas the choice of compatible metric and gradient-like vector field is arbitrary.
\end{Remark}

\subsection{A Review of the Gromov-Lawson Cobordism Theorem}

Let $(W;X_0,X_1)$ be as before and let $g_0$ be a psc-metric on $X_0$. In Part One we discussed the problem of extending the metric $g_0$ to a psc-metric $\bar{g}$ on $W$, which has a product structure near $\p W$. In particular, we proved the following theorem.

\begin{Theorem}\label{GLcob} 
Let $\{W^{n+1};X_0,X_1\}$ be a smooth compact cobordism. Suppose $g_0$ is a metric of positive scalar curvature on $X_0$ and $f=(f,\m,V)$ is an admissible Morse triple on $W$. Then there is a psc-metric $\bar{g}=\bar{g}(g_0,f)$ on $W$ which extends $g_0$ and has a product structure near the boundary.
\end{Theorem}

We call the metric $\bar{g}$, a {\em Gromov-Lawson cobordism} with respect to $g_0$ and $f$. It is worth briefly reviewing the structure of this metric. We begin with some topological observations about the admissible Morse function $f$ in the statement of the theorem. For simplicity, let us assume for now that $f$ has only a single critical point $w$ of index $p+1$. Intersecting transversely at $w$ are a pair of trajectory disks $K_{-}^{p+1}$ and $K_{+}^{q+1}$; see Fig. \ref{trajflow}. The lower disk $K_{-}^{p+1}$ is a $(p+1)$-dimensional disk which is bounded by an embedded $p$-sphere $S_{-}^{p}\subset X_0$. It consists of the union of segments of integral curves of the gradient-like vector field, beginning at the bounding sphere and ending at $w$. Similarly, $K_{+}^{q+1}$ is a $(q+1)$-dimensional disk bounded by an embedded $q$-sphere $S_{+}^{q}\subset X_1$. The bounding spheres $S_{-}^{p}$ and $S_{+}^{q}$ are known as trajectory spheres.

Let $N$ denote a small tubular neighbourhood of $S_{-}^{p}$, defined with respect to the metric $m|_{X_0}$. Consider the region $X_0\setminus N$. For each point $x\in X_0\setminus N$, there is a unique maximal integral curve of the vector field $V$, $\psi_x:[0,1]\rightarrow W$ satisfying $f\circ\psi_x(t)=t$; see section 3 of \cite{Smale} for details. This gives rise to an embedding
\begin{equation*}
\begin{split}
\psi:(X_0\setminus N)\times I&\longrightarrow W\\
(x,t)&\longmapsto (\psi_x(t)).
\end{split}
\end{equation*}
We denote by $U$, the complement of this embedding in $W$. Notice that $U$ is a sort of ``cross-shaped" region and a neighbourhood of $K_{-}^{p+1}\cup K_{+}^{q+1}$; see Fig. \ref{trajflow}. Indeed, a continuous shrinking of the radius of $N$ down to $0$ induces a deformation retract of $U$ onto $K_{-}^{p+1}\cup K_{+}^{q+1}$.

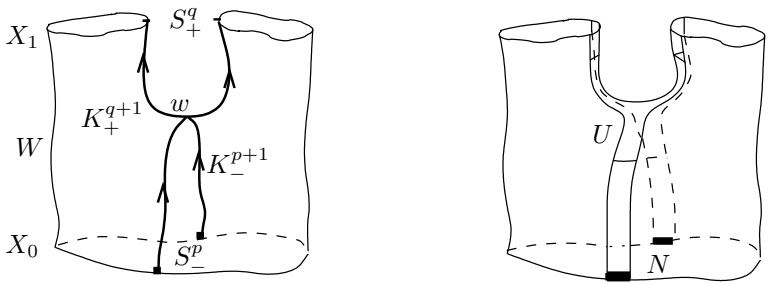
\begin{figure}[htbp]
\vspace{-2cm}
\begin{picture}(0,0)%
\includegraphics{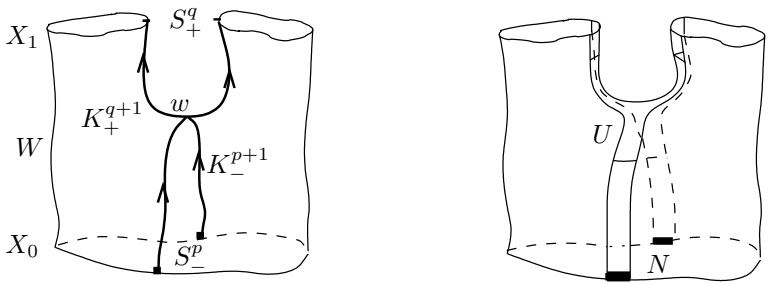}%
\end{picture}%
\setlength{\unitlength}{3947sp}%
\begingroup\makeatletter\ifx\SetFigFont\undefined%
\gdef\SetFigFont#1#2#3#4#5{%
  \reset@font\fontsize{#1}{#2pt}%
  \fontfamily{#3}\fontseries{#4}\fontshape{#5}%
  \selectfont}%
\fi\endgroup%
\begin{picture}(6381,2877)(561,-3478)
\put(1347,-2380){\makebox(0,0)[lb]{\smash{{\SetFigFont{10}{12}{\rmdefault}{\mddefault}{\updefault}{\color[rgb]{0,0,0}$w$}%
}}}}
\put(785,-2480){\makebox(0,0)[lb]{\smash{{\SetFigFont{10}{12}{\rmdefault}{\mddefault}{\updefault}{\color[rgb]{0,0,0}$K_{+}^{q+1}$}%
}}}}
\put(1575,-2768){\makebox(0,0)[lb]{\smash{{\SetFigFont{10}{12}{\rmdefault}{\mddefault}{\updefault}{\color[rgb]{0,0,0}$K_{-}^{p+1}$}%
}}}}
\put(1370,-3350){\makebox(0,0)[lb]{\smash{{\SetFigFont{10}{12}{\rmdefault}{\mddefault}{\updefault}{\color[rgb]{0,0,0}$S_{-}^{p}$}%
}}}}
\put(1342,-1861){\makebox(0,0)[lb]{\smash{{\SetFigFont{10}{12}{\rmdefault}{\mddefault}{\updefault}{\color[rgb]{0,0,0}$S_{+}^{q}$}%
}}}}
\put(4339,-3400){\makebox(0,0)[lb]{\smash{{\SetFigFont{10}{12}{\rmdefault}{\mddefault}{\updefault}{\color[rgb]{0,0,0}$N$}%
}}}}
\put(4000,-2580){\makebox(0,0)[lb]{\smash{{\SetFigFont{10}{12}{\rmdefault}{\mddefault}{\updefault}{\color[rgb]{0,0,0}$U$}%
}}}}
\put(376,-2661){\makebox(0,0)[lb]{\smash{{\SetFigFont{10}{12}{\rmdefault}{\mddefault}{\updefault}{\color[rgb]{0,0,0}$W$}%
}}}}
\put(314,-1961){\makebox(0,0)[lb]{\smash{{\SetFigFont{10}{12}{\rmdefault}{\mddefault}{\updefault}{\color[rgb]{0,0,0}$X_1$}%
}}}}
\put(314,-3274){\makebox(0,0)[lb]{\smash{{\SetFigFont{10}{12}{\rmdefault}{\mddefault}{\updefault}{\color[rgb]{0,0,0}$X_0$}%
}}}}
\end{picture}%
\caption{Trajectory disks of the critical point $w$ contained inside a disk $U$}
\label{trajflow}
\end{figure}

We now define the metric $\bar{g}$ on the region $W\setminus U$ to be simply $g_{0}|_{X\setminus N}+dt^{2}$ where the $t$ coordinate comes from the embedding $\psi$ above. Of course, the real challenge lies in extending this metric over the region $U$. Notice that the boundary of $U$ decomposes as
\begin{equation*}
\p U = (S^{p}\times D^{q+1})\cup (S^{p}\times S^{q}\times I)\cup (D^{p+1}\times S^{q}).
\end{equation*}
The $S^{p}\times D^{q+1}$ part of this decomposition is of course the tubular neighbourhood $N$, while the $D^{p+1}\times S^{q}$ piece is a tubular neighbourhood of the outward trajectory sphere $S_{+}^{q}\subset X_1$. Without loss of generality, assume that $f(w)=\frac{1}{2}$. Let $c_0$ and $c_1$ be constants satisfying $0<c_0<\frac{1}{2}<c_{1}<1$.
The level sets $f=c_0$ and $f=c_1$ divide
$U$ into three regions: 
$$
\begin{array}{rcl}
U_0& = & f^{-1}([0,c_1])\cap U,
\\
\\
U_{w}& = & f^{-1}([c_0, c_1])\cap U,  
\\
\\
U_1 &= & f^{-1}([c_1, 1])\cap U.
\end{array} 
$$

The region $U_0$ can be diffeomorphically identified with $N\times [0,c_0]$ in exactly the way we identified $W\setminus U$ with $X_0\setminus N\times I$. Thus, on $U_0$, we define $\bar{g}$ as simply the product $g_{0}|_{N}+dt^{2}$. Indeed, we can extend this metric $g_{0}|_{N}+dt^{2}$ near the $S^{p}\times S^{q}\times I$ part of the boundary also where, again, $t$ is the trajectory coordinate. Inside $U_w$, which is identified with a cross-shaped region inside the disk product $D^{p+1}\times D^{q+1}$, the metric smoothly transitions to a standard product $\gtor^{p+1}(\epsilon)+g_{tor}^{q+1}(\delta)$ for some appropriately chosen $\epsilon,\delta>0$. This is done so that the induced metric on the level set $f^{-1}(c_1)$, denoted $g_1$, is precisely the metric obtained by application of the Gromov-Lawson construction on $g_0$. Furthermore, near $f^{-1}(c_1)$, $\bar{g}=g_1+dt^{2}$. Finally, on $U_{1}$, which is identified with $D^{p+1}\times S^{q}\times [c_1,1]$ in the usual manner, the metric $\bar{g}$ is simply the product $g_1+dt^{2}$. See Fig. \ref{morsecoordblock} for an illustration.
 
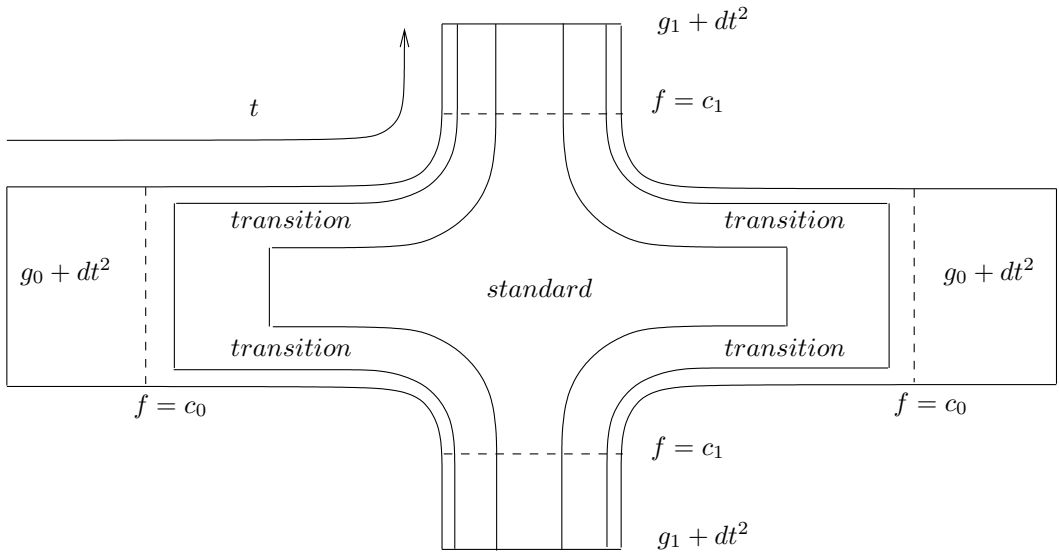
\begin{figure}[htbp]
\begin{picture}(0,0)%
\includegraphics[scale=0.9]{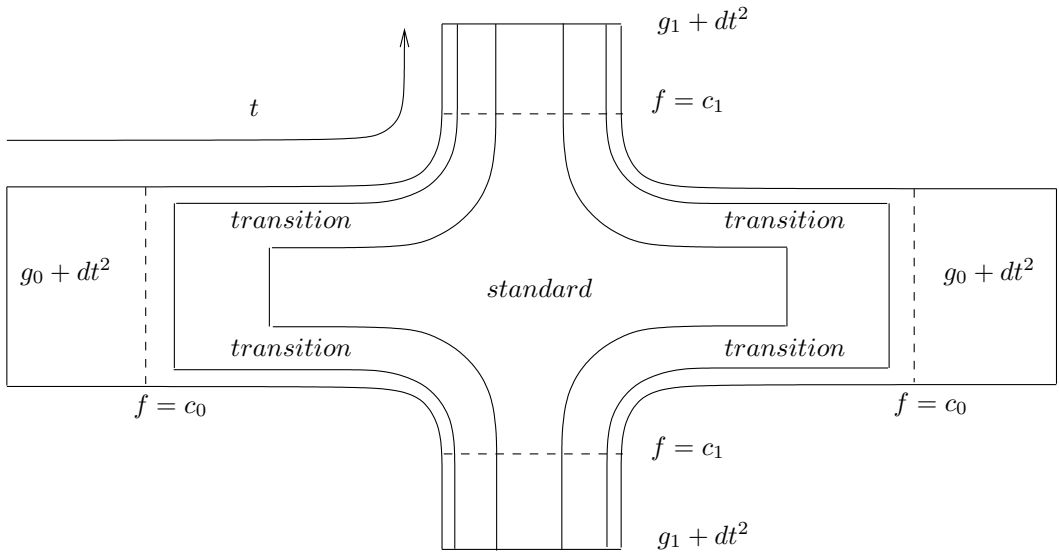}%
\end{picture}%
\setlength{\unitlength}{3947sp}%
\begingroup\makeatletter\ifx\SetFigFont\undefined%
\gdef\SetFigFont#1#2#3#4#5{%
  \reset@font\fontsize{#1}{#2pt}%
  \fontfamily{#3}\fontseries{#4}\fontshape{#5}%
  \selectfont}%
\fi\endgroup%
\begin{picture}(7349,3720)(995,-3850)
\put(4019,-2234){\makebox(0,0)[lb]{\smash{{\SetFigFont{10}{14.4}{\rmdefault}{\mddefault}{\updefault}{\color[rgb]{0,0,0}$standard$}%
}}}}
\put(2532,-1096){\makebox(0,0)[lb]{\smash{{\SetFigFont{10}{14.4}{\rmdefault}{\mddefault}{\updefault}{\color[rgb]{0,0,0}$t$}%
}}}}
\put(5094,-3784){\makebox(0,0)[lb]{\smash{{\SetFigFont{10}{14.4}{\rmdefault}{\mddefault}{\updefault}{\color[rgb]{0,0,0}$g_1+dt^{2}$}%
}}}}
\put(5094,-546){\makebox(0,0)[lb]{\smash{{\SetFigFont{10}{12}{\rmdefault}{\mddefault}{\updefault}{\color[rgb]{0,0,0}$g_1+dt^{2}$}%
}}}}
\put(2414,-1800){\makebox(0,0)[lb]{\smash{{\SetFigFont{10}{12}{\rmdefault}{\mddefault}{\updefault}{\color[rgb]{0,0,0}$transition$}%
}}}}
\put(2414,-2604){\makebox(0,0)[lb]{\smash{{\SetFigFont{10}{12}{\rmdefault}{\mddefault}{\updefault}{\color[rgb]{0,0,0}$transition$}%
}}}}
\put(5514,-1800){\makebox(0,0)[lb]{\smash{{\SetFigFont{10}{12}{\rmdefault}{\mddefault}{\updefault}{\color[rgb]{0,0,0}$transition$}%
}}}}
\put(5514,-2604){\makebox(0,0)[lb]{\smash{{\SetFigFont{10}{12}{\rmdefault}{\mddefault}{\updefault}{\color[rgb]{0,0,0}$transition$}%
}}}}
\put(1089,-2124){\makebox(0,0)[lb]{\smash{{\SetFigFont{10}{12}{\rmdefault}{\mddefault}{\updefault}{\color[rgb]{0,0,0}$g_0+dt^{2}$}%
}}}}
\put(6889,-2124){\makebox(0,0)[lb]{\smash{{\SetFigFont{10}{12}{\rmdefault}{\mddefault}{\updefault}{\color[rgb]{0,0,0}$g_0+dt^{2}$}%
}}}}
\put(5051,-3224){\makebox(0,0)[lb]{\smash{{\SetFigFont{10}{12}{\rmdefault}{\mddefault}{\updefault}{\color[rgb]{0,0,0}$f=c_1$}%
}}}}
\put(5051,-1049){\makebox(0,0)[lb]{\smash{{\SetFigFont{10}{12}{\rmdefault}{\mddefault}{\updefault}{\color[rgb]{0,0,0}$f=c_1$}%
}}}}
\put(1801,-2961){\makebox(0,0)[lb]{\smash{{\SetFigFont{10}{12}{\rmdefault}{\mddefault}{\updefault}{\color[rgb]{0,0,0}$f=c_0$}%
}}}}
\put(6576,-2949){\makebox(0,0)[lb]{\smash{{\SetFigFont{10}{12}{\rmdefault}{\mddefault}{\updefault}{\color[rgb]{0,0,0}$f=c_0$}%
}}}}
\end{picture}%
\caption{The metric $\bar{g}$ on the disk $U$}
\label{morsecoordblock}
\end{figure}
 
We should point out that this construction can be carried out for a tubular neighbourhood $N$ of arbitrarily small radius and for $c_0$ and $c_1$ chosen arbitrarily close to $\frac{1}{2}$. Thus, the region $U_w$, on which the metric $\bar{g}$ is not simply a product and is undergoing some kind of transition, can be made arbitrarily small with respect to the background metric $\m$. As critical points of a Morse function are isolated, it follows that this construction generalises easily to Morse functions with more than one critical point.

\subsection{Equivariance of the Gromov-Lawson Construction}

We now make an important observation, with regard to the above construction. We mentioned earlier that the region $U_w$ is identified with a cross-shaped region inside the disk product $D^{p+1}\times D^{q+1}$. There is of course an obvious action of the group $O(p+1)\times O(q+1)$ on $D^{p+1}\times D^{q+1}$. The fact that the above construction is equivariant with respect to this action is the subject of the following lemma. This fact will be important later on, when we generalise this construction over bundles of fibrewise admissible Morse functions.

\begin{Lemma}\label{GLequiv}
The construction of a Gromov-Lawson cobordism is equivariant with respect to the action of $O(p+1)\times O(q+1)$.
\end{Lemma}
\begin{proof}

The original Gromov-Lawson construction shows that if $g$ is a psc-metric on a manifold $X$ and $S^{p}\times D^{q+1}\subset X$ is an embedding with $q\geq 2$, then $g$ can be replaced by a psc-metric $g'$ which is standard near the embedded sphere $\Sp$ and the original metric $g$ away from $\Sp$. By standard, we mean that near the embedded sphere $\Sp$, the resulting metric takes the form $ds_p+g_{tor}^{q+1}(\delta)$ for some small $\delta>0$. In Part One, we describe this construction in detail. We furthermore show that the metrics $g$ and $g'$ are isotopic through an isotopy $g_s, s\in I$, which fixes the metric away from $\Sp$; see Theorem 3.2  in Part One. It is this isotopy that is used to create a psc-metric on a region diffeomorphic to $\Sp \times \Dq \times I$ which is the orginal metric $g+dt^2$ near $\Sp \times \Dq \times \{0\}\cup \Sp \times \p\Dq \times I$ and which is $\epsilon^2 ds_p+g_{tor}^{q+1}(\delta)+dt^2$ near $\Sp \times \{0\} \times \{1\}$. This is done by appropriately rescaling the isotopy $g_s, s\in I$, to obtain a warped product metric $g_t+dt^2$ with the required properties. 

To prove the lemma, we must show that each metric $g_s$ in this isotopy has been constructed $O(p+1)\times O(q+1)$-equivariantly. In otherwords, the metric $g_s$ is the same if we first act on $(\Sp \times \Dq \times \{0\},g)$ by $O(p+1)\times O(q+1)$, next perform the isotopy along $[0,s]$ and finally undo the orginal action. To see that this is the case we must review the construction of the isotopy $g_s, s \in I$. This construction consists of a number of steps.

The first stage in the isotopy is the so-called ``bending argument''. The metric $g$ is altered on fibres $\Dq$ but not in the $\Sp$ directions. The alteration involves smoothly pushing out geodesic spheres on the fibres $\Dq$ to form a hypersurface in $X\times \mathbb{R}$ and then replacing $g$ with the metric induced on this hypersurface. On each point of $\Sp$, the fibre $\Dq$ is altered by pushing out each geodesic sphere $\Sq(r)$, a fixed distance determined by its radius $r$. As this adjustment takes place on fibres only in the radial direction, equivariance on the $O(q+1)$ factor is guaranteed. Furthermore, the ``pushing out'' process is the same for every point of $\Sp$, i.e. a single ``push-out" curve is chosen to determine this process on all fibres. This guarantees $O(p+1)$-equivariance. 
By homotoping through appropriate ``push-out" curves we construct the first part of the isotopy.

We now come to the second stage of the construction. The resulting metric induced on the fibres is, near $á¹¢p$, close to the torpedo metric $\gtor^{q+1}(\delta)$. This follows from the fact that the push-out cuve ends as the graph of the torpedo function $f_\delta$ defined earlier. Provided $\delta$ is small enough, a straightforward linear homotopy on the fibres results in a metric which, near $S^p$, is a Riemannain submersion with base metric $g|_{\Sp}$ and fibre metric $\gtor^{q+1}(\delta)$. The fact that this process is identical on each fibre guarantees $O(p+1)$-equivariance. We now concentrate on the fibre $\Dq$. At this stage, the metric on the fibre takes the form $dr^2+g|_{\Sq(f_\delta(r))}$. We wish to perform a linear homotopy of this metric, near the centre of the disk, to one which is the standard metric $\gtor^{q+1}(\delta)$ while fixing the original metric away from the centre of the disk. This is possible because the $O(q+1)$-symmetry of the metric $\gtor^{q+1}(\delta)$ guarantees that this process only varies in the radial direction.

The required isotopy on the fibre takes the form
\begin{equation}\label{orig}
 t[dr^2+(1-\tau(r))f_\delta(r)^2 ds_q^2+\tau(r)g|_{\Sq(f_\delta(r))}]+(1-t)[dr^2+g|_{\Sq(f_\delta(r))}],
\end{equation}
\noindent where $t\in I$, $\tau:[0,\infty)\rightarrow[0,1]$ is an appropriately chosen smooth cut-off function and $r$ is the radial distance from the centre of the disc. Now suppose $\psi$ is an element of $O(q+1)$. Applying $\psi$ to the original metric $dr^2+g|_{\Sq(f_\delta(r))}$ and then applying the isotopy gives us the following expression
\begin{equation*}
 t[dr^2+(1-\tau(r))f_\delta(r)^2 ds_q^2+\tau(r)\psi.(g|_{\Sq(f_\delta(r))})]+(1-t)[dr^2+\psi.g|_{\Sq(f_\delta(r))}].
\end{equation*}
\noindent Finally, applying $\psi^{-1}$, we obtain
\begin{equation*}
 t[dr^2+(1-\tau(r))f_\delta(r)^2 \psi^{-1}.ds_q^2+\tau(r)\psi^{-1}.\psi.(g|_{\Sq(f_\delta(r))})]+(1-t)[dr^2+\psi^{-1}.\psi.g|_{\Sq(f_\delta(r))}],
\end{equation*}
\noindent which is precisely expression \ref{orig}, as the metric $ds_q^2$ is $O(q+1)$-symmetric.

At this point, the original metric has been isotoped to one which, near the embedded surgery sphere $\Sp$, takes the form of a Riemannian submersion with base metric $g|_{\Sp}$ and fibre metric $\gtor^{q+1}(\delta)$. There are two remaining tasks. The first involves a linear homotopy of the base metric to the standard round sphere metric $\epsilon^2 ds_p^2$. Again, the symmetry of the round sphere metric means that an almost identical argument to the one above proves the required equivariance. The final task is to isotopy this metric to a standard product metric. This is achieved by a linear homotopy of the horizontal distribution near the embedded $\Sp$ to one which is flat, with appropriate an smoothing off to fix the original distribution away from $\Sp$. The smoothing part is not a problem, as it happens in the radial direction and the same cut-off function is used on every fibre. The linear homotopy itself is less obvious and needs to be analysed.

We will denote by $\H$ and $\H^{flat}$, the respective distributions. The distribution $\H$ associates to every point $x$ of $\Sp \times \Dq$, a subspace $H_x\subset T_x(\Sp \times \Dq)$. In the case of $\H^{flat}$, the corresponding subspace $H_x^{flat}$, is precisely $T_x\Sp\times\{0\}\subset T_x(\Sp \times \Dq)$. Now suppose $\psi\in O(p+1)\times O(q+1)$ and that $\psi(x)=y$ for some $(x,y)\in \Sp \times \Dq$. We need to show that

\begin{equation*}
(1-t)H_x + tH_x^{flat}=\psi_{*}^{-1}[(1-t)\psi_{*}H_x+tH_y^{flat}].
\end{equation*}

\noindent Simplifying the expression on the right yields
\begin{equation*}
\begin{split}
(1-t)H_x+t\psi_{*}^{-1}H_y^{flat}&=(1-t)H_x+t\psi_{*}^{-1}(T_y\Sp\times\{0\})\\
&=(1-t)H_x + tH_x^{flat},
\end{split}
\end{equation*}
\noindent as $\psi_{*}^{-1}$ maps elements of $T_y\Sp\times\{0\}$ into $T_x\Sp\times\{0\}$. This completes the proof.
\end{proof}

\subsection{Continuous Families of Gromov-Lawson Cobordisms}\label{ctsmorsefam} 
\label{morsefamily}
A careful analysis of the Gromov-Lawson construction shows that it can be applied continuously over a compact family of metrics as well as a compact family of embedded surgery spheres; see Theorem 3.10 in Part One. It then follows that the construction of Theorem \ref{GLcob} can be applied continuously over certain compact families of admissible Morse functions to obtain Theorem \ref{GLcobordismcompact}. Before stating it, we introduce some notation. Let $\mathcal{B}=\{g_b\in\Riem^{+}(X_0):b\in B\}$ be a compact continuous family of psc-metrics on $X_0$, parametrised by a compact space $B$. Let $\mathcal{C}=\{f_c\in{\Mor}^{adm}(W):c\in D^{k}\}$ be a smooth compact family of admissible Morse functions on $W$, parametrised by the disk $D^{k}$.

\begin{Theorem}{\rm [Theorem 0.5 in Part One]}\label{GLcobordismcompact}
There is a continuous map 
\begin{equation*}
\begin{split}
\mathcal{B}\times \mathcal{C}&\longrightarrow \Riem^{+}(W)\\
(g_b,f_c)&\longmapsto \bar{g}_{b,c}=\bar{g}(g_b, f_c)
\end{split}
\end{equation*}
\noindent so that for each pair $(b,c)$, the metric $\bar{g}_{b,c}$ is a Gromov-Lawson cobordism.
\end{Theorem}

The proof of Theorem \ref{GLcobordismcompact} relies on two important facts. Firstly, each Morse function in the family $\mathcal{C}$ has the same number of critical points of the same index. Secondly, the fact that the family of Morse functions is parametrised by a contractible space means that as an individual critical point varies over the family, a single choice of Morse coordinates may be chosen to vary with it. In other words, a global choice of Morse coordinates is possible. 

The next stage is to consider ``twisted" families of Morse functions, which are necessarily parameterised by a non-contractible space. This is done in \cite{BHSW}, although, as our goal is a generalisation of this result, we provide a summary below. We adopt the notion of a {\em family of Morse functions}, discussed in the introduction. As before, let $W^{n+1}$ be a smooth compact manifold with $\p W=X_0\sqcup X_1$, a disjoint union of smooth closed $n$-manifolds. We denote by $\Diff(W;X_0,X_1)$, the group of diffeomorphisms of $W$ whose restriction to $\p W$ maps each $X_i$ diffeomorphically to $X_i$, for $i=0,1$. Let $E^{n+k+1}$ and $B^{k}$ be a pair of smooth compact manifolds of dimension $n+1+k$ and $k$ respectively. The manifolds $E$ and $B$ form part of a smooth fibre bundle with fibre $W$, arising from a submersion $\pi:E\rightarrow B$. The structure group of this bundle is assumed to be $\Diff(W;X_0,X_1)$. We will assume also that the boundary of $E$, $\p E$, consists of a pair of disjoint smooth submanifolds $E_0$ and $E_1$. The restriction of $\pi$ to these submanifolds is denoted $\pi_0$ and $\pi_1$ respectively. These maps are also submersions onto $B$ and give rise to a pair of smooth subbundles with respective fibres $X_0, X_1\subset W$ and respective structure groups $\Diff(X_0)$ and $\Diff(X_1)$. All of this gives rise to the commmutative diagram represented in Fig. \ref{subbundles}. 

\begin{figure}[htbp]
\hspace{30mm}
\begin{picture}(0,0)%
\includegraphics{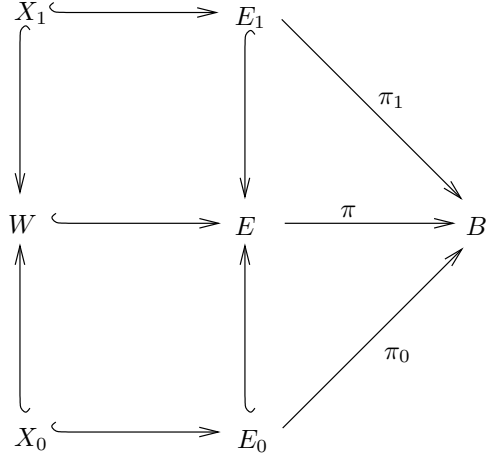}%
\end{picture}%
\setlength{\unitlength}{3947sp}%
\begingroup\makeatletter\ifx\SetFigFont\undefined%
\gdef\SetFigFont#1#2#3#4#5{%
  \reset@font\fontsize{#1}{#2pt}%
  \fontfamily{#3}\fontseries{#4}\fontshape{#5}%
  \selectfont}%
\fi\endgroup%
\begin{picture}(2905,2907)(2261,-3802)
\put(3701,-2399){\makebox(0,0)[lb]{\smash{{\SetFigFont{10}{12}{\rmdefault}{\mddefault}{\updefault}{\color[rgb]{0,0,0}$E$}%
}}}}
\put(2276,-2387){\makebox(0,0)[lb]{\smash{{\SetFigFont{10}{12}{\rmdefault}{\mddefault}{\updefault}{\color[rgb]{0,0,0}$W$}%
}}}}
\put(5151,-2399){\makebox(0,0)[lb]{\smash{{\SetFigFont{10}{12}{\rmdefault}{\mddefault}{\updefault}{\color[rgb]{0,0,0}$B$}%
}}}}
\put(2314,-3724){\makebox(0,0)[lb]{\smash{{\SetFigFont{10}{12}{\rmdefault}{\mddefault}{\updefault}{\color[rgb]{0,0,0}$X_0$}%
}}}}
\put(2314,-1049){\makebox(0,0)[lb]{\smash{{\SetFigFont{10}{12}{\rmdefault}{\mddefault}{\updefault}{\color[rgb]{0,0,0}$X_1$}%
}}}}
\put(3701,-1074){\makebox(0,0)[lb]{\smash{{\SetFigFont{10}{12}{\rmdefault}{\mddefault}{\updefault}{\color[rgb]{0,0,0}$E_1$}%
}}}}
\put(3714,-3736){\makebox(0,0)[lb]{\smash{{\SetFigFont{10}{12}{\rmdefault}{\mddefault}{\updefault}{\color[rgb]{0,0,0}$E_0$}%
}}}}
\put(4364,-2286){\makebox(0,0)[lb]{\smash{{\SetFigFont{10}{12}{\rmdefault}{\mddefault}{\updefault}{\color[rgb]{0,0,0}$\pi$}%
}}}}
\put(4601,-1574){\makebox(0,0)[lb]{\smash{{\SetFigFont{10}{12}{\rmdefault}{\mddefault}{\updefault}{\color[rgb]{0,0,0}$\pi_1$}%
}}}}
\put(4639,-3174){\makebox(0,0)[lb]{\smash{{\SetFigFont{10}{12}{\rmdefault}{\mddefault}{\updefault}{\color[rgb]{0,0,0}$\pi_0$}%
}}}}
\end{picture}%
\caption{The smooth fibre bundle $\pi$ and subbundles $\pi_i$ where $i=0,1$.}
\label{subbundles}
\end{figure}
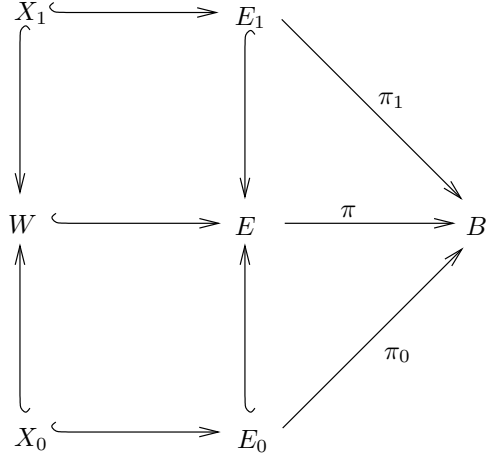

We now equip the bundle $\pi:E\rightarrow B$ with the structure of Riemannian submersion. For each $y\in B$, we denote by $W_y$, the fibre $\pi^{-1}(y)$. The union of tangent bundles $TW_y$, to $W_y$ over $y\in B$, forms a smooth subbundle of $TE$, the tangent bundle to $E$. This subbundle is denoted $\Ver$. Choose a horizontal distribution $\H_E$ for the submersion $\pi$. Now equip the base manifold $B$ with some Riemannian metric $\m_B$ and let $\m_y, y\in B$ be a smooth family of metrics on $W$. From chapter 9 of \cite{B}, we know that this gives rise to a unique submersion metric $\m_E$ on $E$ giving us a Riemannian submersion $\pi:(E, \m_E)\rightarrow (B, \m_B)$. Shortly, we will add some further restrictions on the type of submersion we wish to deal with. Before this we need a way of describing a fibrewise Morse function on $E$.

\begin{Definition}
\label{admissiblemap}
{\rm
A smooth map $F:E\rightarrow B\times I$ is said to be a {\em family of admissible Morse functions} if it satisfies the following conditions.
\begin{enumerate}
\item[{\bf (i)}] For each $w\in E$, $\pi(w)=p_1\circ F(w)$.
\item[{\bf (ii)}] The pre-images $F^{-1}(B\times\{0\})$ and $F^{-1}(B\times\{1\})$ are the submanifolds $E_0$ and $E_1$ respectively.
\item[{\bf (iii)}] The singular set $\Sigma F$ is contained entirely in $E\setminus(E_0\sqcup E_1)$.
\item[{\bf (iv)}] For each $y\in B$, the restriction $f_y=F|_{W_y}$ is an admissible Morse function, i.e. one whose Morse critical points have index $\leq n-2$.
\end{enumerate}
}
\end{Definition}
\noindent This is shown schematically in Fig. \ref{fibrewisemorse}, where we reproduce from the introduction, a family which restricts on fibres to a Morse function with two critical points. 

\begin{figure}[htbp]
\vspace{-2cm}
\begin{picture}(0,0)%
\includegraphics{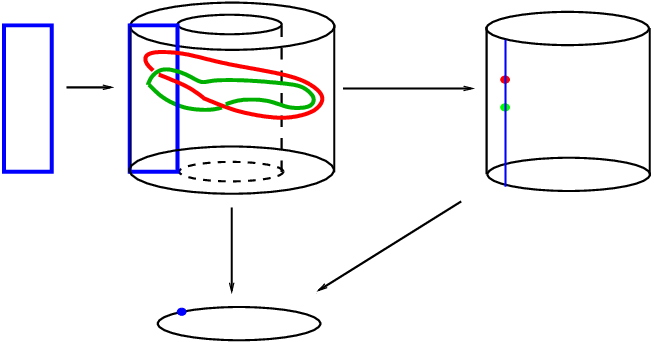}%
\end{picture}%
\setlength{\unitlength}{3947sp}%
\begingroup\makeatletter\ifx\SetFigFont\undefined%
\gdef\SetFigFont#1#2#3#4#5{%
  \reset@font\fontsize{#1}{#2pt}%
  \fontfamily{#3}\fontseries{#4}\fontshape{#5}%
  \selectfont}%
\fi\endgroup%
\begin{picture}(4536,4099)(1496,-5186)
\put(1576,-3249){\makebox(0,0)[lb]{\smash{{\SetFigFont{10}{12}{\rmdefault}{\mddefault}{\updefault}{\color[rgb]{0,0,0}{$W$}}%
}}}}
\put(1276,-2649){\makebox(0,0)[lb]{\smash{{\SetFigFont{10}{12}{\rmdefault}{\mddefault}{\updefault}{\color[rgb]{0,0,0}{$X_1$}}%
}}}}
\put(1276,-3949){\makebox(0,0)[lb]{\smash{{\SetFigFont{10}{12}{\rmdefault}{\mddefault}{\updefault}{\color[rgb]{0,0,0}{$X_0$}}%
}}}}

\put(4676,-3049){\makebox(0,0)[lb]{\smash{{\SetFigFont{10}{12}{\rmdefault}{\mddefault}{\updefault}{\color[rgb]{0,0,0}{$F$}}%
}}}}
\put(6576,-2449){\makebox(0,0)[lb]{\smash{{\SetFigFont{10}{12}{\rmdefault}{\mddefault}{\updefault}{\color[rgb]{0,0,0}{$B\times I$}}%
}}}}

\put(2236,-2649){\makebox(0,0)[lb]{\smash{{\SetFigFont{10}{12}{\rmdefault}{\mddefault}{\updefault}{\color[rgb]{0,0,0}{$W_y$}}%
}}}}

\put(3876,-2449){\makebox(0,0)[lb]{\smash{{\SetFigFont{10}{12}{\rmdefault}{\mddefault}{\updefault}{\color[rgb]{0,0,0}{$E$}}%
}}}}
\put(2751,-4899){\makebox(0,0)[lb]{\smash{{\SetFigFont{10}{12}{\rmdefault}{\mddefault}{\updefault}{\color[rgb]{0,0,0}$y$}%
}}}}
\put(3406,-4486){\makebox(0,0)[lb]{\smash{{\SetFigFont{10}{12}{\rmdefault}{\mddefault}{\updefault}{\color[rgb]{0,0,0}$\pi$}%
}}}}
\put(4156,-5026){\makebox(0,0)[lb]{\smash{{\SetFigFont{10}{12}{\rmdefault}{\mddefault}{\updefault}{\color[rgb]{0,0,0}$B$}%
}}}}
\put(4756,-4486){\makebox(0,0)[lb]{\smash{{\SetFigFont{10}{12}{\rmdefault}{\mddefault}{\updefault}{\color[rgb]{0,0,0}$p_1$}%
}}}}
\end{picture}%
\caption{A family of admissible Morse functions with two folds}
\label{fibrewisemorse}
\end{figure}

The critical set of $F$ is a union of path components, each consisting of Morse singularities of the respective fibrewise restrictions of $F$. Each such path component is known a {\em fold} and the critical points as {\em fold singularities} of $F$. Near any fold singularity, $F$ is equivalent to the map
\begin{equation}\label{s-fold}
\begin{split}
\mathbb{R}^{k}\times\mathbb{R}^{n+1}&\longrightarrow\mathbb{R}^{k}\times\mathbb{R}\\
(y,x)&\longmapsto\left(y, -\sum_{i=1}^{s}{x_i}^{2}+\sum_{i=s+1}^{n-k+1}{x_i}^2\right),
\end{split}
\end{equation}
for some $s\in\{0,1,\ldots,n-2\}$. The index $s$ will be consistent throughout any particular fold of $F$ and so such a fold may be regarded as an {\em $s$-fold}. Regions parametrised by the $\mathbb{R}^{k}$ factor are of course mapped diffeomorphically onto their images in $B$, by $\pi$.

As before, it will be important to have a background metric, in order to define notions such as gradient flow. Generalising our earlier notion of a Riemannian metric which is compatible with a Morse functions $f:W\rightarrow I$, we obtain the following definition.

\begin{Definition}
{\rm
A Riemmanian metric $\m_E$ on the manifold $E$ is said to be {\em compatible} with the admissible map $F$ if the restriction of $\m_E$ to fibres $W_y, y\in B$ is compatible with the function $f_y:W_y\rightarrow I$ on fold singularities. }
\end{Definition}

\begin{Proposition}
  Let $\pi: E\to B$ be a smooth bundle as above and $F: E \to B\times
  I$ be an admissible map with respect to
  $\pi$. Then the bundle $\pi: E\to B$ admits the structure of a
  Riemannian submersion $\pi: (E,{\mathfrak m}_{E})\to (B,{\mathfrak
    m}_{B})$ such that the metric ${\mathfrak m}_{E}$ is compatible
  with the map $F: E \to B\times I$.
\end{Proposition}
\begin{proof}
On each fibre $\pi^{-1}(y), y\in B$, there is a metric $\m_y$ which is compatible with $f_y$. The local triviality condition on folds along with the fact that the space of admissible Morse functions on any fibre is convex, means that such a choice may be made continuously. Now choose a Riemannian metric $\m_B$ on the base B and an integrable horizontal distribution $\H_E$ for the submersion $\pi$. By chapter 9 of \cite{B}, this gives rise to a unique submersion metric with the desired properties. 
\end{proof}

 We now recall Theorem 2.12 of \cite{BHSW}, stated in the introduction, as Theorem \ref{main}. 
\vspace{3mm}

\noindent {\bf Theorem \ref{main}.}
{\em Let $\pi: E\to B$ be a bundle of smooth compact manifolds, where the fibre
  $W$ is a compact manifold with boundary $ \p W= X_0\sqcup X_1$ and the
  structure group is $\Diff(W;X_0,X_1)$.
  Let $F: E\to B\times I$ be an admissible
  family of Morse functions, with respect to $\pi$. In addition, we
  assume that the fibre bundle $\pi : E\to B$ is given the structure
  of a Riemannian submersion $\pi: (E,{\mathfrak m}_{E})\to
  (B,{\mathfrak m}_{B})$ such that the metric ${\mathfrak m}_{E}$ is
  compatible with the map $F: E \to B\times I$.  Finally, let $g_0 : B \to \Riem^+(X_0)$ be a smooth map. 

  Then there exists a metric $\bar{g}=\bar{g}(g_0,F,{\mathfrak m}_{E})$
  such that for each $y\in B$ the restriction $\bar{g}(y)=\bar{g}|_{W_y}$ on the
  fibre $W_y=\pi^{-1}(y)$ satisfies the following properties:
\begin{enumerate}
\item[{\bf (1)}] $\bar{g}(y)$ extends $g_0(y)$;
\item[{\bf (2)}] $\bar{g}(y)$ is a product metric $g_{i}(y)+dt^2$
  near $X_{i}\subset \p W_y$, $i=0,1$;
\item[{\bf (3)}] $\bar{g}(y)$ has positive scalar curvature
  on $W_y$.
\end{enumerate}
} 
\vspace{3mm}

A complete proof is provided in \cite{BHSW}, so we will give only a brief outline.
The background Riemannian metric $\m_E$ on $E$ gives a reduction of the structure group on $\Ver$ to $O(n+1)$. There is a further reduction of this structure group on folds of $F$. Suppose $\Sigma_0 \subset \Sigma F$ is a fold of $F$. In other words, near any point in $\Sigma_0$, $F$ is locally equivalent to the map (\ref{s-fold}). The fold $\Sigma_0$ is thus a smooth $k$-dimensional submanifold of $E$, and each point $w\in \Sigma_0$ is an index $s$ Morse singularity of the function $f_{\pi(w)}$. In keeping with our earlier notation, we will assume that $s=p+1$ and that $p+q+1=n$. Associated to each tangent space $\Ver_w=T_w{W_\pi(w)}$ of $w\in \Sigma_0$, is an orthogonal splitting (with respect to $\m_E$) of the tangent space into positive and negative eigenspaces of the Hessian $d^{2}f_{\pi(w)}$ at $w$. We denote these spaces $\Ver_{w}^{+}$ and $\Ver_{w}^{-}$ and the corresponding positive and negative eigen-subbundles of $\Ver$ by $\Ver^{+}$ and $\Ver^{-}$. They have respective dimensions $p+1$ and $q+1$ and give the restriction to the fold $\Sigma_0$, of $\Ver=\Ver^{-}\oplus \Ver^{+}$, the structure of an $O(p+1)\times O(q+1)$-bundle.

Roughly speaking, the entire construction goes through in such a way that, restricted to any fibre, it is the construction of Theorem \ref{GLcob}. Using the map $g_0$ and the horizontal distribution, the boundary $E_0$ may be equipped with a submersion metric. This metric is then extended over the rest of $E$ using the gradient flow of $F$, as before.  Away from folds, this is obvious. Near folds, we need a family version of the old construction. The main difficulty is that, near folds, non-triviality of the bundle $\Ver$ means that a global choice of Morse coordinates is not possible. In other words, it is not possible to choose globally, a diffeomorphism from a neighbourhood of the critical point to $D^{p+1}\times D^{q+1}$. 

This problem is equivalent, via the fibrewise exponential map, to the problem of choosing global orthonormal frames for the positive and negative eigen-bundles  $\Ver^{+}$ and $\Ver^{-}$.
The solution to this problem lies in the observation made in Lemma \ref{GLequiv}, that the Gromov-Lawson construction is equivariant with respect to the action of $O(p+1)\times O(q+1)$. Thus, globally choosing Morse coordinates is actually unnecessary as the splitting data is all that is required. 

To take advantage of this symmetry in the Gromov-Lawson construction, the authors use the fibrewise exponential map to pull back to a construction on the bundle $\Ver$. There is one technical difficulty here which is worth mentioning. The original construction involves adjusting the metric near a critical point $w$, on a ``cross-shaped neighbourhood" determined by the trajectory disks $D_{w}^{p+1}$ and $D_{w}^{q+1}$. Although these disks agree infinitesimally with the eigenplanes $\Ver_{w}^{-}$ and  $\Ver_{w}^{+}$ of the tangent space $\Ver_{w}$, their images under the exponential map do not line up as we would wish; see Fig. \ref{perturb}.

 \begin{figure}[htbp]
 \vspace{-2.7cm}
\begin{picture}(0,0)%
\includegraphics[scale=0.7]{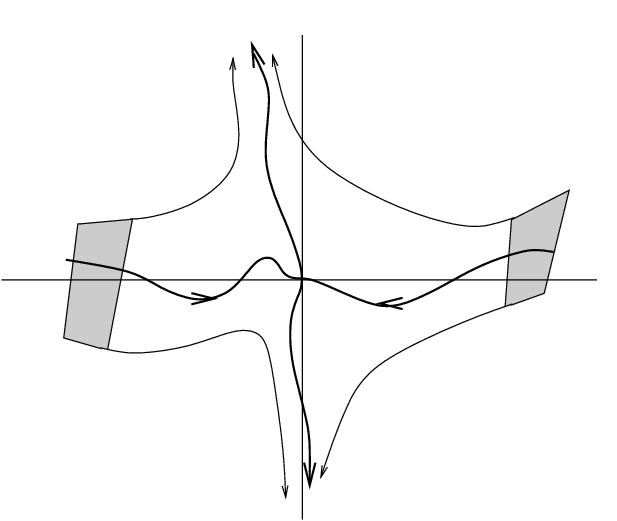}
\end{picture}%
\setlength{\unitlength}{3947sp}%
\begingroup\makeatletter\ifx\SetFigFont\undefined%
\gdef\SetFigFont#1#2#3#4#5{%
  \reset@font\fontsize{#1}{#2pt}%
  \fontfamily{#3}\fontseries{#4}\fontshape{#5}%
  \selectfont}%
\fi\endgroup%
\begin{picture}(4921,4157)(1995,-6514)
  \put(3100,-3771){\makebox(0,0)[lb]{\smash{{\SetFigFont{10}{12}{\rmdefault}{\mddefault}{\updefault}{\color[rgb]{0,0,0}$D_{w}^{q+1}$}%
        }}}}
  \put(2726,-5486){\makebox(0,0)[lb]{\smash{{\SetFigFont{10}{12}{\rmdefault}{\mddefault}{\updefault}{\color[rgb]{0,0,0}$D_{w}^{p+1}$}%
        }}}}
  \put(3700,-3700){\makebox(0,0)[lb]{\smash{{\SetFigFont{10}{12}{\rmdefault}{\mddefault}{\updefault}{\color[rgb]{0,0,0}$D\Ver_{w}^{+}$}%
        }}}}
  \put(5351,-5214){\makebox(0,0)[lb]{\smash{{\SetFigFont{10}{12}{\rmdefault}{\mddefault}{\updefault}{\color[rgb]{0,0,0}$D{\Ver}_{w}^{-}$}%
        }}}}
\end{picture}%
\caption{The images of the trajectory disks $D_{w}^{p+1}$ and
  $D_{w}^{q+1}$ in $D_w\Ver(\Sigma)$ after application of the inverse
  exponential map}
\label{perturb}
\end{figure}
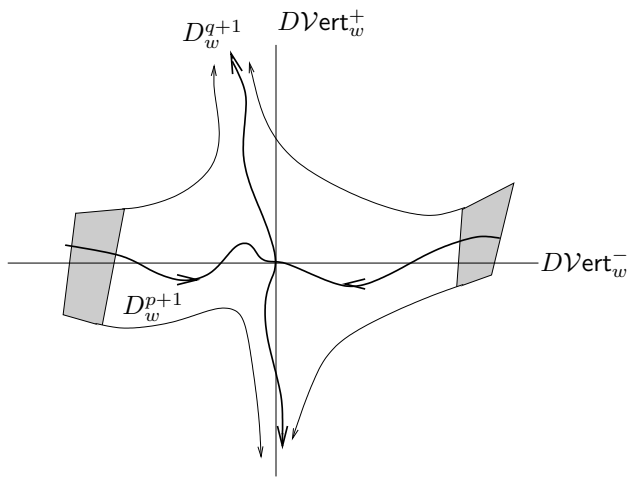 

It possible however, to isotopy $F$ to a function whose trajectory disks and eigenplanes agree on some small neighbourhood of the critical set and which is unchanged away from this neighbourhood. This is a rather delicate construction and full details may be found in \cite{BHSW}. In particular, we point out that this isotopy introduces no new critical points at any stage.

\subsection{A Review of Gromov-Lawson Concordance}
We now consider the case when $W$ is the cylinder $X\times I$ for some closed smooth manifold $X$. If $g_0$ is a psc-metric on $X$ and $f=(f,\m,V)$ is an admissible Morse triple, then the metric $\bar{g}=\bar{g}(g_0,f)$ obtained by application of Theorem \ref{GLcob}, is a concordance. We call this metric a {\em Gromov-Lawson concordance} with respect to $g_0$ and $f$. The main result of Part One can now be stated as follows.

\begin{Theorem} \label{conciso}
Let $X$ be a closed simply connected manifold of dimension $n\geq 5$. Let $g_0$ be a positive scalar curvature metric on $X$. Suppose $\bar{g}=\bar{g}(g_0,f)$ is a Gromov-Lawson concordance with respect to the metric $g_0$ and an admissible Morse triple $f=(f,\m,V)$ on the cylinder $X\times I$. Then the metrics $g_0$ and $g_1=\bar{g}|_{X\times\{1\}}$ are isotopic.
\end{Theorem}

\noindent The key geometric fact used in the proof of Theorem \ref{conciso} is Theorem \ref{concisodoublesurgery} below. 
 
\begin{Theorem} \label{concisodoublesurgery}
Let $f=(f,\m,V)$ be an admissible Morse triple on a smooth compact cobordism $W^{n+1}$. Suppose $f$ satisfies conditions (a),(b) and (c) below.
\begin{enumerate}
\item[{\bf (a)}] The function $f$ has exactly $2$ critical points $w$ and $z$ and $0<f(w)<f(z)<1$. 
\item[{\bf (b)}] The critical points $w$ and $z$ have Morse index $p+1$ and $p+2$ respectively. 
\item[{\bf (c)}] For each $t\in(f(w), f(z))$, the trajectory spheres $S_{t,+}^{q}(w)$ and $S_{t,-}^{p+1}(z)$ on the level set $f^{-1}(t)$, respectively emerging from the critical point $w$ and converging toward the critical point $z$, intersect transversely as a single point.
\end{enumerate}
\noindent Let $g$ be a metric of positive scalar curvature on $X$ and let $\bar{g}=\bar{g}(g,f)$ be a Gromov-Lawson cobordism with respect to $f$ and $g$ on $W$. Then $\bar{g}$ is a concordance and the metric $g''=\bar{g}|_{X\times\{1\}}$ on $X$ is isotopic to the original metric $g$. 
\end{Theorem}

\noindent The fact that $\bar{g}$ is a concordance follows immediately from Theorem 5.4 of \cite{Smale} as conditions (a), (b) and (c) force $W$ to be diffeomorphic to the cylinder $X_0\times I$. The rest of the proof of Theorem \ref{concisodoublesurgery}, which we discuss in the next section, is long and technical and involves explicitly constructing an isotopy between the metrics $g$ and $g''$. Roughly speaking, simple connectivity and the fact that $n\geq 5$ mean that, via Morse-Smale theory, the proof of Theorem \ref{conciso} can be reduced down to finitely many applications of the case considered in Theorem \ref{concisodoublesurgery}.

One of the main goals of this paper is to strengthen Theorem  \ref{conciso} by extending the isotopy between the metrics $g_0$ and $g_1$ to an isotopy on $X\times I$ between the metrics $g_0+dt^{2}$ and $\bar{g}$. Moreover, this isotopy should fix the metric near $X\times \{0\}$ and maintain product structure near the boundary at every stage. As one might expect, it is enough to construct this isotopy for the case of two cancelling critical points, described in Theorem \ref{concisodoublesurgery}. Unfortunately, the original proof of Theorem \ref{concisodoublesurgery} from Part One does not generalise easily to this ``boundary" case. Over the remainder of this section, we will provide a simplification of this proof which will generalise in a very natural way.

\begin{proof}
The proof of Theorem \ref{concisodoublesurgery} involves the construction of an explicit isotopy between the original metric $g$ and the metric $g''$ which has been obtained from $g$ by two  surgeries in consecutive dimensions. Although quite complicated, the construction can be summarised in the following three steps. For details, see \cite{Walsh1}.
\begin{enumerate}
\item[{\bf (1)}]By carefully analysing the Gromov-Lawson construction, we observe that the metric $g''$ can be decomposed into various regions; see Fig. \ref{generalg3 copy}. Roughly speaking, there is an {\em original} region, diffeomorphic to $X\setminus D^{n}$ where $g''$ is still the original metric $g$. There is a {\em transition} region, diffeomorphic to the cylinder $S^{n-1}\times I$ where the metric transitions from the orginal metric near one end to a standard metric near the other. Finally there is a {\em standard} region, diffeomorphic to a disk $D^{n}$ where the metric is completely standard. By standard, we mean a metric which is built using round, torpedo or mixed torpedo metrics or one which is clearly isotopic to such a metric.

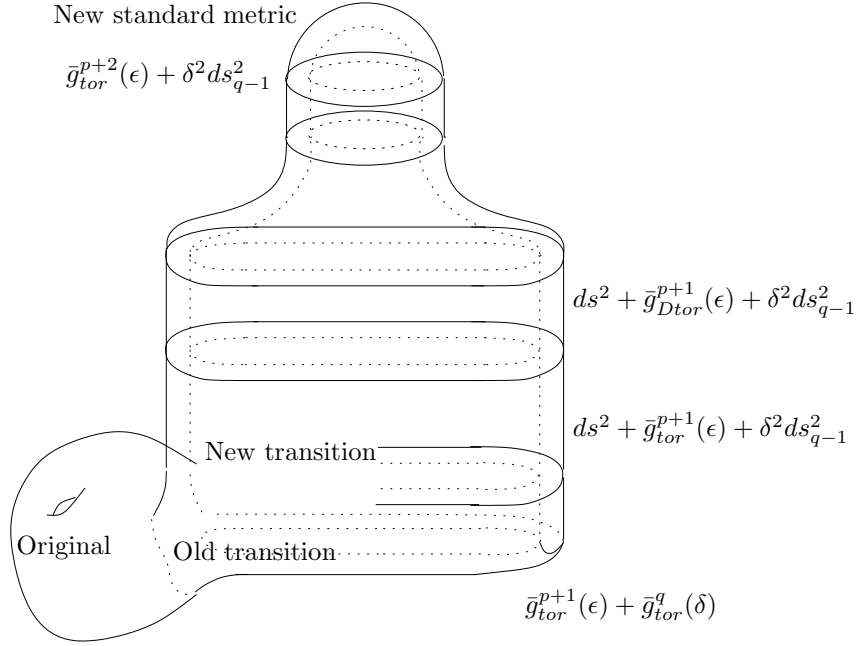
\begin{figure}
\vspace{-3cm}
\begin{picture}(0,0)%
\includegraphics[scale=0.8]{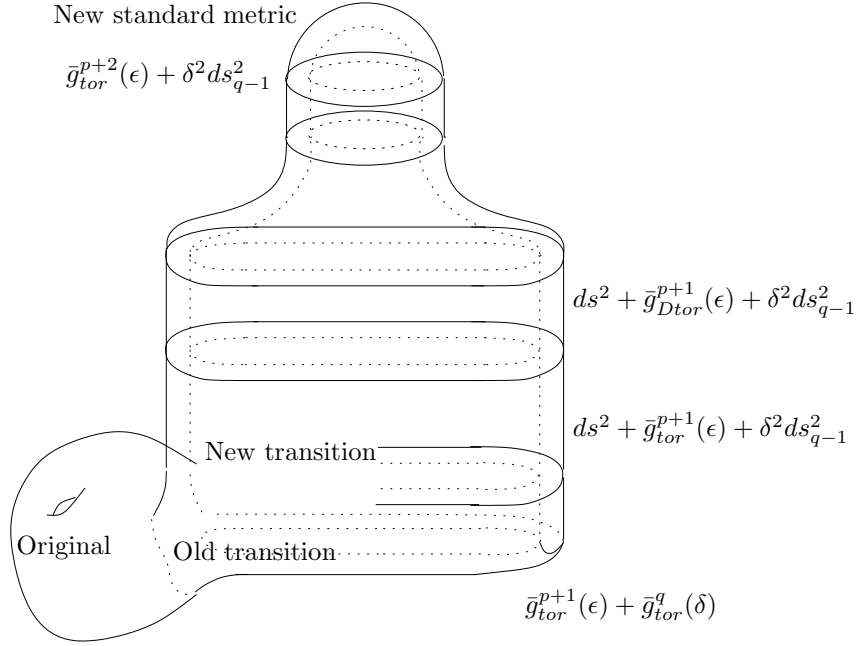}%
\end{picture}%
\setlength{\unitlength}{3947sp}%
\begingroup\makeatletter\ifx\SetFigFont\undefined%
\gdef\SetFigFont#1#2#3#4#5{%
  \reset@font\fontsize{#1}{#2pt}%
  \fontfamily{#3}\fontseries{#4}\fontshape{#5}%
  \selectfont}%
\fi\endgroup%
\begin{picture}(4555,5036)(974,-6870)
\put(1020,-6311){\makebox(0,0)[lb]{\smash{{\SetFigFont{10}{8}{\rmdefault}{\mddefault}{\updefault}{\color[rgb]{0,0,0}Original}%
}}}}
\put(2001,-6349){\makebox(0,0)[lb]{\smash{{\SetFigFont{10}{8}{\rmdefault}{\mddefault}{\updefault}{\color[rgb]{0,0,0}Old transition}%
}}}}
\put(4514,-5574){\makebox(0,0)[lb]{\smash{{\SetFigFont{10}{8}{\rmdefault}{\mddefault}{\updefault}{\color[rgb]{0,0,0}$ds^{2}+\gtor^{p+1}(\epsilon)+{\delta}^{2}ds_{q-1}^{2}$}%
}}}}
\put(4514,-4761){\makebox(0,0)[lb]{\smash{{\SetFigFont{10}{8}{\rmdefault}{\mddefault}{\updefault}{\color[rgb]{0,0,0}$ds^{2}+\bar{g}_{Dtor}^{p+1}(\epsilon)+{\delta}^{2}ds_{q-1}^{2}$}%
}}}}
\put(2201,-5724){\makebox(0,0)[lb]{\smash{{\SetFigFont{10}{8}{\rmdefault}{\mddefault}{\updefault}{\color[rgb]{0,0,0}New transition}%
}}}}
\put(1331,-3349){\makebox(0,0)[lb]{\smash{{\SetFigFont{10}{8}{\rmdefault}{\mddefault}{\updefault}{\color[rgb]{0,0,0}$\gtor^{p+2}(\epsilon)+{\delta}^{2}ds_{q-1}^{2}$}%
}}}}
\put(1251,-2974){\makebox(0,0)[lb]{\smash{{\SetFigFont{10}{8}{\rmdefault}{\mddefault}{\updefault}{\color[rgb]{0,0,0}New standard metric}%
}}}}

\put(4214,-6686){\makebox(0,0)[lb]{\smash{{\SetFigFont{10}{8}{\rmdefault}{\mddefault}{\updefault}{\color[rgb]{0,0,0}$\gtor^{p+1}(\epsilon)+\gtor^{q}(\delta)$}%
}}}}
\end{picture}%

\caption{The metric $g''$}
\label{generalg3 copy}
\end{figure}

\item[{\bf (2)}] It is possible to isotope the metric $g$ to one which agrees completely with $g''$ on the original and transition regions but which has a different sort of standard metric on the standard region; see Fig. \ref{g_2}. We will retain the name $g$ for this metric. We then observe that these respective standard metrics, which are metrics on the disk $D^{n}$, extend to psc-metrics on the sphere $S^{n}$ which are demonstrably isotopic to the mixed torpedo metrics $g_{Mtor}^{p,q}$ and $g_{Mtor}^{p+1, q-1}$.

\begin{figure}[htbp]
\vspace{-4cm}
\begin{picture}(0,0)%
\hspace{2.5cm}
\includegraphics[scale=0.7]{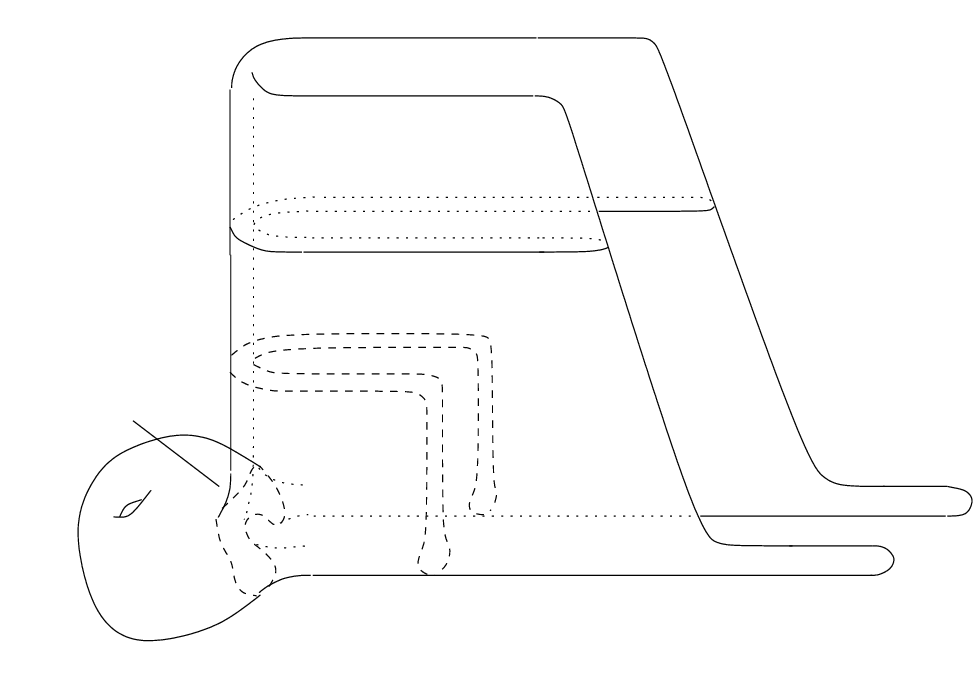}%
\end{picture}%
\setlength{\unitlength}{3947sp}%
\begingroup\makeatletter\ifx\SetFigFont\undefined%
\gdef\SetFigFont#1#2#3#4#5{%
  \reset@font\fontsize{#1}{#2pt}%
  \fontfamily{#3}\fontseries{#4}\fontshape{#5}%
  \selectfont}%
\fi\endgroup%
\begin{picture}(7788,5429)(361,-7165)
\put(1376,-5649){\makebox(0,0)[lb]{\smash{{\SetFigFont{10}{8}{\rmdefault}{\mddefault}{\updefault}{\color[rgb]{0,0,0}New transition}%
}}}}
\put(5551,-6761){\makebox(0,0)[lb]{\smash{{\SetFigFont{10}{8}{\rmdefault}{\mddefault}{\updefault}{\color[rgb]{0,0,0}Old standard metric}%
}}}}
\put(1176,-4224){\makebox(0,0)[lb]{\smash{{\SetFigFont{10}{8}{\rmdefault}{\mddefault}{\updefault}{\color[rgb]{0,0,0}New standard metric}%
}}}}
\put(1464,-4011){\makebox(0,0)[lb]{\smash{{\SetFigFont{10}{8}{\rmdefault}{\mddefault}{\updefault}{\color[rgb]{0,0,0}$\gtor^{p+1}(\epsilon)+\gtor^{q}(\delta)$}%
}}}}
\put(5501,-4799){\makebox(0,0)[lb]{\smash{{\SetFigFont{10}{8}{\rmdefault}{\mddefault}{\updefault}{\color[rgb]{0,0,0}Easy transition}%
}}}}
\end{picture}%

\caption{The metric $g$ after isotopic adjustment to coincide with the metric $g''$ on all but the standard region.}
\label{g_2}
\end{figure}
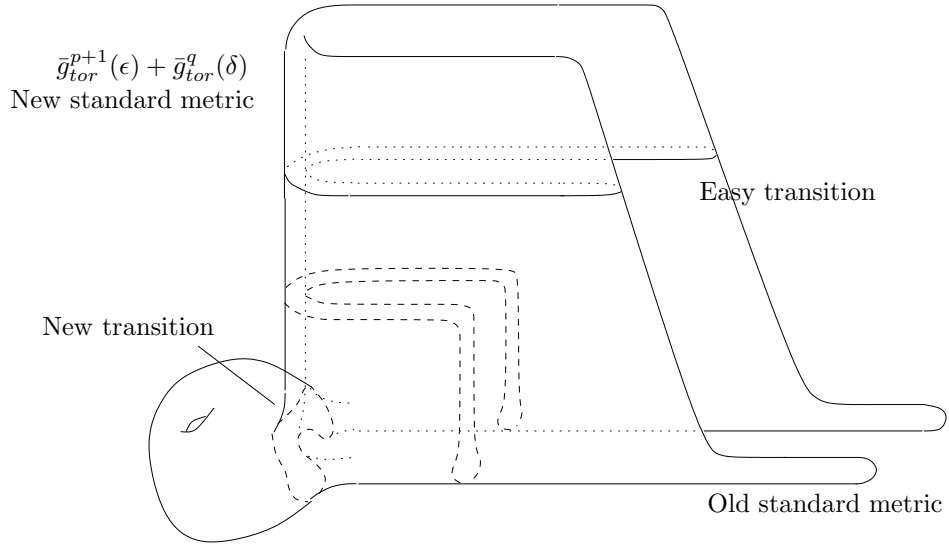

\item[{\bf (3)}] Focussing now on the transition region, where $g$ and $g''$ agree, we make adjustments on a cylindrical region near the standard end to isotopy $g$ and $g''$ into metrics which have the form of metrics obtained by the Gromov-Lawsom connect sum construction. In one case, we obtain a connected sum of a psc-metric on $X$ with the metric $g_{Mtor}^{p,q}$ on the sphere $S^{n}$. In the other case, we obtain a connected sum involving the same psc-metric on $X$ but this time with the sphere metric $g_{Mtor}^{p+1,q-1}$. As the Gromov-Lawson construction goes through for continuous compact families of metrics and as the metrics $g_{Mtor}^{p,q}$ and $g_{Mtor}^{p+1,q-1}$ are isotopic, we get that the metrics $g$ and $g''$ are isotopic also.  
\end{enumerate}

The simplification we propose will focus only on the third step above. We will show that it is unnecessary to first isotopy the metrics $g$ and $g''$ to ones taking the form of the Gromov-Lawson connected sum construction. In fact the required isotopy can be constructed so as to turn the standard region of one metric into the standard region of the other while fixing the original and transition regions (where both metrics already agree).

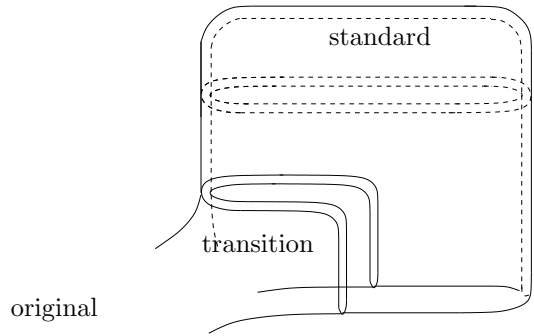
\begin{figure}[htbp]
\vspace{-3cm}
\begin{picture}(0,0)%
\hspace{5cm}
\includegraphics[scale=0.5]{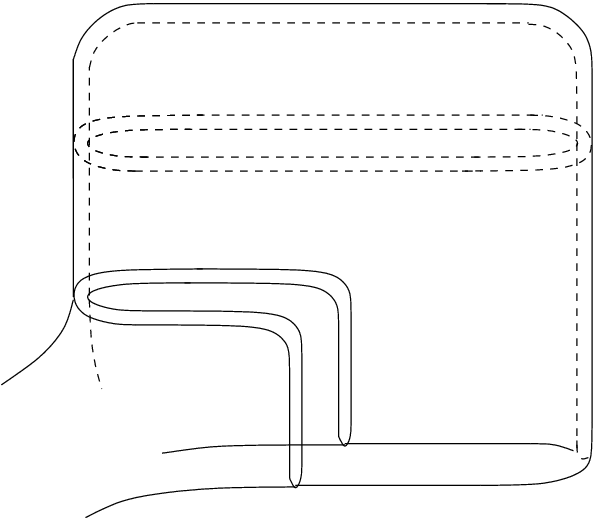}%
\end{picture}%
\setlength{\unitlength}{3947sp}%
\begingroup\makeatletter\ifx\SetFigFont\undefined%
\gdef\SetFigFont#1#2#3#4#5{%
  \reset@font\fontsize{#1}{#2pt}%
  \fontfamily{#3}\fontseries{#4}\fontshape{#5}%
  \selectfont}%
\fi\endgroup%
\begin{picture}(7349,3720)(995,-3850)
\put(4519,-2034){\makebox(0,0)[lb]{\smash{{\SetFigFont{10}{14.4}{\rmdefault}{\mddefault}{\updefault}{\color[rgb]{0,0,0}standard}%
}}}}
\put(3719,-3334){\makebox(0,0)[lb]{\smash{{\SetFigFont{10}{14.4}{\rmdefault}{\mddefault}{\updefault}{\color[rgb]{0,0,0}transition}%
}}}}
\put(2519,-3734){\makebox(0,0)[lb]{\smash{{\SetFigFont{10}{14.4}{\rmdefault}{\mddefault}{\updefault}{\color[rgb]{0,0,0}original}%
}}}}
\end{picture}%
\caption{The metric $g''$ after a minor adjustment.}
\label{g''second}
\end{figure}

\begin{figure}[htbp]
\vspace{-4cm}
\begin{picture}(0,0)%
\hspace{3cm}
\includegraphics[scale=0.5]{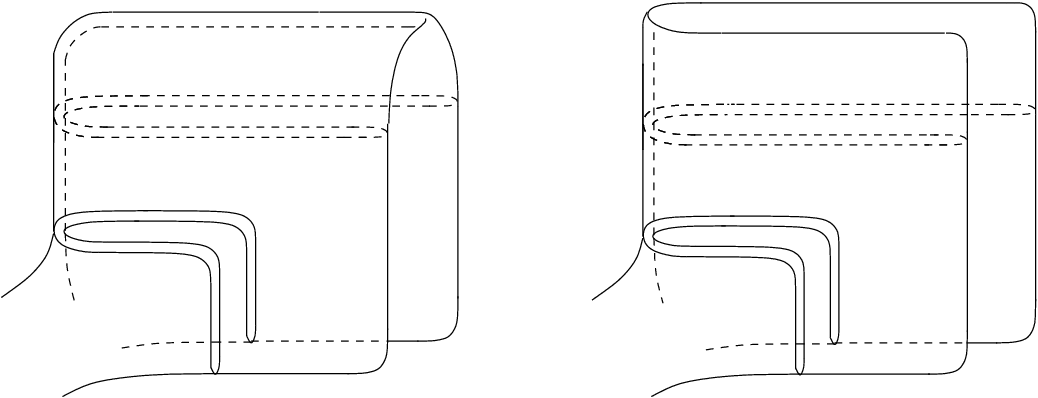}%
\end{picture}%
\setlength{\unitlength}{3947sp}%
\begingroup\makeatletter\ifx\SetFigFont\undefined%
\gdef\SetFigFont#1#2#3#4#5{%
  \reset@font\fontsize{#1}{#2pt}%
  \fontfamily{#3}\fontseries{#4}\fontshape{#5}%
  \selectfont}%
\fi\endgroup%
\begin{picture}(7349,3720)(995,-3850)
\end{picture}%
\caption{The metrics obtained by successive isotopic adjustments on metric $g''$}
\label{g''third}
\end{figure}

Recall from the proof of Theorem \ref{concisodoublesurgery}, that the metric $g''$ can, after a minor isotopy, be assumed to take the form shown in Fig. \ref{g''second}. The right side of the standard region of this metric takes the form of the mixed torpedo metric with boundary $\bar{g}_{Mtor}^{p+1, q}$. Thus, using Lemma \ref{reltoriso}, we can isotopy this metric to one which takes the form shown on the left hand image in Fig. \ref{g''third}. Repeating this procedure on the vertical part of this new standard region, we obtain the metric described in the right hand picture in Fig. \ref{g''third}. This metric is easily isotopied back to the metric shown in Fig. \ref{g_2} and then back to the original metric $g$, via the results of \cite{Walsh1}.
\end{proof}

\section{Isotoping a Gromov-Lawson Concordance}

In this section, we will prove the following extension of Theorem \ref{conciso}. Theorem \ref {conciso} itself then follows as an immediate corollary.

\subsection{Gromov-Lawson Concordance is Isotopic to a Standard Product}

\begin{Theorem} \label{relconciso}
Let $X$ be a closed simply connected manifold of dimension $n\geq 5$. Let $g$ be a positive scalar curvature metric on $X$. Suppose $\bar{g}=\bar{g}(g,f)$ is a Gromov-Lawson concordance with respect to the metric $g$ and an admissible Morse triple $f=(f,\m,V)$ on the cylinder $X\times I$. Then the metrics $\bar{g}$ and $g+dt^{2}$ are isotopic, relative to $g$ on $X\times \{0\}$, in $\Riem^{+}(X\times I, \p (X\times I))$.
\end{Theorem}

\begin{proof} The proof follows along the lines as that of Theorem \ref{conciso}. Using Morse-Smale theory, we may construct an isotopy through Morse functions to one where all of the critical points are arranged into cancelling pairs, as described in Part One. In particular, we may assume that each cancelling pair of critical points and its connecting trajectory arc, is contained in a neighbourhood and that each of these neighbourhoods is disjoint. As the Gromov-Lawson construction goes through for compact contractible families of Morse functions (Part One, Theorem 0.5), this isotopy through Morse functions gives rise to a corresponding isotopy through Gromov-Lawson concordances from the original one to one which has this nice arrangement of critical points.  

Away from the ``critical pair neighbourhoods", the metric $\bar{g}$ is a product $g+dt^{2}$. By making local adjustments to the metric on each of the critical pair neighbourhoods we can globally isotope $\bar{g}$ to a standard product $g+dt^{2}$. Once again, most of the work involves dealing with case of a Morse function with exactly $2$ cancelling critical points, as in the proof of Theorem \ref{concisodoublesurgery}. This is done in Theorem \ref{relconcisodoublesurgery} below. Once the theorem is proved for this case, the remainder of the proof follows almost exactly as before. \end{proof}

\subsection{The Case of Two Cancelling Critical Points}

\begin{Theorem} \label{relconcisodoublesurgery}
Let $f=(f,\m,V)$ be an admissible Morse triple on $X\times I$ and suppose $f$ satisfies conditions (a),(b) and (c) of Theorem \ref{concisodoublesurgery}. Let $g$ be a metric of positive scalar curvature on $X$ and let $\bar{g}=\bar{g}(g,f)$ be a Gromov-Lawson cobordism with respect to $f$ and $g$ on $X\times I$. Then the metrics $\bar{g}$ and $g+dt^{2}$ on $X\times I$ are isotopic, relative to $g$ on $X\times \{0\}$, in $\Riem^{+}(X\times I, \p (X\times I))$. 
\end{Theorem}

The proof of Theorem \ref{relconcisodoublesurgery} is very much in the spirit of the proof Theorem \ref{concisodoublesurgery}. Before beginning, we provide a brief outline of the main steps.

\begin{enumerate}
\item[{\bf (1)}]As in the original theorem, we decompose the metric $\bar{g}$ on $X\times I$ into various regions. Roughly speaking, the original region here is diffeomorphic to $X\times I \setminus D^{n}\times [1-\epsilon, 1]$. The disc $D^{n}\subset X$ is precisely the region of $X$ where adjustments are made in the original construction.

\item[{\bf (2)}] We note that inside the adjusted region $D\times[1-\epsilon, 1]$, the metric has various non-standard (transition) and standard pieces. Recall that in the original case (on just $X$) these standard pieces took the form of the mixed torpedo metrics with boundary. In this case, the standard pieces will take the form of a combination of mixed torpedo metrics with corners. 

\item[{\bf (3)}] We construct an isotopy of the metric $g+dt^{2}$ to one which agrees completely with $\bar{g}$ on the original and transition regions but which, on the standard region consists of a different combination of mixed torpedo metrics with corners. Using Lemma \ref{mixcorners}, we make isotopic adjustments to the standard region of this metric to turn it into the metric $\bar{g}$. These adjustments are completely analogous to those made in the simplified proof of Theorem \ref{concisodoublesurgery} above.
\end{enumerate}

\begin{proof}
We will begin with a careful analysis of the metric $\bar{g}$. It is worth recalling the main steps in the construction of this metric. The submanifold $f^{-1}[0,\frac{1}{2}]$ is the trace of a surgery on an embedded sphere $S^{p}\subset X$. On this region, the metric $\bar{g}$ takes the form shown in the schematic picture in the bottom right of Fig. \ref{gajersurgery}. Near $f^{-1}(0)$, this metric is the standard product $g+dt^{2}$, while near $f^{-1}(\frac{1}{2})$, it is the product $g'+dt^{2}$ where $g'$ is the metric obtained by applying Gromov-Lawson surgery to $g$ with respect to $S^{p}\subset X$. 

Recall from Part One, that this metric is obtained by first constructing an isotopy between $g$ and a metric $g_{std}$. The metric $g_{std}$ agrees with $g$ outside a tubular neighbourhood of the surgery sphere, while near $S^{p}$, takes the form $ds_{p}^{2}+g_{tor}^{q+1}$. The construction of an isotopy between $g$ and $g_{std}$ is detailed in Theorem 2.3 of Part One. Using Lemma 1.3 of Part One, this isotopy gives rise to the concordance on $X\times I$ which is $g+dt^{2}$ near $X\times\{0\}$ and $g_{std}+dt^{2}$ near $X\times \{1\}$. This corresponds to the top left picture in Fig. \ref{gajersurgery}. The picture immediately to the right of this, in Fig. \ref{gajersurgery}, describes a minor but important adjustment to the concordance $\bar{g}$, to better facilitate handle attachment. Essentially we attach a product $I\times S^{p}\times D^{q+1}$ with metric $dt^{2}+ds_{p}^{2}\times g_{tor}^{q+1}$ to the standard part of the concordance and make appropriate smoothing adjustments. In Theorem 0.4 of Part One, we showed how it is possible to do this and adjust the metric accordingly to maintain the product structure. 

In Theorem 2.2 of Part One, we complete the construction of $\bar{g}$ on $f^{-1}[0,\frac{1}{2}]$ by smoothly attaching a standard piece $\gtor^{p+1}+\gtor^{q+1}$ to the above concordance and making necessary adjustments to ensure that the resulting metric is a product near the boundary. This is shown in the bottom right picture of Fig \ref{gajersurgery}. At the bottom left, we show the result of re-attaching the handle we removed, equipped with a standard metric. It follows from Theorem 0.4 of Part One, that this metric is isotopic to the orginal concordance (top-left picture of Fig. \ref{gajersurgery}) and therefore to $g+dt^{2}$. Thus, it will be enough to show that the metric $\bar{g}$ is isotopic to the metric described in the bottom left picture of Fig. \ref{gajersurgery}.

\begin{figure}[htbp]
\vspace{-2cm}
\begin{picture}(0,0)%
\hspace{2cm}
\includegraphics[scale=0.6]{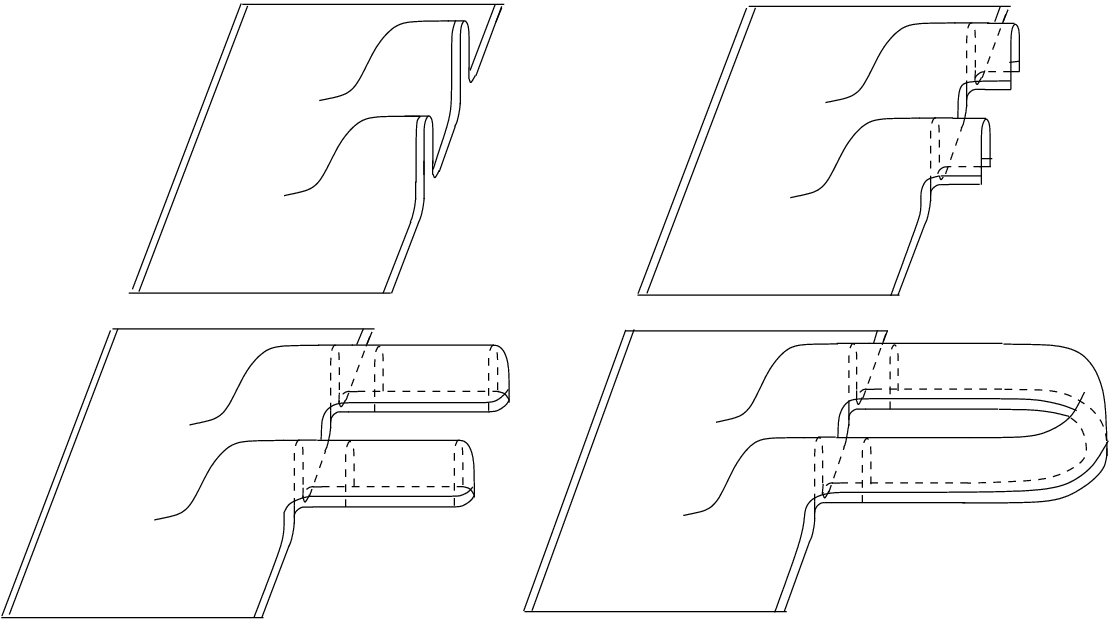}%
\end{picture}%
\setlength{\unitlength}{3947sp}%
\begingroup\makeatletter\ifx\SetFigFont\undefined%
\gdef\SetFigFont#1#2#3#4#5{%
  \reset@font\fontsize{#1}{#2pt}%
  \fontfamily{#3}\fontseries{#4}\fontshape{#5}%
  \selectfont}%
\fi\endgroup%
\begin{picture}(7349,3720)(995,-3850)
\end{picture}%
\caption{The concordance between $g$ and $g_{std}$ (top-left), a slightly adjusted version for ease of handle attachment (top-right), an alternate (but isotopic) version of this concordance (bottom-left) and the metric $\bar{g}$ restricted to $f^{-1}[0, \frac{1}{2}]$ (bottom-right)} 
\label{gajersurgery}
\end{figure}

Unfortunately, in describing the metric $\bar{g}$ obtained by performing a second surgery, the schematic picture of Fig. \ref{gajersurgery} is somewhat inadequate. Instead, we will use ``solid versions" of the schematics in Fig. \ref{g_2} and Fig. \ref{generalg3 copy}. We begin by providing an alternative schematic description of the metrics depicted in Fig. \ref{gajersurgery}. We begin with the concordance between $g$ and $g_{std}$ described in the top-left picture in Fig. \ref{gajersurgery}. It is useful to think of this as the solid object shown in Fig. \ref{Solidmetric}. In turn, the metric depicted in the top right of Fig. \ref{gajersurgery} can be re-interpreted as the metric shown in Fig. \ref{Solidmetric3}. Notice that the shaded disks at the right end correspond to torpedo metrics of the form $\gtor^{q+1}$ (as described in Fig. \ref{solidschematic}).

\begin{figure}[htbp]
\vspace{-3cm}
\begin{picture}(0,0)%
\hspace{4cm}
\includegraphics[scale=0.7]{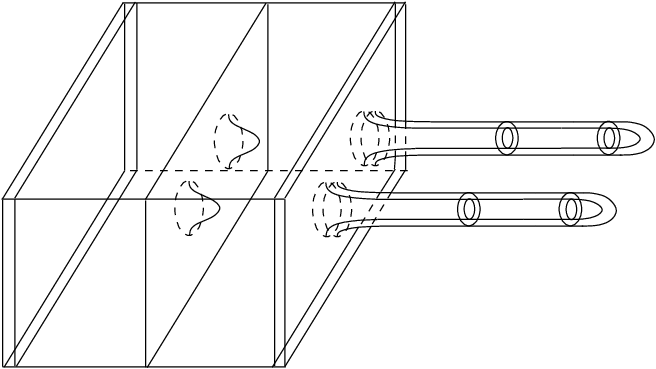}%
\end{picture}%
\setlength{\unitlength}{3947sp}%
\begingroup\makeatletter\ifx\SetFigFont\undefined%
\gdef\SetFigFont#1#2#3#4#5{%
  \reset@font\fontsize{#1}{#2pt}%
  \fontfamily{#3}\fontseries{#4}\fontshape{#5}%
  \selectfont}%
\fi\endgroup%
\begin{picture}(7349,3720)(995,-3850)
\end{picture}%
\caption{An alternative schematic for the concordance between $g$ and $g_{std}$ described in the top-left of Fig. \ref{gajersurgery}} 
\label{Solidmetric}
\end{figure}

\begin{figure}[htbp]
\vspace{-4cm}

\begin{picture}(0,0)%
\hspace{4cm}
\includegraphics[scale=0.7]{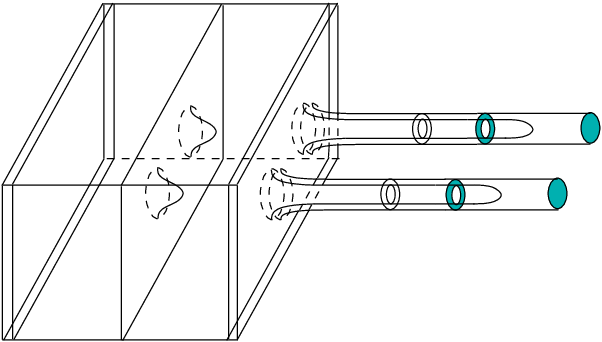}%
\end{picture}%
\setlength{\unitlength}{3947sp}%
\begingroup\makeatletter\ifx\SetFigFont\undefined%
\gdef\SetFigFont#1#2#3#4#5{%
  \reset@font\fontsize{#1}{#2pt}%
  \fontfamily{#3}\fontseries{#4}\fontshape{#5}%
  \selectfont}%
\fi\endgroup%
\begin{picture}(7349,3720)(995,-3850)
\end{picture}%
\caption{An alternative schematic for the adjusted concordance shown in the top-right of Fig. \ref{gajersurgery}} 
\label{Solidmetric3}
\end{figure}

We may now smoothly attach either the solid handle $ds_{p}^{2}+\bar{\bar{g}}_{tor}^{q+2}$, to obtain the metric corresponding to the bottom left picture in Fig. \ref{gajersurgery}, or the solid handle $\gtor^{p+1}+\gtor^{q+1}$ to obtain the metric $\bar{g}|_{f^{-1}[0,1]}$. The former is depicted in Fig. \ref{Solidmetric4cyl} below. The latter was earlier depicted as the bottom right image in Fig. \ref{gajersurgery} and is now alternatively described in Fig. \ref{Solidmetric4} below. We include in each case, a schematic representation of the embedded surgery $(p+1)$-sphere (or disk in the case of Fig \ref{Solidmetric4cyl}) corresponding to the second critical point.

\begin{figure}[htbp]
\vspace{-3cm}

\begin{picture}(0,0)%
\hspace{1cm}
\includegraphics[scale=0.7]{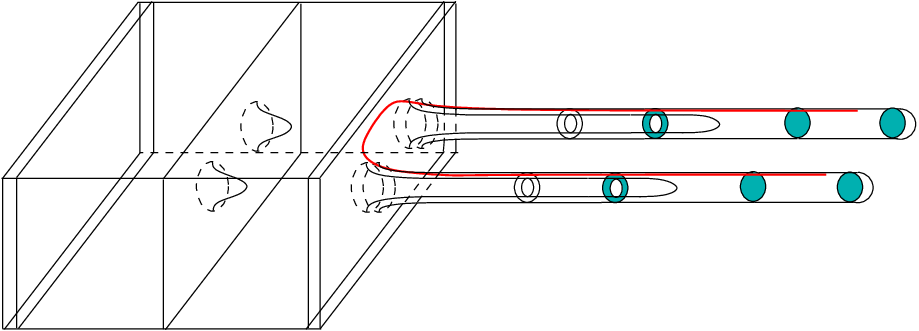}%
\end{picture}%
\setlength{\unitlength}{3947sp}%
\begingroup\makeatletter\ifx\SetFigFont\undefined%
\gdef\SetFigFont#1#2#3#4#5{%
  \reset@font\fontsize{#1}{#2pt}%
  \fontfamily{#3}\fontseries{#4}\fontshape{#5}%
  \selectfont}%
\fi\endgroup%
\begin{picture}(7349,3720)(995,-3850)
\end{picture}%
\caption{An alternative schematic for alternate concordance between $g$ and $g_{std}$ shown in the bottom-left of Fig \ref{gajersurgery}} 
\label{Solidmetric4cyl}
\end{figure}

\begin{figure}[htbp]
\vspace{-3cm}

\begin{picture}(0,0)%
\hspace{2cm}
\includegraphics[scale=0.7]{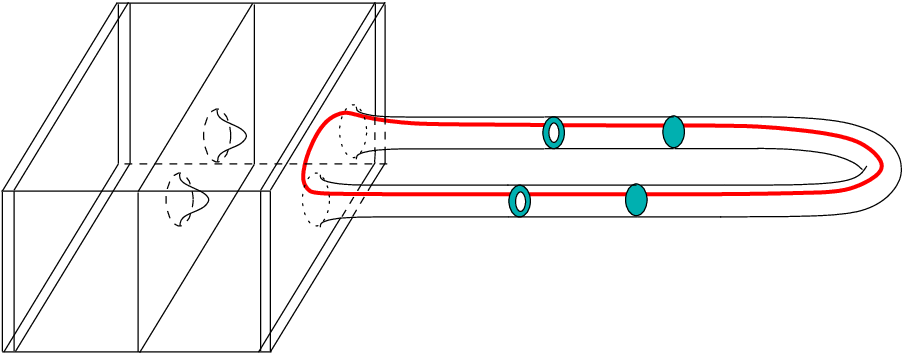}%
\end{picture}%
\setlength{\unitlength}{3947sp}%
\begingroup\makeatletter\ifx\SetFigFont\undefined%
\gdef\SetFigFont#1#2#3#4#5{%
  \reset@font\fontsize{#1}{#2pt}%
  \fontfamily{#3}\fontseries{#4}\fontshape{#5}%
  \selectfont}%
\fi\endgroup%
\begin{picture}(7349,3720)(995,-3850)
\end{picture}%
\caption{An alternative schematic for the metric $\bar{g}$ on $f^{-1}[0,\frac{1}{2}]$ initially depicted in the bottom-right of Fig. \ref{gajersurgery}}
\label{Solidmetric4}
\end{figure}

The next figure, Fig. \ref{Solidmetricsix}, describes the entire metric on $X\times I$ after extension via the second surgery. The shaded strips correspond to a product metric on a region differeomorphic to $S^{p}\times D^{q+2}$. On the $S^{p}$ factor, the metric is the standard round metric. However, on the $D^{q+2}$ factor, the metric is a psc-metric with corners of the type shown in the middle part of Fig. \ref{solidfibre1}. This metric is almost standard and can be isotoped, using the techniques of Lemma \ref{relcornerlemma}, to the standard torpedo metric with corners, $\bar{\bar{g}}^{q+2}$. Recall here that $q\geq 3$ and so the hypotheses of Lemma \ref{relcornerlemma} are satisfied. Thus the metric $\bar{g}$ can easily be isotoped to the one described in Fig. \ref{Solidmetric6next}. 
\begin{figure}[htbp]
\vspace{-1cm}

\begin{picture}(0,0)%
\hspace{2cm}
\includegraphics[scale=0.6]{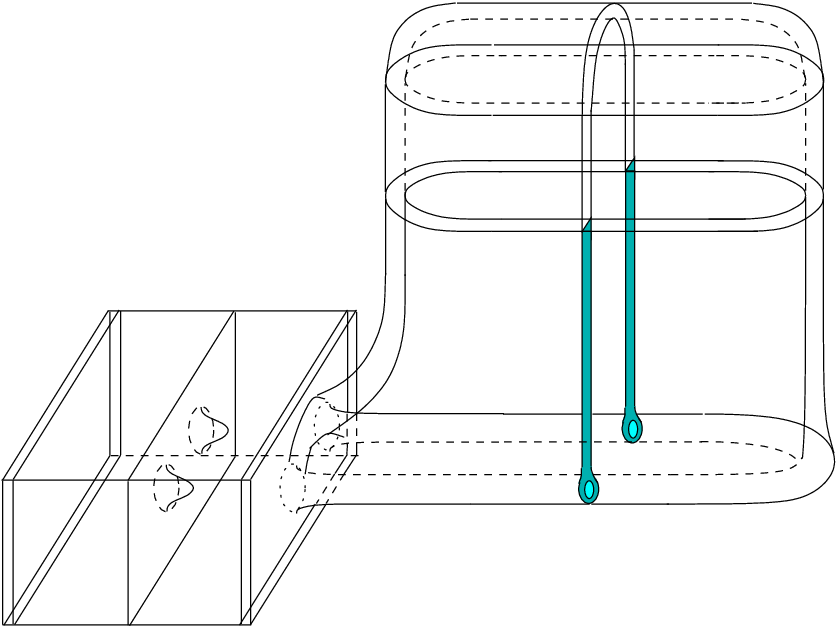}%
\end{picture}%
\setlength{\unitlength}{3947sp}%
\begingroup\makeatletter\ifx\SetFigFont\undefined%
\gdef\SetFigFont#1#2#3#4#5{%
  \reset@font\fontsize{#1}{#2pt}%
  \fontfamily{#3}\fontseries{#4}\fontshape{#5}%
  \selectfont}%
\fi\endgroup%
\begin{picture}(7349,3720)(995,-3850)
\end{picture}%
\caption{The metric $\bar{g}$ on $X\times I$ after some minor standardising adjustments} 
\label{Solidmetricsix}
\end{figure}

\begin{figure}[htbp]
\vspace{-3.5cm}

\begin{picture}(0,0)%
\hspace{3cm}
\includegraphics[scale=0.6]{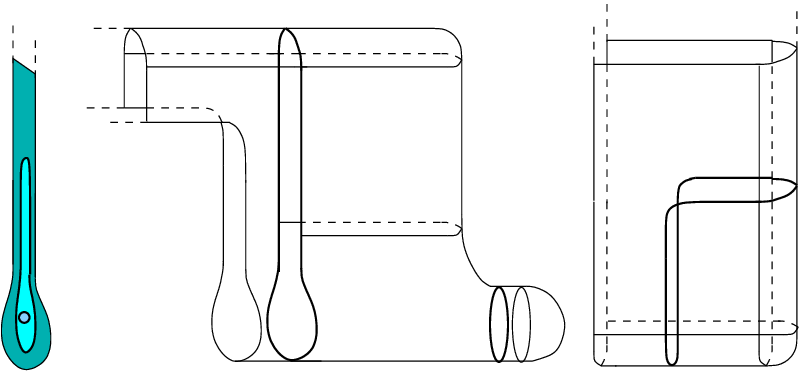}%
\end{picture}%
\setlength{\unitlength}{3947sp}%
\begingroup\makeatletter\ifx\SetFigFont\undefined%
\gdef\SetFigFont#1#2#3#4#5{%
  \reset@font\fontsize{#1}{#2pt}%
  \fontfamily{#3}\fontseries{#4}\fontshape{#5}%
  \selectfont}%
\fi\endgroup%
\begin{picture}(7349,3720)(995,-3850)
\end{picture}%
\caption{A closer view and interpretation of the shaded fibre metric on $D^{q+2}$ from Fig. \ref{Solidmetricsix} (left and middle) along with a more standard version after isotopy (right)} 
\label{solidfibre1}
\end{figure}

\begin{figure}[htbp]
\vspace{-0.5cm}

\begin{picture}(0,0)%
\hspace{2cm}
\includegraphics[scale=0.6]{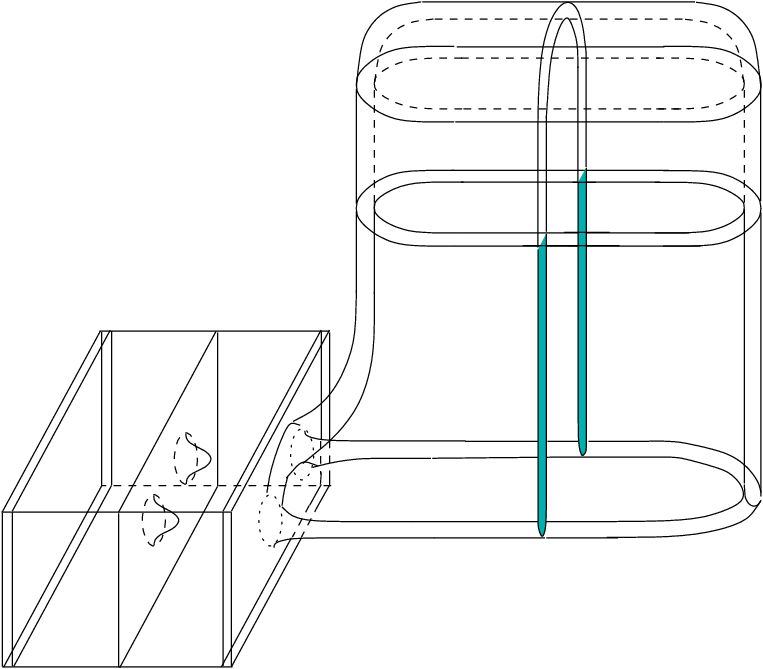}%
\end{picture}%
\setlength{\unitlength}{3947sp}%
\begingroup\makeatletter\ifx\SetFigFont\undefined%
\gdef\SetFigFont#1#2#3#4#5{%
  \reset@font\fontsize{#1}{#2pt}%
  \fontfamily{#3}\fontseries{#4}\fontshape{#5}%
  \selectfont}%
\fi\endgroup%
\begin{picture}(7349,3720)(995,-3850)
\end{picture}%
\caption{A more standard version of the metric $\bar{g}$ after some elementary isotopy} 
\label{Solidmetric6next}
\end{figure}

At this stage, we employ the method used in the simplified proof of Theorem \ref{concisodoublesurgery} (see in particular Fig. \ref{g''third}),  to adjust the metric described in Fig \ref{Solidmetric6next}. After two iterations of Lemma \ref{mixcorners} to the standard region, we obtain the metric shown in Fig. \ref{Solidmetric7}. This metric is obtained from the metric shown in Fig. \ref{Solidmetric4cyl} by two iterations of the isotopy from Lemma \ref{relativemixcorners} which exactly mimics that of the analogous case in the proof of Theorem \ref{concisodoublesurgery}. This  is exactly the metric depicted in the top left schematic of Fig. \ref{gajersurgery} and in turn is isotoped back to the standard product $g+dt^{2}$ by means of Theorem 0.4 of Part One, completing the proof.

\begin{figure}[htbp]
\vspace{-2cm}

\begin{picture}(0,0)%
\hspace{1cm}
\includegraphics[scale=0.6]{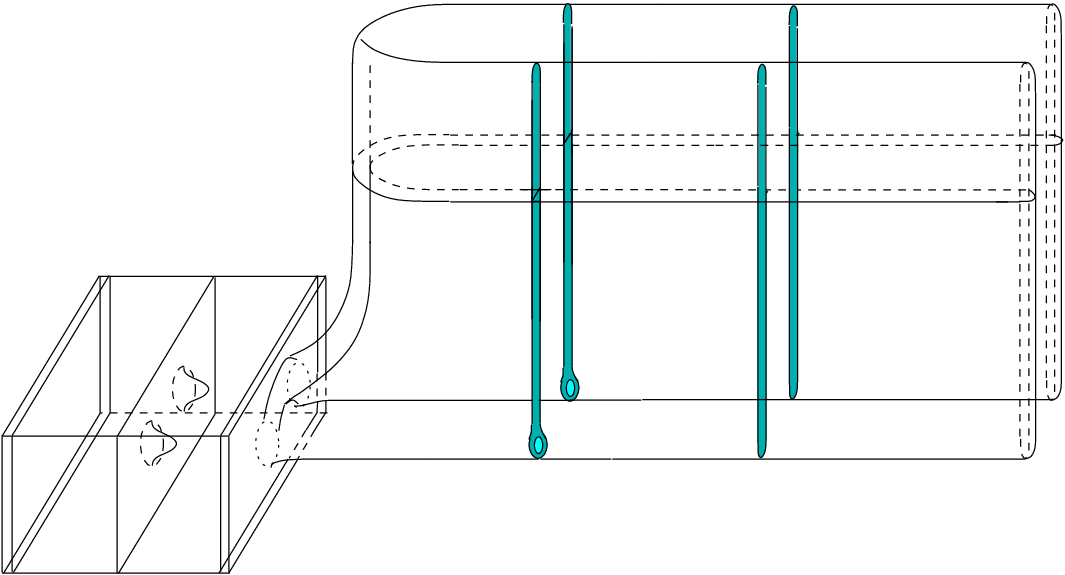}%
\end{picture}%
\setlength{\unitlength}{3947sp}%
\begingroup\makeatletter\ifx\SetFigFont\undefined%
\gdef\SetFigFont#1#2#3#4#5{%
  \reset@font\fontsize{#1}{#2pt}%
  \fontfamily{#3}\fontseries{#4}\fontshape{#5}%
  \selectfont}%
\fi\endgroup%
\begin{picture}(7349,3720)(995,-3850)
\end{picture}%
\caption{The metric obtained by applying to the standard version of $\bar{g}$ depicted in Fig. \ref{Solidmetric6next}, two iterations of the isotopy from Lemma \ref{relativemixcorners}} 
\label{Solidmetric7}
\end{figure}
\end{proof}

We close this section with the following observation. Let $f:X\times I\rightarrow I$ be the function described in the statement of Theorem \ref{relconcisodoublesurgery} above. Recall, that near each of the critical points $w$ and $z$, there are neighbourhoods, respectively diffeomorphic to $D^{p+1}\times D^{q+1}$ and $D^{p+2}\times D^{q}$, on which there are actions of $O(p+1)\times O(q+1)$ and $O(p+2)\times O(q)$. In particular, a choice of Morse coordinates near one of these critical points corresponds to a particular element of the respective orthogonal group. Although we explicitly choose a set of Morse coordinates for each of these neighbourhoods during the construction of the metric $\bar{g}$ above, Lemma \ref{GLequiv} shows that, in fact, all of the relevant data is contained in the splitting of the tangent space into positive and negative eigenspaces of the Hessian. In particular, the metric produced by the construction is independent of the particular choice of Morse coordinates and is uniquely determined by the splitting. In other words, the Gromov-Lawson construction is equivariant with respect to these orthogonal actions.

Now, let $D^{p+1}\times D^{q}$ be identified with the subset of these neighbourhoods which are perpendicular to the trajectory arc connecting $w$ and $z$. There is a corresponding copy of $O(p+1)\times O(q)\times SO(1)$ common to both groups, which acts accordingly on this subset. 

\begin{Lemma}\label{genGLequiv}
The isotopy constructed in Theorem  \ref{relconcisodoublesurgery} is equivariant with respect to the action action of $O(p+1)\times O(q)\times SO(1)$.
\end{Lemma}
\begin{proof}
This follows from the fact that at each stage in the isotopy, the changes to the metric are either entirely vertical (in the direction of the trajectory arc) or, by Lemma \ref{GLequiv}, respect the action of $O(p+1)\times O(q)\times SO(1)$. 
\end{proof}

\section{Generalised Morse Functions}

One drawback to working exclusively with Morse functions is that for any continuous family of Morse functions, the number of critical points of any index cannot vary over the family. In other words, it is not possible for two Morse functions in the family to have differing numbers of critical points of the same index. There is, however, a very natural way to vary the numbers of critical points. This involves weakening the Morse singularity condition to allow for a certain degenerate critical point  called a {\em birth-death singularity}. Roughly speaking, a birth-death singularity allows for the cancellation of two consecutively indexed Morse singularities, of the type discussed earlier. A {\em generalised Morse function} is one whose singular set contains only Morse and birth-death singularities. Thus, by working with families of generalised Morse functions it is possible to have varying numbers of critical points. 

The long term goal of this work is to generalise the ``twisted family" construction from \cite{BHSW} 
(which we discussed earlier) to work for bundles of fibrewise generalised Morse functions. Before discussing this further, it is worth reviewing some basic singularity theory. 


\subsection{Families of Generalised Morse Functions}\label{genmorsefamily}
Let $M$ be a smooth manifold of dimension $n$ and $f:M\rightarrow \mathbb{R}$ a smooth function. The singular set of $f$ is the set $\Sigma f=\{w\in M:df_w=0\}$ and a point $w\in \Sigma f$ is said to be a non-degenerate singularity if $\det{d^{2}f_w}\neq 0$ and a degenerate singularity otherwise. Non-degenerate singularities are of course just the Morse singularities discussed earlier. This is proved in a lemma of Morse; see Lemma 2.2 of \cite{MilnorMorse}. Degenerate singularities on the other hand, may be much more complicated. We will restrict our attention to one type of degenerate singularity, the birth-death singularity. A critical point $w\in \Sigma f$ is said to be {\em birth-death of index {\rm $s+\frac{1}{2}$}} if, near $w$, $f$ is locally equivalent to the map
\begin{equation*}
\begin{split}
\mathbb{R}\times\mathbb{R}^{n-1}&\longrightarrow \mathbb{R}\\
(z, x)&\longmapsto z^{3}-\sum_{i=1}^{s}{x_i}^{2}+\sum_{i=s+1}^{n-1}{x_i}^{2}. 
\end{split}
\end{equation*}
The assignment of a non-integer index to $w$ conveys the fact that at a birth-death critical point, regular Morse critical points of index $s$ and $s+1$ may cancel. 
\begin{Definition}{\rm
The smooth function $f:M\rightarrow \mathbb{R}$ is said to be a {\em generalised Morse function} if all of its degenerate singularities are of birth-death type. }
\end{Definition}
As usual, we will require some admissibility conditions on the indices of the critical points. This motivates the following definition.

\begin{Definition} 
{\rm A generalised Morse function $f:W\rightarrow I$ is said to be a {\em admissible} if all of its Morse and birth-death singularities have index $\leq n-2$.}
\end{Definition}




In discussing families of generalised Morse functions, we will restrict to the case considered in section \ref{morsefamily}. In particular, let $\pi: E^{n+k+1}\rightarrow B^{k}$ be the smooth submersion with fibre $W=~\{W^{n+1};X_0,X_1\}$, as before, and let $F:E\rightarrow B\times I$ be a smooth map satisfying the conditions (i), (ii) and (iii) of Definition \ref{admissiblemap}. The singular set $\Sigma F$ is the set $\{w\in E: \text{rank } dF_{w}<k+1\}$. Recall, a point $w\in\Sigma F$ is called a {\em fold type singularity of index {\rm $s$}} if, near $w$, the map $F$ is locally equivalent to
\begin{equation*}
\begin{split}
\mathbb{R}^{k}\times\mathbb{R}^{n+1}&\longrightarrow\mathbb{R}^{k}\times\mathbb{R}\\
(y,x)&\longmapsto\left(y, -\sum_{i=1}^{s}{x_i}^{2}+\sum_{i=s+1}^{n+1}{x_i}^2\right).
\end{split}
\end{equation*}
Furthermore, a {\em fold} of $F$ is a connected component of $\Sigma F$ which contains only fold-type singularities.

In the case when $B$ is a point, a fold singularity is just a Morse singularity of index $s$ and is thus non-degenerate, i.e. $\det{d^{2}F_w}\neq 0$. More generally, this is a degenerate singularity with $\dim (\ker {d^{2}F_w})=k$. In this case, we may regard $F$, locally, as a constant $k$-parameter family of Morse functions 
\begin{equation*}
\begin{split}
\mathbb{R}^{n+1}&\longrightarrow \mathbb{R}\\
x&\longmapsto -\sum_{i=1}^{s}{x_i}^{2}+\sum_{i=s+1}^{n+1}{x_i}^2,
\end{split}
\end{equation*}
over $\mathbb{R}^{k}$.  

In section \ref{morsefamily}, we restricted to the case when $F$ had only fold singularities. In other words the restriction of $F$ to fibres was a Morse function. In this section, we wish to weaken this condition so that the restriction of $F$ on fibres is a generalised Morse and so may have birth-death singularities. Thus, as well as folds, $F$ is allowed to have what are called cusps. 

\begin{Definition} {\rm Let $k\geq 1$. A point $w\in\Sigma F$ is called a {\em cusp type singularity of index {\rm $s+\frac{1}{2}$}} if, near $w$, the map $F$ is locally equivalent to 
\begin{equation*}
\begin{split}
\mathbb{R}^{k}\times\mathbb{R}\times\mathbb{R}^{n}&\longrightarrow\mathbb{R}^{k}\times\mathbb{R}\\
(y,z,x)&\longmapsto\left(y, z^{3}+3y_1z-\sum_{i=1}^{s}{x_i}^{2}+\sum_{i=s+1}^{n-k}{x_i}^2\right).
\end{split}
\end{equation*}
}
\end{Definition}

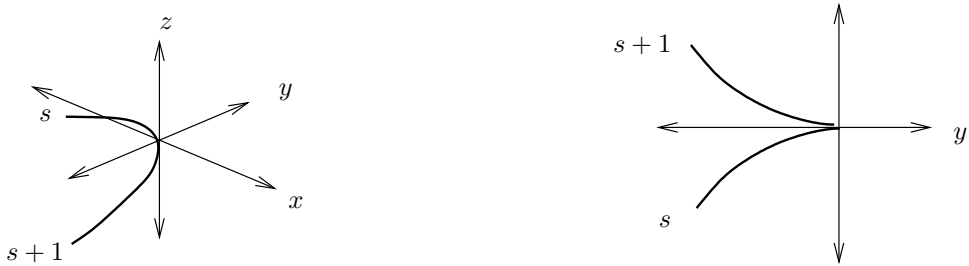
\begin{figure}[htbp]
\vspace{0mm}
\hspace{1cm}
\begin{picture}(0,0)%
\includegraphics{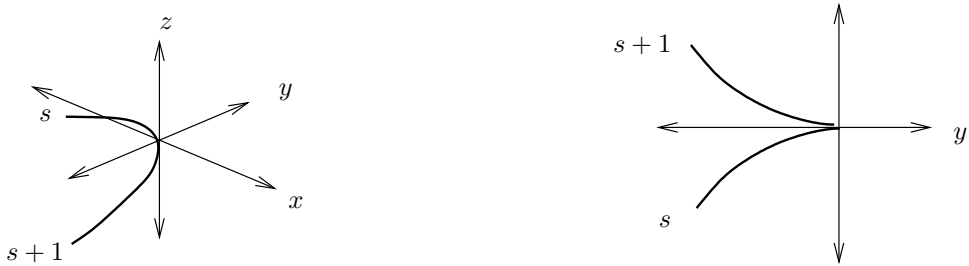}%
\end{picture}%
\setlength{\unitlength}{3947sp}%
\begingroup\makeatletter\ifx\SetFigFont\undefined%
\gdef\SetFigFont#1#2#3#4#5{%
  \reset@font\fontsize{#1}{#2pt}%
  \fontfamily{#3}\fontseries{#4}\fontshape{#5}%
  \selectfont}%
\fi\endgroup%
\begin{picture}(5880,1698)(1573,-3054)
\put(3263,-2663){\makebox(0,0)[lb]{\smash{{\SetFigFont{10}{12}{\rmdefault}{\mddefault}{\updefault}{\color[rgb]{0,0,0}$x$}%
}}}}
\put(1696,-2113){\makebox(0,0)[lb]{\smash{{\SetFigFont{10}{12}{\rmdefault}{\mddefault}{\updefault}{\color[rgb]{0,0,0}$s$}%
}}}}
\put(1488,-2988){\makebox(0,0)[lb]{\smash{{\SetFigFont{10}{12}{\rmdefault}{\mddefault}{\updefault}{\color[rgb]{0,0,0}$s+1$}%
}}}}
\put(3201,-1950){\makebox(0,0)[lb]{\smash{{\SetFigFont{10}{12}{\rmdefault}{\mddefault}{\updefault}{\color[rgb]{0,0,0}$y$}%
}}}}
\put(2451,-1538){\makebox(0,0)[lb]{\smash{{\SetFigFont{10}{12}{\rmdefault}{\mddefault}{\updefault}{\color[rgb]{0,0,0}$z$}%
}}}}
\put(5301,-1680){\makebox(0,0)[lb]{\smash{{\SetFigFont{10}{12}{\rmdefault}{\mddefault}{\updefault}{\color[rgb]{0,0,0}$s+1$}%
}}}}
\put(5588,-2780){\makebox(0,0)[lb]{\smash{{\SetFigFont{10}{12}{\rmdefault}{\mddefault}{\updefault}{\color[rgb]{0,0,0}$s$}%
}}}}
\put(7438,-2230){\makebox(0,0)[lb]{\smash{{\SetFigFont{10}{12}{\rmdefault}{\mddefault}{\updefault}{\color[rgb]{0,0,0}$y$}%
}}}}
\end{picture}
\caption{A cusp singularity and its image where $k=1$.}
\label{cusp}
\end{figure}

As before, we may regard $F$ as a $k$-parameter family of functions, although unlike the fold case this family is not constant. In the above coordinates, the singular set of $F$ is
\begin{equation*}
\Sigma F=\{(y,z,x): z^{2}+y_1=0, x=0\}.
\end{equation*}
Thus, when $y_1>0$, the function $F$ is locally a $k$-parameter family of Morse functions with no critical points, parametrised by $y\in(0,\infty)\times\mathbb{R}^{k-2}$. At $y_1=0$, the function $F$ is a $(k-1)$-parameter family of generalised Morse functions, each with exactly one birth-death critical point occurring at $(z=0, x=0)$. When $y_1<0$, $F$ is a $k$-parameter family of Morse functions each with exactly two critical points, parametrised by $y\in (-\infty, 0)\times \mathbb{R}^{k-1}$. Each Morse function in this family has a critical point of index $s$ at $(z=\sqrt{-y_1},x=0)$ and a critical point of index $s+1$ at $(z=-\sqrt{-y_1}, x=0)$.  Thus, as $y_1\rightarrow 0^{-}$, these pairs of Morse critical points converge and cancel as a $(k-1)$-parameter family of birth-death singularities. 

The case when $k=1$ is illustrated in Figures \ref{cusp} and \ref{genmorseunfolding}. This is best thought of as a $1$-parameter family of functions 
\begin{equation*}
\begin{split}
q_y:\mathbb{R}\times\mathbb{R}^{n-2}&\longrightarrow\mathbb{R}\\
(z, x)&\longmapsto z^{3}+3yz-\sum_{i=1}^{s}{x_i}^{2}+\sum_{i=s+1}^{n-2}{x_i}^2,
\end{split}
\end{equation*}
parametrised by $y\in\mathbb{R}$. In these coordinates, the singular set $\Sigma F$ is the curve $z^{2}+y=0$ on the plane $x=0$, shown in Fig. \ref{cusp}. This particular example is known as the {\em standard unfolding of a birth-death singularity}. We close this section by observing the topological effects of the unfolding. These are illustrated in Fig. \ref{genmorseunfolding} by selected level sets $q_y=q_y(\sqrt{c},0)-\epsilon$, $q_y=0$ and $q_y=q_y(-\sqrt{c},0)+\epsilon$ for $y=-c, 0$ and $c$, where $c$ and $\epsilon$ are both positive constants. The critical points of index $s$ and $s+1$ occur at $z=\sqrt{c}$ and $z=-\sqrt{c}$ respectively for the function $q_{-c}$.
The birth-death singularity occurs on the level set $q_0=0$, shown in the centre of this figure, while the function $q_c$ has no critical points.

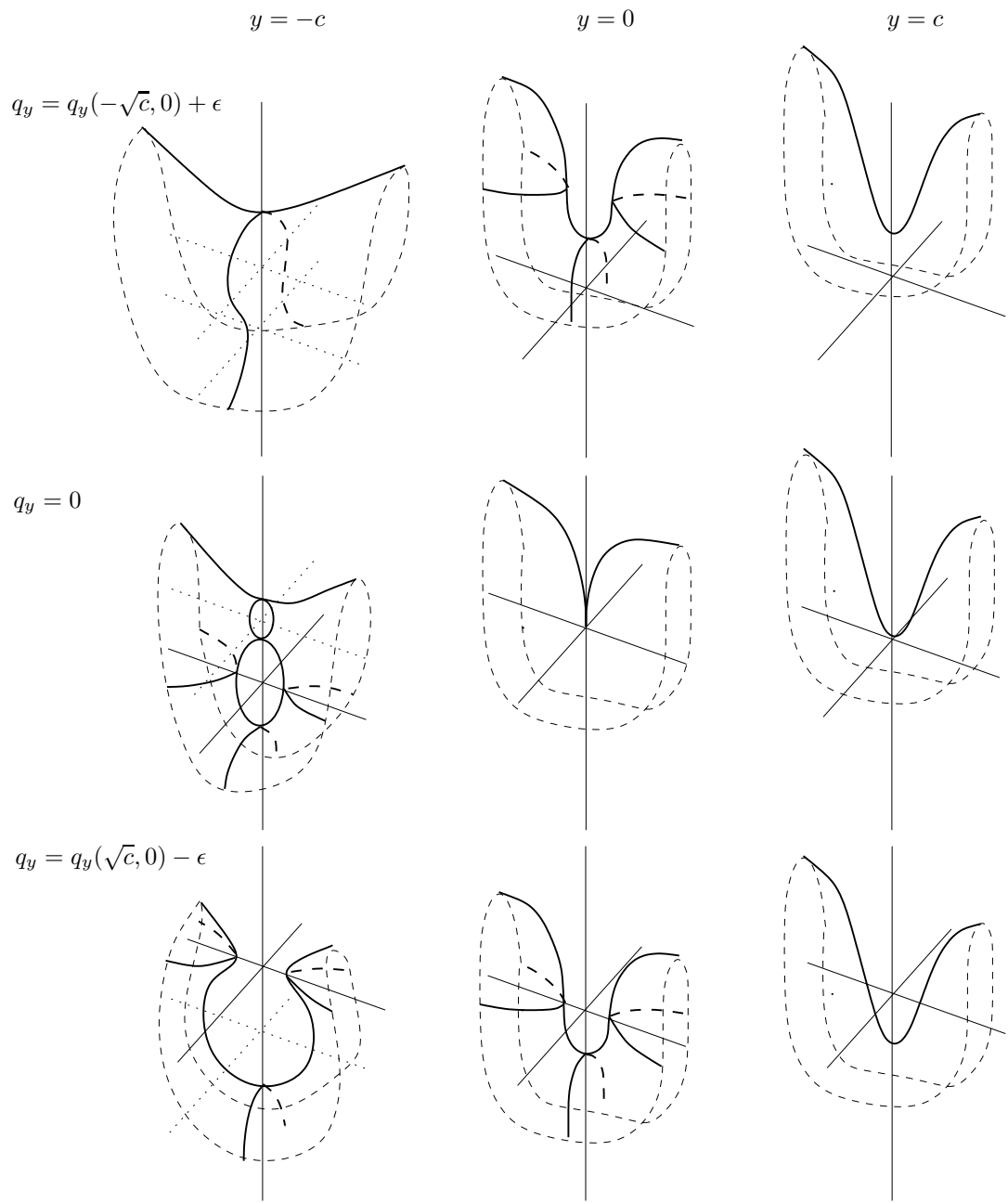
\begin{figure}[htbp]
\vspace{20mm}
\begin{picture}(0,0)%
\includegraphics[scale=0.8]{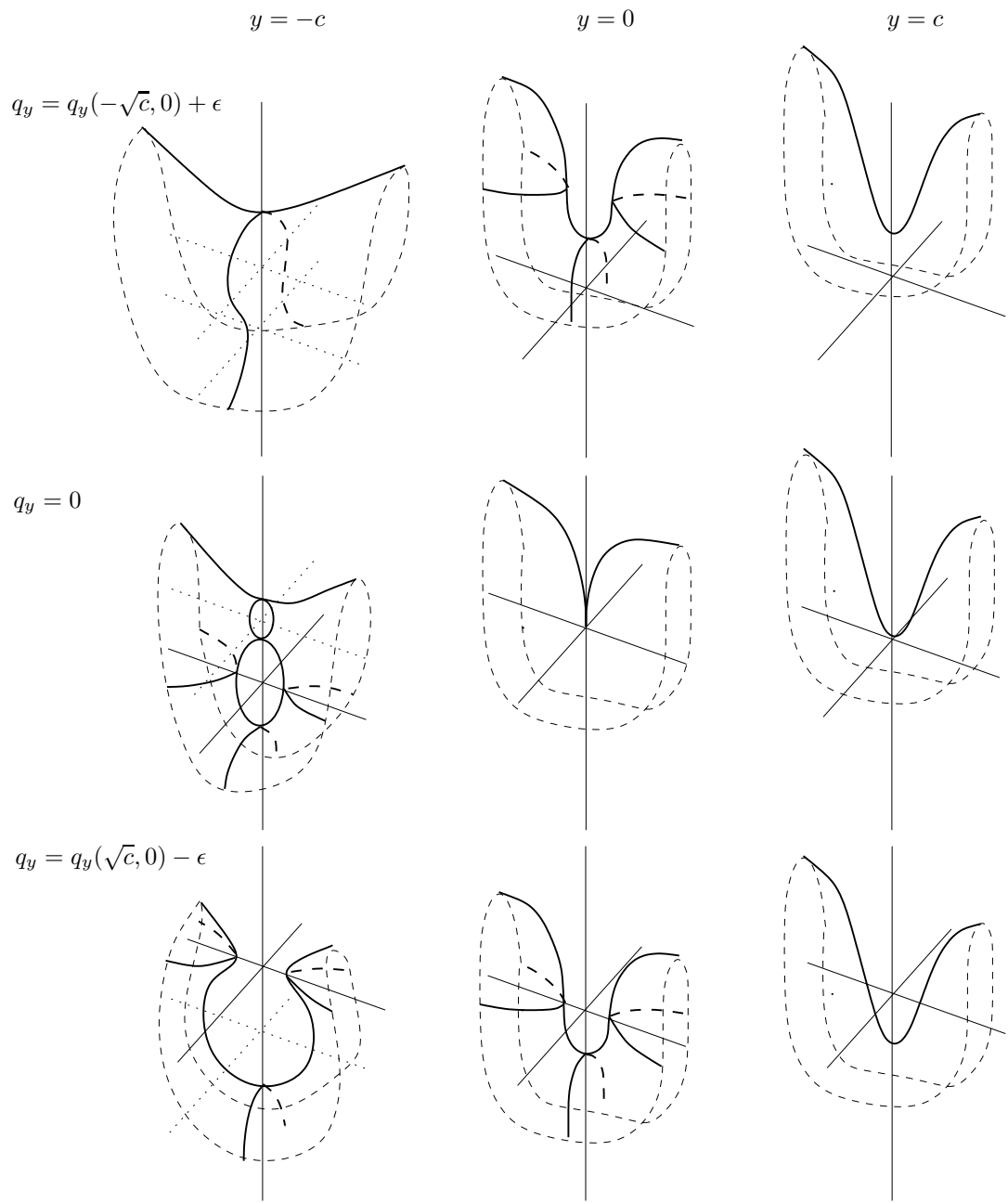}%
\end{picture}%
\setlength{\unitlength}{3947sp}%
\begingroup\makeatletter\ifx\SetFigFont\undefined%
\gdef\SetFigFont#1#2#3#4#5{%
  \reset@font\fontsize{#1}{#2pt}%
  \fontfamily{#3}\fontseries{#4}\fontshape{#5}%
  \selectfont}%
\fi\endgroup%
\begin{picture}(8438,8638)(236,-10627)
\put(1864,-2000){\makebox(0,0)[lb]{\smash{{\SetFigFont{10}{12}{\rmdefault}{\mddefault}{\updefault}{\color[rgb]{0,0,0}$y=-c$}%
}}}}
\put(4076,-2000){\makebox(0,0)[lb]{\smash{{\SetFigFont{10}{12}{\rmdefault}{\mddefault}{\updefault}{\color[rgb]{0,0,0}$y=0$}%
}}}}
\put(6176,-2000){\makebox(0,0)[lb]{\smash{{\SetFigFont{10}{12}{\rmdefault}{\mddefault}{\updefault}{\color[rgb]{0,0,0}$y=c$}%
}}}}
\put(251,-2574){\makebox(0,0)[lb]{\smash{{\SetFigFont{10}{12}{\rmdefault}{\mddefault}{\updefault}{\color[rgb]{0,0,0}$q_y=q_y(-\sqrt{c},0)+\epsilon$}%
}}}}
\put(264,-5261){\makebox(0,0)[lb]{\smash{{\SetFigFont{10}{12}{\rmdefault}{\mddefault}{\updefault}{\color[rgb]{0,0,0}$q_y=0$}%
}}}}
\put(276,-7674){\makebox(0,0)[lb]{\smash{{\SetFigFont{10}{12}{\rmdefault}{\mddefault}{\updefault}{\color[rgb]{0,0,0}$q_y=q_y(\sqrt{c},0)-\epsilon$}%
}}}}
\end{picture}%
\caption{Selected level sets showing the unfolding of a birth-death singularity}
\label{genmorseunfolding}
\end{figure}

Finally, we define the generalised analogue of the map $F$, from section \ref{morsefamily}.

\begin{Definition}\label{admgenmorsefam}{\rm
Let $\pi:E^{n+1+k}\rightarrow B^{k}$ be the smooth bundle described above, with fibre $W^{n+1}$, a smooth compact cobordism of closed manifolds, $\p W= X_0\sqcup X_1$. Let $F:E\rightarrow B\times I$ be a smooth map. The map $F$ is said to be an {\em admissible family of generalised Morse functions} provided it satisfies the following conditions.
\begin{enumerate}
\item[{\bf (i)}] For all $w\in E$, $\pi(w)=p_1\circ F(w)$.
\item[{\bf (ii)}] The pre-images $F^{-1}(B\times\{0\})$ and $F^{-1}(B\times\{1\})$ are the submanifolds $E_0$ and $E_1$ respectively.
\item[{\bf (iii)}] The singular set $\Sigma F$ is contained entirely in $E\setminus(E_0\sqcup E_1)$.
\item[{\bf (iv)}] For each $y\in B$, the restriction $f_y=F|_{W_y}$ is an admissible generalised Morse function, i.e. one whose Morse and birth-death critical points have index $\leq n-2$.
\end{enumerate}}
\end{Definition}

Once again we will require a background metric. Recall, that for an admissible Morse function $f:W\rightarrow I$, a Riemannian metric $\m$ on $W$ was said to be compatible if at each critical point of $f$, the corresponding eigenspaces of the Hessian of $f$ were orthogonal with respect to $\m$. Now suppose $f$ is a generalised Morse function, and $w\in W$ is a birth-death singularity of $f$. In this case the tangent space $T_wW$ splits into three subspaces, a pair of positive and negative eigenspaces $T_wW^{+}$ and $T_wW^{-}$, as well as a one-dimensional kernel~$T_wW^{0}$,
\begin{equation*}\label{genmorsplit}
T_{w}W=T_{w}W^{0}\oplus T_{w}W^{-}\oplus T_wW^{+}.
\end{equation*}
In this case, the notion of compatibility generalises as follows.
\begin{Definition}\label{gencompat} {\rm Let $f:W\rightarrow I$ be a generalised Morse function. A Riemannian metric $\m$ on $W$ is {\em compatible} with $f$ if the following conditions are satisfied.
\begin{enumerate}
\item{} For every Morse critical point $w$ of $f$, $\m$ is compatible in the original sense.
\item{} At each birth-death critical point $w\in W$, the splitting in \ref{genmorsplit} is orthogonal with respect to $\m$.
\item{} $\m|_{T_wW^{-}}=d^{2}f|_{T_wW^{-}}$, $\m|_{T_wW^{+}}=d^{2}f|_{T_wW^{+}}$ and $\m|_{T_wW^{0}}=dt^{2}$.
\end{enumerate}}
\end{Definition}

Now, let $F$ be a family of admissible generalised Morse functions. That is, $F$ is a map from $E$ to $B\times I$ satisfying the conditions of Definition \ref{admgenmorsefam} and containing only fold and cusp singularities. Let $\m|_{E}$ and $\m|_B$ be a pair of Riemannian metrics on $E$ and $B$ which give rise to a Riemannian submersion $\pi:(E, \m_E)\rightarrow (B, \m_{B})$. We now make the following definitions.
\begin{Definition}
{\rm The metric $\m_E$ is {\em compatible} with $F$ provided the restriction $\m_{y}$ of the metric $\m_E$ to the fibre $W_{y}$, for each $y\in B$, is compatible in the sense of Definition \ref{gencompat}.}
\end{Definition}

\begin{Definition}
{\rm A vector field $V_E$ is said to be {\em gradient-like} on $F$ and $\m_E$ provided the following hold.
\begin{enumerate}
\item{} For each $y\in B$, the restriction of $V_E$ to the fibre $W_y$, denoted $V_w$ satisfies $df_y(V_w)>0$ away from critical points. Recall that $f_y=F|_{W_y}$. 
\item{} Near critical points of $F$, $V_E$ agrees with $\grad F$.
\end{enumerate}}
\end{Definition}
 
\noindent As before we will consider triples $(F, \m_{E}, V_E)$, consisting of an admissible family of generalised Morse functions $F$, a compatible Riemannian metric $\m_E$ (which actually denotes a further Riemannian submersion structure) and a gradient-like vector field $V_E$. Notationally, this is extremely cumbersome and so once again we abbreviate such a triple by $F$.

Finally, we must say some words about the restriction of the tangent bundle $TE$ to singular points of $F$. The singular set $\Sigma=\Sigma F$ decomposes as
\begin{equation*}
\Sigma = \Sigma^{0}\cup\Sigma^{1},
\end{equation*}
where $\Sigma^{0}$ denotes the fold, and $\Sigma^{1}$ the cusp singularities of $F$. 
Recall that the tangent bundle $TE$ contains a vertical subbundle $\Ver$ with fibre at each $w\in W$, $\Ver_w=\Ker{d\pi_w}$. From the decomposition described in \ref{genmorsplit} above, we obtain a corresponding decomposition of $\Ver$ on $\Sigma^{1}F$:
\begin{equation*}
\Ver|_{\Sigma^{1}}=\Ver^{0}\oplus\Ver^{-}\oplus\Ver^{+}.
\end{equation*}
The Riemannian metric $\m|_E$ reduces the structure group to $SO(1)\times O(p+1)\times O(q)$, where $p+\frac{3}{2}$ is the index of critical points in $\Sigma^{1}$. Shortly, we will prove an analogue of Theorem \ref{main}, for the case of an admissible family of generalised Morse functions. Before then, it is worth saying a few words about the space of generalised Morse functions.

\subsection{The Space of Generalised Morse Functions}

Let $\{W^{n+1}; X_0, X_1\}$ denote a smooth compact cobordism of the type discussed earlier. Recall, $\F=\F(W)$ denotes the  space of smooth functions $W\rightarrow I$ with $f^{-1}(0)=X_0$, $f^{-1}(1)=X_1$ and with $\Sigma f$ contained entirely in the interior of $W$. This is a subspace of $C^{\infty}(W, I)$ and inherits the subspace topology.  The subspace of $\F$ consisting of Morse functions, is denoted $\Mor$. This space is not path connected as functions lying in the same path component of $\Mor$ must have the same number of critical points of the same index. There is, as we discussed earlier, a natural way to connect up path components of this space.  Let $\GMor=\GMor (W)$ denote the subspace of $\F$ which consists of all generalised Morse functions. Recall that the singular set of a generalised Morse function consists of both Morse and birth-death singularities.  Summarising, we have the following inclusions:
\begin{equation*}
\Mor\subset \GMor \subset \F.
\end{equation*}

\begin{Remark}
As before, we observe that for a given generalised Morse function $f$, the space of compatible metrics is a convex space and so the space of compatible triples $(f,\m,V)$, where $f$ is an admissible generalised Morse function, $\m$ is a compatible metric and $V$ is a gradient-like vector field is homotopy equivalent to the space $\GMor$. 
\end{Remark}

From the work of K.Igusa, we obtain the following lemma.  
\begin{Lemma}\label{genmorseconnect}
The space $\GMor$ is path-connected. In particular, any two Morse functions on $W$ may be connected by a path through generalised Morse functions on $W$.
\end{Lemma} 
 \begin{proof}
In \cite{I1}, the author constructs an $(n+1)$-connected map
\begin{equation*}
\GMor\longrightarrow \Omega^{\infty}\Sigma^{\infty}(BO\wedge W_{+}),
\end{equation*}
where $W_{+}=W\sqcup pt$. As the right-hand side is path-connected, the result follows.
 \end{proof}
 
 In the previous section, we defined what is meant by a family of generalised Morse functions (albeit in the admissible case). For technical reasons, it will be necessary that we are working with families whose birth-death points unfold in the way described in section \ref{genmorsefamily}. In other words, we are interested in families whose singular sets consist of only folds and cusps. 
 
\begin{Definition}
{\rm A family $F$, of generalised Morse functions, is said to be {\em moderate} provided that the singular set $\Sigma F$, consists of only fold and cusp singularities. }
\end{Definition}

We are almost in a position to state the main results. Before doing this, we will briefly discuss some work of Hatcher on the connectivity of some important subspaces of the space of generalised Morse functions. 
  

\subsection{Hatcher's 2-Index Theorem}

Given the importance of critical points of admissible index in this work, it is worth defining the following spaces. Let $\GMor_{i,j}$ denote the subspace of $\GMor$ consisting of all generalised Morse functions with only critical points of index between $i$ and $j$ inclusively. Of special interest to us is the space $\GMor^{adm}=\GMor_{0,n-2}$, the space of admissible generalised Morse functions on $W$. It will be important for us to be able to connect up an arbitrary pair of admissible Morse functions with a path through admissible generalised Morse functions. To do this, we will need the following corollary of Hatcher's 2-index Theorem; see Theorem 1.1, Chapter VI, Section 1 of \cite{I3}. 

\begin{Theorem}\label{2index}{\rm (Corollary 1.4, Chapter VI, \cite{I3})}
Under the following conditions the inclusion map $\GMor_{i,j-1}\longrightarrow \GMor_{i,j}$ is $k$-connected.
\begin{enumerate}
\item[(a)] $(W,X_1)$ is $(n-j+1)$-connected.
\item[(b)] $j\geq i+2$.
\item[(c)] $n-j+1\leq n-k-1-\min(j-1,k-1)$.
\item[(d)] $n-j+1\leq n-k-3$.
\end{enumerate}
\end{Theorem}

\section{Parameterising Gromov-Lawson Cobordisms by Generalised Morse Functions}\label{genparam}
In this section, we will generalise Theorem \ref{main} to work for families of admissible generalised Morse functions. As a first step, we consider the case of families obtained as paths in the space of generalised Morse functions. As usual, $W=\{W^{n+1};X_0, X_1\}$ is a smooth compact cobordism and $\GMor=\GMor{(W)}$ is the space of generalised Morse functions on $W$.

\subsection{Applying the Construction over One-Parameter Families}
The simplest non-trivial case of a path in $\GMor$ is one which connects a Morse function with two cancelling critical points (as in Theorem \ref{concisodoublesurgery}) to a function which has no critical points. 
Let $f=(f,\m,V)$ be an be an admissible Morse triple on $W$, satisfying conditions (a), (b) and (c) of Theorem \ref{concisodoublesurgery}. Recall these conditions are as follows.
\begin{enumerate}
\item[{\bf (a)}] The function $f$ has exactly $2$ critical points $w$ and $z$ and $0<f(w)<f(z)<1$. 
\item[{\bf (b)}] The points $w$ and $z$ have Morse index $p+1$ and $p+2$ respectively. 
\item[{\bf (c)}] For each $t\in(f(w), f(z))$, the trajectory spheres $S_{t,+}^{q}(w)$ and $S_{t,-}^{p+1}(z)$ on the level set $f^{-1}(t)$, respectively emerging from the critical point $w$ and converging toward the critical point $z$, intersect transversely as a single point.
\end{enumerate}
Let $K_{-}^{p+1}(w)\subset f^{-1}([0,f(w)])$ denote the inward trajectory disk of $w$. This disk is bounded by a trajectory sphere which we denote $S_{-}^{p}\subset X_0$. Let $t\in(f(w), f(z))$. Emerging from $w$ is an outward trajectory disk $K_{t,+}^{q+1}(w)\subset f^{-1}([f(w),t])$ which is bounded by an outward trajectory sphere $S_{t,+}^{q}\subset f^{-1}(t)$. Similarly, associated to $z$ is an inward trajectory disk $K_{t,-}^{p+2}(w)\subset f^{-1}([t, f(z)])$ bounded by an inward trajectory sphere $S_{t,-}^{p+1}\subset{f^{-1}(t)}$ and an outward trajectory disk $K_{+}^{q}(z)\subset f^{-1}([f(z),1])$ bounded by an outward trajectory sphere $S_{+}^{q-1}\subset X_1$. We define a smooth {\em trajectory arc} $\gamma:[f(w),f(z)]\rightarrow W$ by the formula
\begin{equation*}\label{pwf(t)}
\begin{array}{c}
{\gamma}(t) =
\begin{cases}
w, & \text{when $t=f(w)$}\\
S_{t,+}^{q}\cap S_{t,-}^{p+1}, & \text{when $t\in(f(w), f(z))$}\\
z, & \text{when $t=f(z)$}.
\end{cases}
\end{array}
\end{equation*} 
Condition (c) means that for each $t\in (f(w), f(z))$, the intersection $S_{t,+}^{q}\cap S_{t,-}^{p+1}$ is a single point and so this formula makes sense. 

The embedded sphere $S_{-}^{p}$ in $X\times \{0\}$, bounds a particular embedded disk which we denote $D_{-}^{p+1}$. This disk is determined as follows. Let $t\in(f(w), f(z))$. Each point in $S_{t,-}^{p+1}\setminus\gamma(t)\subset f^{-1}(t)$ is the end point of an integral curve of $V$ beginning in $X_0$. Thus, applying in reverse the trajectory flow generated by $V$, to $S_{t,-}^{p+1}\setminus\gamma(t)$, specifies a diffeomorphism
\begin{equation*}
S_{t,-}^{p+1}\setminus\gamma(t)\longrightarrow D_-^{p+1}\subset X_0.
\end{equation*}
The boundary of this disk is of course the inward trajectory sphere $S_{-}^{p}$ which collapses to a point at $w$. 

Let $N_w$ and $N_z$ denote respective tubular neighbourhoods in $X_0$ of the sphere $S_{-}^{p}$ and the disk $D_{-}^{p+1}$ with respect to the background metric $\m$. We will assume that $N_w\subset N_z$. Note that $N_z$ is topologically a disk and the radii of these neighbourhoods can be chosen to be arbitrarily small. Each point $x\in X_0\setminus N_z$ is the starting point of a maximal integral curve $\psi_x:[0,1]\rightarrow W$ of $V$, which ends in $X_1$. As before, this gives rise to an embedding $\psi:(X_0\setminus N_z)\times I\rightarrow W$. 
We denote by $U$, the complement in $W$ of the image of this embedding. The region $U$ contains both critical points $w$ and $z$, the trajectory disks $K_{-}^{p+1}(w)$ and $K_{+}^{q}(z)$ as well as the trajectory arc $\gamma$; see Fig. \ref{neigh}. It is immediately clear that $U$ is diffeomorphic to $N_z\times I$, however, the gradient-like vector field $V$ has zeros in $U$ and so we cannot use its trajectory to construct an explicit diffeomorphism here in the way we can outside of $U$.   
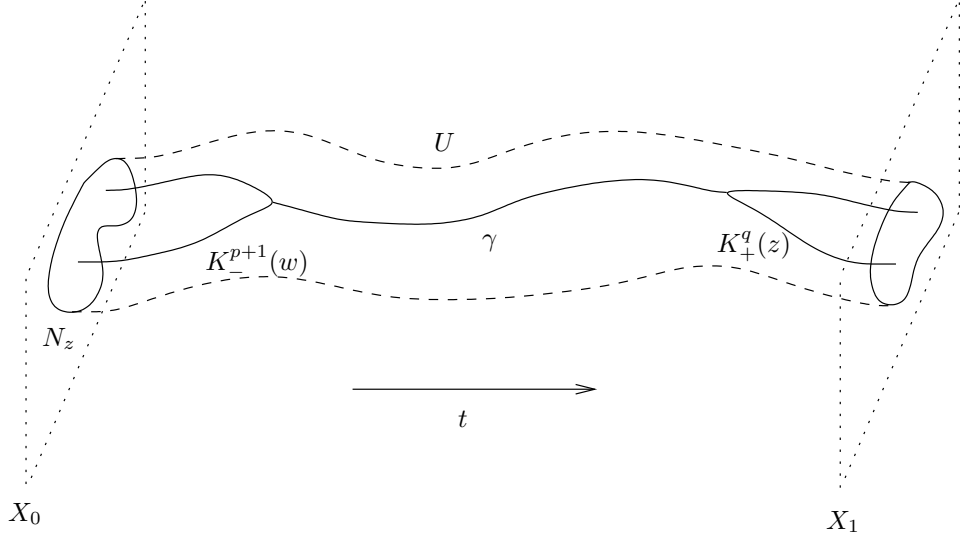
\begin{figure}[htbp]
\hspace{10mm}
\begin{picture}(0,0)%
\includegraphics{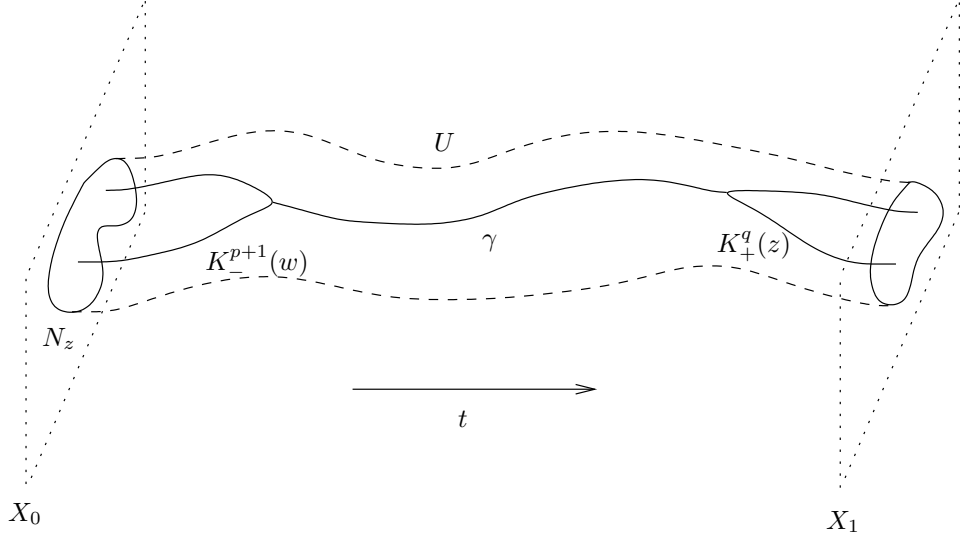}%
\end{picture}%
\setlength{\unitlength}{3947sp}%
\begingroup\makeatletter\ifx\SetFigFont\undefined%
\gdef\SetFigFont#1#2#3#4#5{%
  \reset@font\fontsize{#1}{#2pt}%
  \fontfamily{#3}\fontseries{#4}\fontshape{#5}%
  \selectfont}%
\fi\endgroup%
\begin{picture}(6002,3408)(649,-3277)
\put(876,-2061){\makebox(0,0)[lb]{\smash{{\SetFigFont{10}{12}{\rmdefault}{\mddefault}{\updefault}{\color[rgb]{0,0,0}$N_z$}%
}}}}
\put(5114,-1474){\makebox(0,0)[lb]{\smash{{\SetFigFont{10}{12}{\rmdefault}{\mddefault}{\updefault}{\color[rgb]{0,0,0}$K_{+}^{q}(z)$}%
}}}}
\put(3639,-1436){\makebox(0,0)[lb]{\smash{{\SetFigFont{10}{12}{\rmdefault}{\mddefault}{\updefault}{\color[rgb]{0,0,0}$\gamma$}%
}}}}
\put(1901,-1586){\makebox(0,0)[lb]{\smash{{\SetFigFont{10}{12}{\rmdefault}{\mddefault}{\updefault}{\color[rgb]{0,0,0}$K_{-}^{p+1}(w)$}%
}}}}
\put(3489,-2574){\makebox(0,0)[lb]{\smash{{\SetFigFont{10}{12}{\rmdefault}{\mddefault}{\updefault}{\color[rgb]{0,0,0}$t$}%
}}}}
\put(664,-3161){\makebox(0,0)[lb]{\smash{{\SetFigFont{10}{12}{\rmdefault}{\mddefault}{\updefault}{\color[rgb]{0,0,0}$X_0$}%
}}}}
\put(5801,-3211){\makebox(0,0)[lb]{\smash{{\SetFigFont{10}{12}{\rmdefault}{\mddefault}{\updefault}{\color[rgb]{0,0,0}$X_1$}%
}}}}
\put(3339,-836){\makebox(0,0)[lb]{\smash{{\SetFigFont{10}{12}{\rmdefault}{\mddefault}{\updefault}{\color[rgb]{0,0,0}$U$}%
}}}}
\end{picture}%
\caption{The neighbourhood $U$, diffeomorphic to the cylinder $N_z\times I$}
\label{neigh}
\end{figure}
It is always possible to {\em regularise} the admissible Morse triple $f$, replacing it with an admissible Morse triple $f'$ which agrees with $f$ on $W\setminus U$ and near $X_0$ and $X_1$, but which has no critical points. This is Theorem 5.4 of \cite{Smale}. The key point, which requires much work to show, is that there is a coordinate neighbourhood $U'\subset U$, containing the trajectory arc $\gamma$, on which $f|_{U'}$ takes the form
\begin{equation*}
\begin{split}
\mathbb{R}\times\mathbb{R}^{n}&\longrightarrow\mathbb{R}\\
(z, x)&\longmapsto z^{3}+3z-\sum_{i=1}^{s}{x_i}^{2}+\sum_{i=s+1}^{n-2}{x_i}^2.
\end{split}
\end{equation*}
This coordinate neighbourhood forms one slice in a neighbourhood $U'\times [-1, 1]\subset W\times [-\delta, 1]$ which is equipped with a map $F$ which takes the form
\begin{equation*}\label{coordf}
\begin{split}
\mathbb{R}\times\mathbb{R}^{n}&\longrightarrow\mathbb{R}\\
(z, x)&\longmapsto z^{3}+3yz-\sum_{i=1}^{s}{x_i}^{2}+\sum_{i=s+1}^{n-2}{x_i}^2,
\end{split}
\end{equation*}
on $U'\times [-1, 1]$, with $y\in[-1,1]$, and $f$ outside of $U\times [-1, 1]$. The map $f'$ is then the restriction $F|_{W\times \{1\}}$ while $f=F|_{W\times \{1\}}$. 
Furthermore, the family $F$ is a moderate family 
\begin{equation*}
F:W\times [-1,1]\longrightarrow [-1,1]\times I 
\end{equation*}
containing a single cusp singularity.

More generally, suppose $F$ is a moderate family defining a path through admissible generalised Morse functions which connects a pair of admissible Morse functions $f_0$ and $f_1$ and satisfies the following conditions. 
\begin{enumerate}
\item{}$f_0$ contains a pair of cancelling Morse critical points.
\item{}$f_1$ has no critical points.
\item{}$F$ contains exactly one cusp singularity.
\end{enumerate}
In other words, $F$ is a moderate family of admissible generalised Morse functions
\begin{equation*}
F:W\times I\longrightarrow I\times I
\end{equation*}
containing only one cusp singularity. Applying the construction from Theorem \ref{relconcisodoublesurgery}, allows us to prove Theorem \ref{genmorsepath} from the introduction, stated below. We will assume that $F$ comes equipped with a compatible reference metric and gradient-like vector field, although we will suppress their notation.
\vspace{3mm}

\noindent {\bf Theorem \ref{genmorsepath}.}
{\em Let $\{W; X_0, X_1\}$ be a smooth compact cobordism and let $F:W\times I\rightarrow I\times I$ be a moderate family of admissible generalised Morse functions. Suppose there is a point $y_0\in (0,1)$ so that $f_y=F|_{W\times \{y\}}$ is a Morse function for all $y\in I\setminus \{y_0\}$ and that $f_{y_0}$ contains exactly one birth-death critical point. Finally, let $g_0:I\rightarrow \Riem^{+}(X_0)$ be a family of psc-metrics on $X_0$   Then, there is a metric $\bar{\bar{g}}=\bar{\bar{g}}(F, g_0)$ on $W\times I$ which satisfies the following conditions.
\begin{enumerate}
\item{} For each $y\in [0,1]$,  the restriction of $\bar{\bar{g}}$ on slices $W\times\{y\}$ is a psc-metric which extends $g_0(y)$ and which has a product structure near the boundary $\p W\times\{y\}$.
\item{} For $y\in[0,1]$, away from $y_0$, the restriction of $\bar{\bar{g}}$ on slices $W\times\{y\}$, is a Gromov-Lawson cobordism.
\end{enumerate}
}
\vspace{3mm}

\begin{proof}
Let $w_0$ be the cusp singularity of $F$, with $F(w_0)=(y_0, c_0)$. Choose some small $\epsilon>0$. Then on $(p_2\circ F)^{-1}[0, c_0-\epsilon]$, we may use Theorem \ref{GLcobordismcompact} to construct a one-parameter family of metrics $\bar{g}_{y}=\bar{g}(F|_{y}, g_y)$,  each of which is a Gromov-Lawson cobordism on $(p_2\circ F)^{-1}[0, c_0-\epsilon]\cap (\{y\}\times W)$. This gives a metric $\bar{\bar{g}}=dy^{2}+\bar{g}_y$ on $(p_2\circ F)^{-1}[0,c_0-\epsilon]$, which satisfies the conditions above. Furthermore, away from $y_0$, i.e. for $y\in [0,y_0-\delta]\cup[y_0+\delta, 1]$ and some small $\delta>0$, we may extend this metric again using Theorem \ref{GLcobordismcompact}, to obtain the desired metric on $(p_2\circ F)^{-1}[0, c_0-\epsilon]\cup (W\times([0, y_0-\delta]\cup[y_0+\delta, 1]))$.

The difficult part is to extend this metric past the cusp singularity.  On the region  $(p_2\circ F)^{-1}[c_0-3\epsilon, c_0+3\epsilon]\cap W\times [y_0-3\delta, y_0+3\delta]$, the map $F$ is equivalent to a map
\begin{equation*}\label{coordcusp}
[-3\delta, 3\delta]\times [-3\epsilon, 3\epsilon]\times {D}^{n}\longrightarrow\mathbb{R},
\end{equation*}
which takes the form
\begin{equation*}\label{equivmap}
\begin{array}{c}
(y,z,x)\longmapsto
\begin{cases}
z^{3}+3yz-\sum_{i=1}^{s}{x_i}^{2}+\sum_{i=s+1}^{n-2}{x_i}^2, & \text{on $[-3\delta, \delta]\times [-\epsilon, \epsilon]\times {D}^{n}$}\\
z, & \text{outside $[-3\delta, 2\delta]\times [-2\epsilon, 2\epsilon]\times {D}^{n}$},
\end{cases}
\end{array}
\end{equation*} 
and which contains no critical points outside the region $[-3\delta, \delta]\times [-\epsilon, \epsilon]\times {D}^{n}$.

The metric $\bar{\bar{g}}$, as it is constructed so far, is defined near the boundary of this region and so pulls back to a metric defined near the boundary of the standard coordinate block. Now, applying the isotopy construction in Theorem \ref{relconcisodoublesurgery} to the restriction of this metric to $y=y_0-\delta$, we may extend the metric over the rest of this standard block. Pulling this metric back to $(p_2\circ F)^{-1}[c_0-2\epsilon, c_0+2\epsilon]\cap (W\times [y_0-2\delta, y_0+2\delta])$ results in a metric which agrees with $\bar{\bar{g}}$ near the boundary.  At this stage, the metric $\bar{\bar{g}}$ has been extended over $(p_2\circ F)^{-1}[0, c_0+2\epsilon]\cup (W\times([0, y_0-\delta]\cup[y_0+\delta, 1]))$. As there are no more singularities, this metric easily extends over the rest of $W\times I$, to obtain the desired metric.
\end{proof}

This theorem easily generalises to hold for any moderate path $F$ in $\GMor$ allowing us to prove Theorem \ref{isothm} from the introduction.
\vspace{3mm}

\noindent {\bf Theorem \ref{isothm}.}
{\em Let $\{W;X_0,X_1\}$ be a smooth compact cobordism, with $\pi_{1}(W)=\pi_{1}(X_0)=\pi_1(X_1)=0$ and $\dim W=n+1\geq 6$. Let $f_0$ and $f_1$ be a pair of admissible Morse functions on $W$ and let $g_0\in \Riem^{+}(X_0)$. Then the metrics $\bar{g}(g_0, f_0)$ and $\bar{g}(g_0, f_1)$ are isotopic, relative to the metric $g_0$, in $\Riem^{+}(W, \p W)$.  
}
\vspace{3mm}

\begin{proof}
With Lemma \ref{genmorsepath} in hand, it suffices to exhibit a moderate path of admissible generalised Morse functions connecting $f_0$ to $f_1$. From the work of Igusa, in particular Lemma 3.2 of \cite{I3}, any path through admissible generalised Morse functions can be easily adjusted to obtain a moderate one. Thus, it suffices to find a path. The existence of a path in $\GMor(W)$, connecting $f_0$ and $f_1$ follows from Theorem \ref{genmorseconnect}. Finally, the fact that any such path may be adjusted (fixing the endpoints) to lie entirely in $\GMor^{adm}(W)$ follows by application of the $2$-Index Theorem of Hatcher, Theorem \ref{2index} above.
\end{proof}

\subsection{Applying the Construction over General Families}
We now come to the main application of Theoerem \ref{relconcisodoublesurgery}. This is Theorem \ref{genmain}, which generalises Theorem \ref{main} to work for moderate families of admissible generalised Morse functions. 
\vspace{3mm}

\noindent {\bf Theorem \ref{genmain}.}
 {\em Let $\pi: E\to B$ be a bundle of smooth compact manifolds, where the fibre
  $W$ is a compact manifold with boundary $ \p W= X_0\sqcup X_1$ and the
  structure group is $\Diff(W;X_0,X_1)$.
  Let $F: E\to B\times I$ be a moderate
  family of admissible generalised Morse functions, with respect to $\pi$. In addition, we
  assume that the fibre bundle $\pi : E\to B$ is given the structure
  of a Riemannian submersion $\pi: (E,{\mathfrak m}_{E})\to
  (B,{\mathfrak m}_{B})$, such that the metric ${\mathfrak m}_{E}$ is
  compatible with the map $F: E \to B\times I$, and a gradient-like vector field $V_E$. Finally, let $g_0 : B \to \Riem^+(X_0)$ be a smooth map.

  Then there exists a metric $\bar{g}=\bar{g}(g_0,F)$ (where $F=(F, \m_E, V_E)$)
  such that for each $y\in B$ the restriction $\bar{g}(y)=\bar{g}|_{W_y}$ on the
  fibre $W_y=\pi^{-1}(y)$ satisfies the following properties:
\begin{enumerate}
\item[{\bf (1)}] $\bar{g}(y)$ extends $g_0(y)$;
\item[{\bf (2)}] $\bar{g}(y)$ is a product metric $g_{i}(y)+dt^2$
  near $X_{i}\subset \p W_y$, $i=0,1$;
\item[{\bf (3)}] $\bar{g}(y)$ has positive scalar curvature
  on $W_y$.
\end{enumerate}
} 
\vspace{3mm}

 \begin{proof}
The proof essentially mimics that of Theorem \ref{main}. Indeed, away from cusp singularities, the construction is identical. Near cusps, we perform an analogous procedure to that of Theorem \ref{main}, using the exponential map to pull the construction back to a disk bundle over the cusp and making use of the equivariance proved in Lemma \ref{genGLequiv} to perform the construction globally. 

Without loss of generality, we will assume that $\Sigma^{1}$ consists of only one path component with critical points of index $p+\frac{3}{2}$. 
By Lemma 3.4 of \cite{I3}, there is, for some $\delta>0$, a codimension $0$ immersion
\begin{equation*}
i:\Sigma^{1}\times(-3\delta, 3\delta)\longrightarrow B
\end{equation*}
such that the orientation of the normal bundle induced by $i$ is given by the parameter direction in which the pair of Morse singular points becomes a birth-death singular point (i.e. the ``birth" direction). 

Consider the space $E\setminus \pi^{-1}(i(\Sigma^{1}\times(-\delta, \delta)))$. On this space, we may proceed exactly as we did in the proof of Theorem \ref{main}. Furthermore, below $\Sigma^{1}$ the construction also proceeds as normal. It remains to show how we can extend this construction past $\Sigma^1$. Let $c\in I$ denote the image $p_2\circ F(\Sigma^1)$. Assume that the metric $\bar{g}$ is constructed on $(p_2\circ F)^{-1}[0, c-\epsilon]$ for some small $\epsilon>0$. Now consider the space $(p_2\circ F)^{-1}[0, c-\epsilon]\cup (E\setminus \pi^{-1}(i(\Sigma^{1}\times(-\delta, \delta))))$. Here, the metric $\bar{g}$ is constructed and so our task is to extend it over the region $(p_2\circ F)^{-1}[0, c+\epsilon]\setminus (p_2\circ F)^{-1}[0, c-\epsilon]\cup (E\setminus \pi^{-1}(i(\Sigma^{1}\times(-\delta, \delta))))$.

Let $D$ denote the neighbourhood $(p_2\circ F)^{-1}[0, c+3\epsilon]\setminus (p_2\circ F)^{-1}[0, c-\epsilon]\cup (E\setminus \pi^{-1}(i(\Sigma^{1}\times(-3\delta, 3\delta))))$. Using the exponential map, we identify $D$ with a disk bundle of $\Ver|_{\Sigma^1}\times (-3\delta, 3\delta)$. In particular, we obtain the splitting
\begin{equation*}
\begin{split}
D&\cong (-3\delta, 3\delta)\times D(\Ver^{0})\times D(\Ver^{-})\times D(\Ver^{+})\\ 
&\cong (-3\delta, 3\delta)\times (-3\epsilon, 3\epsilon)\times D(\Ver^{-})\times D(\Ver^{+}).
\end{split}
\end{equation*}


Identifying $D$ with its image under the above composition of isomorphisms, we see that on each fibre of $D$, the situation is of the type described in Lemma \ref{genmorsepath}. Let $\rho=\rho(x)$ and $r=r(y)$ denote radial distance coordinates, where $x\in D(\Ver^{-})$ and $y\in D(\Ver^{+})$. Note that, as we are completing the construction on the bundle $D$, we are actually interested in the function $F\circ\exp$, but for convenience we will refer to this function simply as $F$ for the remainder of the proof. From Theorem 2.6 of \cite{I3}, we know that on $D$, $F$ takes the form \begin{equation*}
F(t,z,x,y )=z^{3}+tz+\alpha(x, y),
\end{equation*}
for some coordinates $(t,z)\in (-3\delta, 3\delta)\times (-3\epsilon, 3\epsilon)$ and some smooth function $\alpha$ which is independent of $t$ and $z$ and agrees infinitesimally with $-\rho^{2}+r^{2}$ on the $0$-section. It is demonstrated in the the proof of Theorem \ref{main} (\cite{BHSW}), that $\alpha$ may be adjusted to agree with $-\rho^{2}+r^{2}$ near the zero section without introducing any new singularities.

Fibrewise, we now have almost exactly the situation described in Lemma \ref{genmorsepath}. The only difference is that we don't quite have global Morse coordinates and are making do instead with radial distance coordinates determined by the splitting. This is not a problem as the $SO(1)\times O(p+1)\times O(q)$ equivariance demonstrated in Lemma \ref{genGLequiv}, means that these radial distance coordinates are sufficient. The construction then proceeds in the manner of Theorem \ref{main}.

\end{proof}

\vspace{0.5cm}

\noindent {\em E-mail address:} walsmark@math.oregonstate.edu

\end{document}